\documentclass[11pt,reqno,a4paper]{amsart}

\usepackage{ot}
\usepackage{ulem}
\usepackage[noend]{algpseudocode}
\usepackage{algorithmicx}

\author{Quentin Mérigot \and Boris Thibert}
\title{Optimal transport: discretization and algorithms}

\begin{document}
\normalem

\begin{abstract}
This chapter describes techniques for the numerical resolution of
optimal transport problems. We will consider several discretizations of
these problems, and we will put a strong focus on the mathematical
analysis of the algorithms to solve the discretized problems. We will
describe in detail the following discretizations and corresponding
algorithms: the assignment problem and Bertsekas auction's algorithm;
the entropic regularization and Sinkhorn-Knopp's algorithm;
semi-discrete optimal transport and Oliker-Prussner or damped Newton's
algorithm, and finally semi-discrete entropic regularization. Our
presentation highlights the similarity between these algorithms and
their connection with the theory of Kantorovich duality.
\end{abstract}

\maketitle
\tableofcontents
\clearpage

\section{Introduction}

The problem of optimal transport, introduced by Gaspard Monge in 1871
\cite{monge1781memoire}, was motivated by military applications. The
goal was to find the most economical way to transport a certain amount
of sand from a quarry to a construction site. The source and target
distributions of sand are seen as probability measures, denoted $\mu$
and $\nu$, and $c(x,y)$ denotes the cost of transporting a grain of
sand from the position $x$ to the position $y$, and the goal is to
solve the non-convex optimization problem
\begin{equation}
  \MP = \min_{T_{\#}\mu = \nu} \int c(x,T(x)) \dd \mu,
\end{equation}
where $T_{\#}\mu = \nu$ means that $\nu$ is the push-forward of $\mu$
under the transport map $T$.  The modern theory of optimal transport
has been initiated by Lenoid Kantorovich in the 1940s, via a convex
relaxation of Monge's problem. Given two probability measures $\mu$
and $\nu$, it consists in minimizing 
\begin{equation} \label{eq:kp0}
  \KP = \min_{\gamma \in \Gamma(\mu,\nu)} \int c(x,y)\dd\gamma(x,y),
\end{equation}
over the set $\Gamma(\mu,\nu)$ of \emph{transport plans}\footnote{A
  probability measure $\gamma$ is a transport plan between $\mu$ and
  $\nu$ if its marginals are $\mu$ and $\nu$.} between $\mu$ and
$\nu$. Kantorovich's theory has been used and revisited by many
authors from the 1980s, allowing a complete solution to Monge's
problem in particular for $c(x,y) = \nr{x - y}^p$.  Since then,
optimal transport has been connected to various domains of mathematics
(geometry, probabilities, partial differential equations) but also to
more applied domains.  Current applications of optimal transport
include machine learning \cite{peyre2019computational}, computer
graphics \cite{rubner2000earth}, quantum chemistry
\cite{buttazzo2012optimal,cotar2013density}, fluid dynamics
\cite{brenier1999minimal,de2015power,merigot2016minimal}, optics
\cite{oliker2003mathematical,caffarelli1999problem,wang2004design,caffarelli2008weak},
economy \cite{galichon2018optimal}, statistics
\cite{carlier2016vector,chernozhukov2017monge,hutter2019minimax}. The
selection of citation above is certainly quite arbitrary, as optimal
transport is now more than ever a vivid topic, with more than several
hundreds (perhaps even thousands) of articles published every year and
containing the words <<optimal transport>>.

There exist many books on the theory of optimal transport, e.g. by
Rachev-R\"uschendorf \cite{rachev1998mass,rachev2006mass}, by Villani
\cite{villani2003topics,villani2008optimal} and by Santambrogio
\cite{santambrogio2015optimal}. However, there exist fewer books
dealing with the numerical aspects, by Galichon
\cite{galichon2018optimal}, by Cuturi-Peyré
\cite{peyre2019computational} and one chapter of Santambrogio
\cite{santambrogio2015optimal}. The books by Galichon and Cuturi-Peyré
are targeted toward applications (in economy and machine learning,
respectively) and do not deal in full detail with the mathematical
analysis of algorithms for optimal transport.  In this chapter, we
concentrate on numerical methods for optimal transport relying on
Kantorovich duality. Our aim in particular is to provide a
self-contained mathematical analysis of several popular algorithms to
solve the discretized optimal transport problems.

\vspace{.5cm}\noindent\textbf{Kantorovich duality.}  In
the 2000s, the theory of optimal transport was already mature and was
used within mathematics, but also in theoretical physics or in
economy.  However, numerical applications were essentially limited to
one-dimensional problems because of the prohibitive cost of existing
algorithms for higher dimensional problems, whose complexity was in
general more than quadratic in the size of the data.  Numerous
numerical methods have been introduced since then. Most of them rely
the dual problem associated to Kantorovich's problem \eqref{eq:kp0}, namely
\begin{equation} \label{eq:dp0}
  \DP = \max_{\phi \ominus \psi \leq c} \int \phi\dd \mu -\int \psi\dd\nu,
\end{equation}
where the maximum is taken over pairs $(\phi,\psi)$ of functions
satisfying $\phi\ominus \psi\leq c$, meaning that $\phi(x)
-\psi(y)\leq c(x,y)$ for all $x,y$.  Equivalently, the dual problem
can be written as the unconstrained maximization problem
\begin{equation}\label{eq:dpbis0}
  \DP = \max_{\psi} \Kant(\psi)\
  \hbox{ where }\ \Kant(\psi) = \int \psi^c \dd \mu -\int \psi\dd\nu
\end{equation}
and $\psi^c(x) := \min_{y} c(x,y)+ \psi(y)$ is the $c$-transform of
$\psi$, a notion closely related to the Legendre-Fenchel transform in
convex analysis. The function $\Kant$ is called the \emph{Kantorovitch
  functional}. Kantorovich's duality theorem asserts that the values
of \eqref{eq:kp0} and \eqref{eq:dp0} (or \eqref{eq:dpbis0}) agree
under mild assumptions.

\vspace{.5cm}\noindent\textbf{Overview of numerical methods.}  We now
briefly review the most used numerical methods for optimal
transport. Note that there is no <<free lunch>> in the sense that
there exists no method able to deal efficiently with arbitrary cost
function; the computational complexity of most methods depend on the
complexity of computing $c$-transforms or smoothed $c$-transforms.  In
this overview, we skip linear programming methods such as the network
simplex, for which we refer to \cite{peyre2019computational}.

\vspace{.5cm}\noindent\textbf{A. Assignment problem.} When the two
measures are uniformly supported on two finite sets with cardinal $N$,
the optimal transport problem coincides with the \emph{assignment
  problem}, described in \cite{burkard2009assignment}. The assignment
problem can be solved using various techniques, but in this chapter we
will concentrate on a dual ascent method called Bertsekas' auction
algorithm \cite{bertsekas1981new}, whose complexity is also
$\mathrm{O}(N^3)$, but which is very simple to implement and
analyze. In \S\ref{subsec:auction-implementation} we note that
\emph{the complexity can be improved when it is possible to compute
  discrete $c$-transforms efficiently}, namely
  \begin{equation}
    \psi^c(x_i) = \min_{j} c(x_i,y_j) + \psi(y_j).
  \end{equation}

\vspace{.5cm}\noindent\textbf{B. Entropic regularization.}
  In this approach, one does not solve the original optimal transport
  problem \eqref{eq:kp0} exactly, but instead replaces it with a
  regularized problem involving the entropy of the transport plan. In
  the discrete case, it consists in minimizing
  \begin{equation} \label{eq:ent0}
    \min_{\gamma} \sum_{i,j} \gamma_{i,j} c(x_i,y_j) + \eta \sum_{i,j} h(\gamma_{i,j}),
  \end{equation}
  where $h(r) = r(\log r - 1)$  and $\eta>0$ is a small parameter, under the constraints
  \begin{equation} \forall i,~ \sum_{j} \gamma_{i,j} = \mu_i, \qquad \forall j,~
  \sum_i \gamma_{i,j} = \nu_j
  \end{equation}
  This idea has been introduced in the field of optimal transport by
  Galichon and Salanié \cite{galichon2010matching} and by Cuturi
  \cite{cuturi2013sinkhorn}, see \cite[Remark
    4.5]{peyre2019computational} for a brief historical
  account. Adding the entropy of the transport plan makes the problem
  \eqref{eq:ent0} strongly convex and smooth. The dual problem can be
  solved efficiently using Sinkhorn-Knopp's algorithm, which involves
  computing repeatedly the \emph{smoothed $c$-transform}
  \begin{equation}\label{eq:psi-c-eta} \psi^{c,\eta}(x_i) = \eta\log(\mu_i) - \eta\log\left(\sum_{j}
    e^{\frac{1}{\eta}(-c(x_i,y_j) - \psi(y_j))}\right).
  \end{equation}
  Sinkhorn-Knopp's algorithm can be very efficient, \emph{provided
    that the smoo\-thed $c$-transform can be computed efficiently
    (e.g. in near-linear time).}
    
\vspace{.5cm}\noindent\textbf{C. Distance costs.}
When the cost $c$ satisfies the triangle inequality, the dual problem
\eqref{eq:dp0} can be further simplified:
  \begin{equation}\label{eq:dpbislip}
  \max_{\Lip_c(\psi)\leq 1} \int \psi \dd \mu -\int \psi\dd\nu,
  \end{equation}
  where the maximum is taken over functions satisfying $\abs{\psi(x) -
    \psi(y)}\leq c(x,y)$ for all $x,y$. The equality between the
  values of \eqref{eq:kp0} and \eqref{eq:dpbislip} is called
  Kantorovich-Rubinstein's theorem. This leads to very efficient
  algorithms when the $1$-Lipschitz constraint can be enforced using only
  \emph{local} information, thus reducing the number of constraints. This is
  possible when the space is discrete and the distance is induced by a
  graph, or when $c$ is the Euclidean norm or more generally a
  Riemannian metric. In the latter case, the maximum in
  \eqref{eq:dpbislip} can be replaced by a supremum over $\Class^1$
  functions $\psi$ satisfying $\nr{\nabla \psi}_\infty \leq 1$
  \cite{solomon2014earth,benamou2015augmented}. Note that \emph{the
    case of distance costs is particularly easy because the
    $c$-transform of a $1$-Lipschitz function is trivial: $\psi^c =
    -\psi$.} 

\vspace{.5cm}\noindent\textbf{D. Monge-Ampère equation.}
  When the cost is the Euclidean scalar product, $c(x,y) =
  -\sca{x}{y}$, the dual problem \eqref{eq:dp0} can be reformulated as
  \begin{equation} \label{eq:dp0ma}
    \max_{\psi} - \int \psi^* \dd\mu -\int \psi\dd \nu,
    \end{equation} where
  $\psi^*(x) = \max_y \sca{x}{y} - \psi(y)$ is the Legendre-Fenchel
  transform of $\psi$. 
If the maximizer $\psi$ is smooth and strongly convex and $\mu$, $\nu$
are probability densities, the optimality condition associated to the dual
problem is the \emph{Monge-Ampère equation,}
  \begin{equation} \label{eq:ma0}
    \begin{cases}
    \mu(\nabla \psi(y)) \det(\D^2\psi(y)) = \nu(y),\\
    \nabla \psi(\spt(\nu)) \subseteq \spt(\mu).
    \end{cases}
  \end{equation}
  Note the non-standard boundary conditions appearing on the second
  line of the equation. The first methods able to deal with these
  boundary conditions use a ``wide-stencil'' finite difference
  discretization
  \cite{froese2012numerical,benamou2014numerical,benamou2019minimal}.
  These methods are able to solve optimal transport problems provided
  that the maximizer of \eqref{eq:dp0ma} is a viscosity solution to
  the Monge-Ampère equation~\eqref{eq:ma0}, imposing restrictions on its
  regularity.  For the Monge-Ampère equation with Dirichlet
  conditions, we refer to the recent survey by Neilan, Salgado and
  Zhang \cite{neilan2019monge}.

\vspace{.5cm}\noindent\textbf{E. Semi-discrete formulation.} The semi-discrete formulation of
  optimal transport involves a source measure that is a probability
  density $\mu$ and a target measure $\nu$ which is finitely
  supported, i.e. $\nu = \sum_i \nu_i \delta_{y_i}$. It was introduced
  by Cullen in 1984 \cite{cullen1984extended}, without reference to
  optimal transport, and much refined since then
  \cite{aurenhammer1998minkowski,merigot2011multiscale,de2012blue,gu2013variational,kitagawa2014iterative,levy2015numerical,kitagawa2016newton}. In
  this setting, the dual problem \eqref{eq:dpbis0} amounts to maximizing the Kantorovitch functional given by
  \begin{equation}\label{eq:sd0}
  \Kant(\psi) =  \int \psi^c \dd\mu - \int \psi \dd\nu = 
    \sum_i \int_{\Lag_{y_i}(\psi)} c(x,y_i) + \psi(y_i) \dd\mu(x) - \int \psi \dd\nu ,
  \end{equation}
  where
  the Laguerre cells are defined by
  \begin{equation}\label{eq:lag0}
    \Lag_{y_i}(\psi) = \{ x \mid \forall j, c(x,y_i) + \psi(y_i) \leq
    c(x,y_j) + \psi(y_j) \}.
  \end{equation}
  The optimality condition for \eqref{eq:sd0} is the following
  non-linear system of equations,
  \begin{equation}\label{eq:masd0}
    \forall i, \mu(\Lag_{y_i}(\psi)) = \nu_i.
  \end{equation}
  In the case $c(x,y)= -\sca{x}{y}$, this system of equations can see
  as a weak formulation (in the sense of Alexandrov, see \cite[Chapter
    1]{gutierrez2001monge}) of the Monge-Ampère equation
  \eqref{eq:ma0}.  This ``semi-discrete'' approach can also be used to
  solve Monge-Ampère equations with Dirichlet boundary conditions, and
  has been originally introduced for this purpose
  \cite{oliker1989numerical,mirebeau2015discretization}. \emph{Again,
    the possibility to solve \eqref{eq:masd0} efficiently requires one
    to be able to compute the Laguerre tessellation \eqref{eq:lag0},
    and thus the $c$-transform, efficiently.}

\vspace{.5cm}\noindent\textbf{F. Dynamic formulation.}
  This formulation relies on the dynamic formulation of optimal
  transport, which holds when the cost is $c(x,y) = \nr{x-y}^2$ on $\Rsp^d$ (or
  more generally induced by a Riemannian metric), and is known as the
  \emph{Benamou-Brenier formulation}:
  $$ \min_{(\rho,v)} \int_{0}^1 \int_{\Rsp^d} \rho_t \nr{v_t}^2 \dd x \dd t \hbox{\quad with\quad}
  \begin{cases}
    \partial_t \rho_t + \div(\rho_t v_t) = 0, \\
    \rho_0  = \mu, \rho_1 = \nu.
  \end{cases}$$
Introducing the momentum $m_t = \rho_t v_t$, the problem can be
rewritten as
$$ \min_{(\rho_t,m_t)} \int_{0}^1 \int_{\Rsp^d} \frac{\nr{m_t}^2}{\rho_t} \dd x \dd t \hbox{\quad with\quad}
  \begin{cases}
    \partial_t \rho_t + \div(m_t) = 0, \\
    \rho_0  = \mu, \rho_1 = \nu.
  \end{cases}$$
  This optimization problem can be discretized using finite elements
  \cite{benamou2002monge}, finite differences
  \cite{papadakis2014optimal} or finite volumes
  \cite{erbar2017computation}, and the discrete problem is then
  usually solved using a primal-dual augmented Lagrangian method
  \cite{benamou2002monge} (see also
  \cite{papadakis2014optimal,hug2016analyse}).  In practice, the
  convergence is very costly in terms of number of iterations; note
  also that each iteration requires the resolution of a
  $(d+1)$-dimensional Poisson problem to project on the admissible set
  $\{ (\rho,m) \mid \partial_t \rho + \div(m) = 0 \}$. Another
  possibility is to use divergence-free wavelets
  \cite{henry2015optimal}.  One advantage of the Benamou-Brenier
  approach is that it is very flexible, easily allowing (Riemannian)
  cost functions \cite{papadakis2014optimal}, additional quadratic
  terms \cite{benamou2001mixed}, penalization of congestion
  \cite{buttazzo2009optimization}, partial transport
  \cite{lombardi2015eulerian,chizat2018interpolating}, etc. Finally,
  the convergence from the discretized problem to the continuous one
  is subtle and depends on the choice of the discretization, see
  \cite{lavenant2019unconditional,carrillo2019primal}.
\medskip  
  
%
%



In this chapter, we will describe in detail the following
discretizations for optimal transport and corresponding algorithms to
solve the discretized problems : the assignment problem (A.) through
Bertsekas auction's algorithm, the entropic regularization (B.)
through Sinkhorn-Knopp's algorithm, semi-discrete optimal transport
(E.) through Oliker-Prussner or Newton's methods. These algorithms
share a common feature, in that they are all derived from Kantorovich
duality.  Some of them have been adapted to variants of optimal
transport problems, such as multi-marginal optimal transport problems
problems \cite{pass2015multi}, barycenters with respect to optimal
transport metrics \cite{agueh2011barycenters}, partial
\cite{caffarelli2010free} and unbalanced optimal
transport~\cite{chizat2018interpolating,kondratyev2016new}, gradient
flows in the Wasserstein space
\cite{jordan1998variational,ambrosio2008gradient}, generated Jacobian
equations \cite{guillen2019primer,trudinger2014local}.  However, we
consider these extensions to be out of the scope of this chapter.


\section{Optimal transport theory}
This part contains a self-contained introduction to the theory of
optimal transport, putting a strong emphasis on Kantorovich
Kantorovich duality. Kantorovich duality is at the heart of the most
important theorems of optimal transport, such as Brenier and
Gangbo-McCann's theorems on the existence and uniqueness of solution
to Monge's problems and the stability of optimal transport plans and
optimal transport maps maps. Kantorovich's duality is also used in all
the numerical methods presented in this chapter.


\subsubsection*{Background on measure theory.}
In the following, we assume that $X$ is a compact metric space, and we
denote $\Class^0(X)$ the space of continuous functions over $X$. We
denote $\Meas(X)$ the space of \emph{finite (Radon) measures} over
$X$, identified with the set of continuous linear forms over
$\Class^0(X)$. Given $\phi\in\Class^0(X)$ and $\mu\in\Meas(X)$, we
will often denote $\sca{\phi}{\mu} = \int_X\phi \dd\mu$.  The spaces
of \emph{non-negative measures} and \emph{probability measures} are
defined by
$$\Meas^+(X) := \{ \mu \in \Meas(X) \mid \mu\geq 0 \},$$
$$\Prob(X) := \{ \mu \in \Meas^+(X) \mid \mu(X) = 1 \}, $$ where
$\mu\geq 0$ means $\sca{\mu}{\phi}\geq 0$ for all
$\phi\in\Class^0(X,\Rsp^+)$.  The three spaces $\Meas(X)$, $\Meas^+(X)$ and
$\Prob(X)$ are endowed with the weak topology induced by duality with
$\Class^0(X)$, namely $\mu_n \to \mu$ weakly if
$$ \forall \phi\in\Class^0(X),~ \sca{\phi}{\mu_n} 
\xrightarrow{n\to\infty} \sca{\phi}{\mu}.$$ A point $x$ belongs to the \emph{support} of a
non-negative measure $\mu$ iff for every $r>0$ one has
$\mu(\B(x,r))>0$. The support of $\mu$ is denoted $\spt(\mu)$.  We
recall that by Banach-Alaoglu theorem, the set of probability measures
$\Prob(X)$ is weakly compact, a fact which will be useful to prove
existence and convergence results in optimal transport.


\paragraph{Notation.} Given two functions
$\phi\in \Class^0(X)$ and $\psi \in \Class^0(Y)$ we will define
$\phi\oplus\psi \in \Class^0(X\times Y)$ by $\phi \oplus \psi(x,y) =
\phi(x) + \psi(y)$. We define $\phi\ominus\psi$ and $\phi\otimes\psi$
similarly. 


\subsection{The problems of Monge and Kantorovich}

\subsubsection{Monge's problem}
Before introducing Monge's problem, we recall the definition
of \emph{push-forward} or \emph{image} measure.
\begin{definition}[Push-forward and transport map]
  Let $X,Y$ be compact metric spaces, $\mu \in \Meas(X)$ and $T:X\to
  Y$ be a measurable map. The \emph{push-forward} of $\mu$ by $T$ is
  the measure $T_\#\mu$ on $Y$ defined by
  $$ \forall \phi\in\Class^0(Y), \sca{\phi}{T_\#\mu} := \sca{\phi\circ
    T}{\mu},$$ or equivalently if for every Borel subset $B\subseteq
  Y,~~ \T_\#\mu(B) = \mu(T^{-1}(B))$.  A measurable map $T:X\to Y$
  such that $T_\#\mu = \nu$ is also called a \emph{transport map}
  between $\mu$ and $\nu$.
\end{definition}

\begin{example}
  If $Y = \{y_1,\hdots,y_n\}$, then $T_\#\mu = \sum_{1\leq i\leq n}
  \mu(T^{-1}(\{y_i\})) \delta_{y_i}$.
\end{example}

\begin{example}
  Assume that $T$ is a $\Class^1$ diffeomorphism between compact
  domains $X,Y$ of $\Rsp^d$, and assume also that the probability
  measures $\mu,\nu$ have continuous densities $\rho,\sigma$ with
  respect to the Lebesgue measure. Then,
  $$ \int_Y \phi(y) \sigma(y) \dd y = \int_{X} \phi(T(x)) \sigma(T(x)) \det(\D T(x)) \dd x.$$
Hence, $T$ is a transport map between $\mu$ and $\nu$ iff 
$$ \forall \phi\in \Class^0(X), \int_{X} \phi(T(x)) \sigma(T(x))
\det(\D T(x)) \dd x = \int_X \phi(T(x)) \rho(x) \dd x, $$ or equivalently if  the (non-linear) Jacobian equation holds
$$\rho(x) =
\sigma(T(x))\det(\D T(x)).$$
\end{example}

\begin{definition}[Monge's problem]
  Consider two compact metric spaces $X,Y$, two probability measures
  $\mu \in \Prob(X)$, $\nu \in \Prob(Y)$ and a \emph{cost function} $c
  \in \Class^0(X\times Y)$. \emph{Monge's problem} is the following
  optimization problem
  \begin{equation}
    \label{eq:monge}
    \MP := \inf \left\{ \int_{X} c(x,T(x)) \dd \mu(x) \mid T:X\to Y \hbox{ and } T_\# \mu = \nu \right\}
  \end{equation}
\end{definition}

Monge's problem exhibits several difficulties, one of which is that both
the transport constraint ($T_\#\mu = \nu$) and the functional are
non-convex. Note also that there might exist no transport map between
$\mu$ and $\nu$. For instance, if $\mu = \delta_x$ for some $x\in X$,
then, $\T_\#\mu(B) = \mu(T^{-1}(B)) = \delta_{T(x)}$. In particular,
if $\card(\spt(\nu))>1$, there exists no transport map between $\mu$
and $\nu$.

\subsubsection{Kantorovich's problem}

\begin{definition}[Marginals]
    The \emph{marginals} of a measure $\gamma$ on a product space
    $X\times Y$ are the measures $\Pi_{X\#} \gamma$ and
    $\Pi_{Y\#}\gamma$, where $\Pi_X: X\times Y\to X$ and $\Pi_Y:
    X\times Y\to Y$ are their projection maps.
\end{definition}

\begin{definition}[Transport plan]
  A transport plan between two probability measures $\mu,\nu$ on two
  metric spaces $X$ and $Y$ is a probability measure $\gamma$ on the
  product space $X\times Y$ whose marginals are $\mu$ and $\nu$. The
  space of transport plans is denoted $\Gamma(\mu,\nu)$, i.e.
  $$ \Gamma(\mu,\nu) = \left\{ \gamma \in \Prob(X\times Y) \mid
  \Pi_{X\#} \gamma = \mu,~~\Pi_{Y\#}\gamma= \nu \right\}. $$
  Note that $\Gamma(\mu,\nu)$ is a convex set.
\end{definition}

\begin{example}[Product measure]
Note that the set of transport plans $\Gamma(\mu,\nu)$ is never empty,
as it contains the measure $\mu\otimes \nu$.
\end{example}

\begin{definition}[Kantorovich's problem] 
  Consider two compact metric spaces $X,Y$, two probability measures
  $\mu \in \Prob(X)$, $\nu \in \Prob(Y)$ and a \emph{cost function}
  $c\in\Class^0(X\times Y)$. \emph{Kantorovich's problem} is the
  following optimization problem
  \begin{equation}
    \label{eq:kantorovich}
    \KP := \inf \left\{ \int_{X\times Y} c(x,y) \dd \gamma(x,y) \mid \gamma \in \Gamma(\mu,\nu) \right\}
  \end{equation}
\end{definition}

\begin{remark}
  The infimum in Kantorovich's problem is less than the infimum in
  Monge's problem. Indeed, to any transport map $T$ between $\mu$ and
  $\nu$ one can associate a transport plan, by letting $\gamma_T =
  (\id, T)_{\#} \mu$. One can easily check that $\Pi_{X\#}\gamma_T =
  \mu$ and $\Pi_{Y\#}\gamma_T = \nu$ so that $\gamma_T \in
  \Gamma(\mu,\nu)$ is a transport plan between $\mu$ and
  $\nu$. Moreover, by the definition of push-forward,
$$ \sca{c}{\gamma_T} = \sca{c}{(\id,T)_\#\mu} =
\sca{c\circ(\id,T)}{\mu} = \int_{X} c(x,T(x))\dd\mu$$ thus showing
that $\KP\leq \MP$.
\end{remark}
  
\begin{proposition} \label{prop:existence-kp}
  Kantorovich's problem $\KP$ admits a minimizer.
\end{proposition}

\begin{proof} The definition of
  $\Pi_{X\#} \gamma = \mu$ can be expanded into
  $$ \forall \phi\in \Class^0(X), \sca{\phi\otimes 1}{\gamma}
  = \sca{\phi}{\mu},$$ from which it is easy to see that the set
  $\Gamma(\mu,\nu)$ is weakly closed, and therefore weakly compact as
  a subset of $\Prob(X\times Y)$, which is weakly compact by
  Banach-Alaoglu's theorem. We conclude the existence proof by
  remarking that the functional that is minimized in $\KP$, namely
  $\mu \mapsto \sca{c}{\mu}$, is weakly continuous by definition.
\end{proof}

\subsection{Kantorovich duality}\label{sec:kantorovitchduality}
\subsubsection{Derivation of the dual problem}
The primal Kantorovich problem $\KP$ can be reformulated by
introducing Lagrange multipliers for the constraints. Namely, we use
 that for any $\gamma\in \Meas^+(X\times Y)$,
$$ \sup_{\phi\in\Class^0(X)} -\sca{\phi\otimes 1}{\gamma} + \sca{\phi}{\mu} = \begin{cases} 0 & \hbox{ if } \Pi_{X\#}\gamma = \mu \\
+\infty & \hbox{ if not}\end{cases}$$
$$ \sup_{\phi\in\Class^0(X)} \sca{1\otimes \psi}{\gamma}
- \sca{\psi}{\mu} = \begin{cases} 0 & \hbox{ if } \Pi_{X\#}\gamma
= \mu \\ +\infty & \hbox{ if not}\end{cases}$$ to deduce
  $$ \sup_{\phi\in\Class^0(X),\psi\in\Class^0(Y)} \sca{\phi}{\mu} -
\sca{\psi}{\nu} - \sca{\phi \ominus\psi}{\gamma} = \begin{cases} 0 &
  \hbox{ if } \gamma\in\Gamma(\mu,\nu)\\ +\infty & \hbox{ if not}.\end{cases}$$
This leads to the following
formulation of the Kantorovich problem
\begin{equation*}  \KP = \inf_{\gamma \in \Meas^+(X\times Y)} \sup_{(\phi,\psi)\in\Class^0(X)\times \Class^0(Y)} \sca{c - (\phi \ominus \psi)}{\gamma} + \sca{\phi}{\mu} - \sca{\psi}{\nu}
\end{equation*}
Kantorovich dual problem is simply obtained by inverting the infimum and the
supremum:\begin{align*} \DP := \sup_{\phi,\psi} \inf_{\gamma\geq 0}
  \sca{c - (\phi \ominus\psi)}{\gamma} + \sca{\phi}{\mu} -
  \sca{\psi}{\nu}.
\end{align*}
Note that we will often omit the assumptions that
$\gamma\in\Meas(X\times Y)$ and $\phi,\psi$ are continuous, when the
context is clear.  The dual problem can further be simplified by
remarking that
$$ \inf_{\gamma\geq 0}  \sca{c - \phi \ominus\psi}{\gamma} =
 \begin{cases}
0 &\hbox{ if }  \phi \ominus \psi \leq c  \\
-\infty & \hbox{ if not. }
 \end{cases}$$
 \begin{definition}[Kantorovich's dual problem]
Given $\mu\in \Prob(X)$ and $\nu\in\Prob(Y)$ with $X,Y$ compact metric
spaces and $c\in\Class^0(X\times Y)$, we define Kantorovich's dual
problem by
\begin{equation}
\DP = \sup  \left\{ \int_{X} \phi \dd \mu - \int_Y \psi\dd\nu \mid (\phi,\psi) \in \Class^0(X)\times\Class^0(Y), \phi \ominus \psi \leq c \right\} 
\end{equation}
\end{definition}

\begin{proposition} Weak duality  holds, i.e. $\KP\geq \DP$.
\end{proposition}

\begin{proof} Given $(\phi,\psi,\gamma) \in \Class^0(X)\times \Class^0(Y)\times
\Gamma(\mu,\nu)$ satisfying the constraint $\phi\ominus\psi \leq c$,
one has
$$ \sca{\phi}{\mu} - \sca{\psi}{\nu} = \sca{\phi\ominus\psi}{\gamma}
\leq \sca{c}{\gamma}, $$ where we used $ \gamma \in \Gamma(\mu,\nu)$ to get the
equality and $\phi\ominus\psi\leq c$ to get the inequality. As a conclusion,
\begin{equation*}
  \DP = \min_{\phi \ominus \psi \leq c} \sca{\phi}{\mu} - \sca{\psi}{\nu} \leq \max_{\gamma \in\Gamma(\mu,\nu)} \sca{c}{\gamma} = \KP \qedhere
\end{equation*}
\end{proof}


\subsubsection{Existence of solution for the dual problem}
Kantorovich's dual problem $\DP$ consists in maximizing a concave
(actually linear) functional under linear inequality constraints. It
can also also easily be turned into an unconstrained minimization
problem. The idea is quite simple: given a certain $\psi\in
\Class^0(Y)$, one wishes to select $\phi$ on $X$ which is as large as
possible (to maximize the term $\sca{\phi}{\mu}$ in $\DP$) while
satisfying the constraint $\phi \ominus \psi \leq c$. This constraint
can be rewritten as $$\forall x\in X,~~\phi(x) \leq \min_{y\in Y} c(x,y) + \psi(y).$$
The largest function $\phi$ satisfying it is
$\phi(x) = \min_{y\in Y} c(x,y) + \psi(y)$. Thus,
\begin{align*}
  \KP &= \sup_{\phi \ominus \psi \leq c} \sca{\phi}{\mu} - \sca{\psi}{\nu} \\
  &= \sup_{\psi \in \Class^0(Y)} \int_{X} \left(\min_{y\in Y} c(x,y) + \psi(y)\right)\dd\mu(x)
  - \int \psi(y)\dd\nu(y).
\end{align*}
This idea is at the basis of many algorithms to solve discrete
instances of optimal transport, but also useful in theory. It also
suggests to introduce the notion of $c$-transform.  .

\begin{definition}[$c$-Transform] \label{def:ctransform}
  The $c$-transform (resp. $\bar{c}$-transform) of a function $\psi:Y\to\Rsp \cup\{+\infty\}$
  (resp. $\phi:X\to\Rsp\cup\{+\infty\}$) is defined as
\begin{align}
  &\psi^{c}: x\in X\mapsto \inf_{y\in Y} c(x,y) + \psi(y)\\
    &\phi^{\bar{c}}: y\in Y\mapsto \sup_{x\in X} - c(x,y) + \phi(x) 
\end{align}
\end{definition}
Thanks to this notion of $c$-transform, one can reformulate the dual
problem $\DP$ as an unconstrained maximization problem:
\begin{equation}
  \label{eq:KD-unconstrained}
  \DP = \sup_{\psi\in\Class^0(Y)} \int_{X} \psi^c \dd \mu - \int_Y \psi\dd\nu.
\end{equation}

\begin{remark}[$c$-concavity, $\bar{c}$-convexity and $c$-subdifferential]
One can call a function $\phi$ on $X$ \emph{$c$-concave} if $\phi =
\psi^c$ for some $\psi:Y\to\Rsp\cup\{+\infty\}$ on $Y$. Note that we
use the word \emph{concave} because $\psi^c$ is defined through an
infimum. Conversely, a function $\psi$ on $Y$ is called
\emph{$\bar{c}$-convex} if $\psi = \phi^{\bar{c}}$ for some
$\phi:X\to\Rsp\cup\{+\infty\}$. Note the asymetry between the two
notions, which is due to the choice of the sign in the constraint in
Kantorovich's problem: in the two equivalent formulations
  $$ \KP = \sup_{\phi\oplus \psi \leq c} \sca{\phi}{\mu} +
\sca{\psi}{\nu} = \sup_{\phi\ominus \psi \leq c} \sca{\phi}{\mu} -
\sca{\psi}{\nu}, $$ we chose the second one, involving two minus
signs.  This choice will make it easier to explain some of the
algorithms we will present later in the chapter.  The
\textit{$c$-subdifferential} of a function $\psi$ on $Y$ is a subset
of $X\times Y$ defined by
\begin{equation} \label{eq:csubdiff}
\partial^c \psi := \left\{(x,y)\in X\times Y\mid \psi^c(x)-\psi(y) = c(x,y)  \right\},\\
\end{equation}
while the \textit{$c$-subdifferential} at a point  $y$ in $Y$ is given by 
\begin{equation} \label{eq:csubdiffy}
\partial^c \psi (y):= \left\{x\in X,\quad (x,y)\in \partial^c\psi\right\}.
\end{equation}
\end{remark}
\begin{remark}[Bilinear cost] When $c(x,y) = -\sca{x}{y}$, a function is $\bar{c}$-convex if
  and only if it is convex, and $\phi^{\bar{c}}$ is  the
  Legendre-Fenchel transform of $-\phi$.
\end{remark}

\begin{proposition}[Existence of dual potentials] \label{prop:dp-maximizer}
  $\DP$ admits a maximizer, which one can assume to be of the form
  $(\phi,\psi)$ such that $\phi = \psi^c$ and $\psi = \phi^{\bar{c}}$.
\end{proposition}

The existence of maximizers follows from the fact that a
$c$-concave/$\bar{c}$-convex function has the same modulus of
continuity as $c$.

(Recall that $\omega: \Rsp^+ \to \Rsp$ is a modulus
of continuity of a function $f:Z\to \Rsp$ on a metric space $(Z,d_Z)$
if it satisfies $\lim_{t\to 0} \omega(t) = 0$ and for every $z,z'\in
Z$, $\abs{f(z) - f(z')}\leq \omega(\dd_Z(z,z'))$.)
\begin{lemma}[Properties of $c$-transforms]\label{lem:cconc-modulus}
  Let $\omega: \Rsp^+\to\Rsp^+$ be a modulus of continuity for $c\in
  \Class^0(X\times Y)$ for the distance $$\dd_{X\times
    Y}((x,y),(x',y'))=\dd_X(x,x')+\dd_Y(y,y').$$ Then for every $\phi
  \in \Class^0(X)$ and every $\psi\in\Class^0(Y)$,
  \begin{itemize}
\item $\phi^{\bar{c}}$ and $\psi^c$ also admits $\omega$ as modulus of continuity.
\item $\psi^{c\bar{c}} \leq \psi$ and $\psi^{c\bar{c}c}=\psi^c$.
\item $\phi^{\bar{c}c} \geq \phi$ and $\phi^{\bar{c}c\bar{c}}=\phi^{\bar{c}}$.
  \end{itemize}
\end{lemma}

\begin{proof}
  Let us first prove the first point. Let $\psi\in\Class^0(Y)$ and for
  $x\in X$, let $y_x\in Y$ be a point realizing the minimum in the
  definition of $\psi^c$. Then,
  $$\psi^c(x') \leq c(x',y_x) + \psi(y_x) = \psi^c(x) + c(x',y_x) -
  c(x,y_x) \leq \psi^c(x) + \omega(\dd_X(x,x')). $$ Exchanging the
  role of $x$ and $x'$ we get $\abs{\psi^c(x') - \psi^c(x)}\leq
  \omega(\dd_X(x,x'))$ as desired.  The proof that $\phi^{\bar{c}}$ has
  the $\omega$ as modulus of continuity is similar. We prove now the
  second point. By definition, one has
  $$ \psi^{c\bar{c}}(y)=\max_{x\in X} \left( -c(x,y) +
  \min_{\tilde{y}\in Y} c(x,\tilde{y})+ \psi(\tilde{y})\right).
  $$
  By taking $\tilde{y}=y$, one gets $\psi^{c\bar{c}}(y)\leq \psi(y)$. Again, by definition, we have
  $$
  \psi^{c\bar{c}c}(x)=\min_{y\in Y}\left(c(x,y)+\max_{\tilde{x}\in X} \left( -c(\tilde{x},y) + \min_{\tilde{y}\in Y} c(\tilde{x},\tilde{y}) + \psi(\tilde{y})\right)\right).
  $$ By taking $\tilde{x}=x$ , one gets $\psi^{c\bar{c}c}(x)\geq
  \psi^c(x)$, while taking $\tilde{y}=y$ gives us $\psi^{c\bar{c}c}(x)\leq
  \psi^c(x)$.  The last point is obtained similarly.
\end{proof}

\begin{proof}[Proof of Proposition~\ref{prop:dp-maximizer}]
  Let $(\phi_n,\psi_n)$ be a maximizing sequence for $\DP$, i.e.
  $\phi_n \ominus \psi_n \leq c$ and $\lim_{n\to +\infty}
  \sca{\phi_n}{\mu} - \sca{\psi_n}{\nu} = \DP.$ Define $\hat{\phi}_n =
  \psi_n^c$ and $\hat{\psi}_n = \hat{\phi_n}^{\bar{c}}$. Then $\hat{\phi}_n
  \ominus\hat{\psi}_n \leq c$, $\phi_n \leq \hat{\phi_n}$ and $\psi_n \geq \hat{\psi_n}$, 
  which implies
    $$ \sca{\phi_n}{\mu} - \sca{\psi_n}{\nu} \leq
  - \sca{\psi_n}{\nu} \leq \sca{\hat{\phi}_n}{\mu}
  - \sca{\hat{\psi}_n}{\nu}, $$ implying that
  $(\hat{\phi}_n,\hat{\psi}_n)$ is also a maximizing sequence. Our
  goal is now to show that this sequence admits a converging
  subsequence. We first note that we can assume that
  $\hat{\phi}_n(x_0) = 0$ for all $n$, where $x_0$ is a given point in
  $X$: if this is not the case, we replace
  $(\hat{\phi}_n,\hat{\psi}_n)$ by
  $(\hat{\phi}_n-\hat{\phi}_n(x_0),\hat{\psi}_n+\hat{\phi}_n(x_0))$),
  which is also admissible and has the same dual value. In addition,
  by Lemma~\ref{lem:cconc-modulus}, the sequences $(\hat{\phi}_n)_{n}$
  and $(\hat{\psi}_n)_{n}$ are equicontinuous. By Arzel\`a-Ascoli's
  theorem, we deduce that they admit subsequences converging
  respectively to $\phi\in\Class^0(x)$ and $\psi\in\Class^0(Y)$, which
  are then maximizers for $\DP$.
\end{proof}


\subsubsection{Strong duality and stability of optimal transport plans}

We will prove strong duality first in the case where $\mu,\nu$ are
finitely supported, and will then use a density argument to deduce the
general case. As a byproduct of this theorem, we get a stability
result for optimal transport plans (i.e. a limit of optimal transport
plans is also optimal).

\begin{theorem}[Strong duality] \label{th:Kantorovich}
Let $X, Y$ be compact metric spaces and $c \in \Class^0(X\times
Y)$. Then the maximum is attained in $\DP$ and $\KP = \DP$.
\end{theorem}

\begin{corollary}[Support of OT plans] \label{coro:support}
  Let $\psi$ be a maximizer of \eqref{eq:KD-unconstrained} and $\gamma \in \Gamma(\mu,\nu)$ a transport plan. Then the
  two assertions are equivalent
 \begin{itemize}
\item $\gamma$ is an optimal transport plan  
\item 
  $\spt(\gamma) \subset \partial^c\psi := \{(x,y)\in X \times Y \mid \psi^c(x) - \psi(y) = c(x,y)\}$.
  \end{itemize}
\end{corollary}

As a consequence of Kantorovich duality, we can prove stability of
optimal transport plans and optimal transport maps. 
\begin{theorem}[Stability of OT plans] \label{th:stab}
  Let $X,Y$ be compact metric spaces and let $c\in \Class^0(X\times
 Y)$.  Consider $(\mu_k)_{k\in \Nsp}$ and $(\nu_k)_{k\in \Nsp}$ in
 $\Prob(X)$ and $\Prob(Y)$ converging weakly to $\mu$ and $\nu$
 respectively.  \begin{itemize} \item If
 $\gamma_k \in \Gamma(\mu_k,\nu_k)$ is optimal then, up to subsequences,
 $(\gamma_k)$ converges weakly to an optimal transport plan
 $\gamma \in \Gamma(\mu,\nu)$.  \item Let $(\phi_k,\psi_k)$ be optimal
 Kantorovich potentials in the dual problem between $\mu_k$ and
 $\nu_k$, satisfying $\psi_k = \phi_k^{\bar{c}}$ and $\phi_k
 = \psi_k^c$. Given a point $x_0 \in X$, define $\tilde{\psi}_k
 = \psi_k - \psi_k(x_0)$ and $\tilde{\phi}_k = \phi_k
 + \psi_k(x_0)$. Then, up to subsequences,
 $(\tilde{\psi}_k,\tilde{\phi_k})$ converges uniformly to
 $(\phi,\psi)$ a maximizing pair for $\DP$ satisfying $\phi=\psi^c$
 and $\psi=\phi^{\bar{c}}$.  \end{itemize}
\end{theorem}

The proof of Theorem~\ref{th:Kantorovich} relies on a simple
reformulation of strong duality -- similar to the Karush-Kuhn-Tucker
optimality conditions for optimization problems with inequality
constraints:

\begin{proposition}\label{prop:complementary-slackness} 
Let $\gamma \in \Gamma(\mu,\nu)$ and let $(\phi,\psi) \in
\Class^0(X)\times \Class^0(Y)$ such that $\phi \ominus \psi\leq
c$. Then, the following statements are equivalent:
\begin{itemize}
  \item  $\phi \ominus \psi = c$ $\gamma$-a.e.
  \item $\gamma$ minimizes  $\KP$, $(\phi,\psi)$ maximizes
     $\DP$ and $\KP = \DP$.
\end{itemize}
\end{proposition}

\begin{proof}  Assume that $\phi\ominus\psi = c$ $\gamma$-a.e. Then,
\begin{align*}
  \KP \leq \sca{c}{\gamma} = \sca{\phi\ominus\psi}{\gamma} =
  \sca{\phi}{\mu} - \sca{\psi}{\nu} \leq \DP
\end{align*}
Since in addition $\KP \geq \DP$, all inequalities are equalities,
which implies that $\KP = \DP$, $\gamma$ miminizes $\KP$ and
$(\phi,\psi)$ maximizes $\DP$.
Conversely, if $\KP = \DP$, $\gamma$ miminizes $\KP$ and $(\phi,\psi)$
maximizes $\DP$, then
$$  \sca{\phi}{\mu} - \sca{\psi}{\nu} = \DP  =   \KP = \sca{c}{\gamma} \geq \sca{\phi\ominus \psi}{\gamma} = 
\sca{\phi}{\mu} - \sca{\psi}{\nu} ,$$ implying that
$\phi\ominus\psi =c$ $\gamma$ a.e.
\end{proof}

The proof of Theorem~\ref{th:Kantorovich} also relies on a few
elementary lemmas from measure theory.
\begin{lemma}  \label{lemma:Support-hausdorff}
  If $\mu_N$ converges weakly to $\mu$, then for any point
  $x \in \spt(\mu)$ there exists a sequence $x_N\in\spt(\mu_{N})$
  converging to $x$.
\end{lemma}
\begin{proof} 
Consider $x\in \spt(\mu)$. For any $k\in\mathbb{N}$, consider the function
$\phi_k(z) = \max(1 - k d(x,z), 0),$ in
$\Class^0(X)$. Then,
$$ \lim_{N\to\infty} \sca{\phi_k}{\mu_N} = \sca{\phi_k}{\mu} > 0,$$
where the last inequality holds because $x$ belongs to the support of
$\mu$. Then, there exists $N_k$ such that for any $N\geq N_k$,
$\sca{\phi_{k}}{\mu_{N}}>0$, implying the existence of $x_N \in X$ such
that $x_N\in\spt(\mu_N)$ and $d(x_N,x) \leq 1/k$. By a diagonal argument, this allows to
construct a sequence of points $(x_N)_{N\in \Nsp}$ such that
$x_N\in\spt(\mu_N)$ and $\lim_{N\to +\infty} x_N = x$.
\end{proof}

\begin{lemma} \label{lemma:density-finite}
 Let $X$ be a compact space and  $\mu \in \Prob(X)$.  Then, there exists a
 sequence of finitely supported probability measures weakly converging
 to $\mu$.
\end{lemma}

\begin{proof}
  For any $\eps>0$, by compactness there exists $N$ points
  $x_1,\hdots,x_N$ such that $X \subseteq \bigcup_{i}
  \B(x_i,\eps)$. We define a partition $K_1,\hdots,K_N$ of $X$ recursively by
  $K_i = \B(x_i,\eps) \setminus (K_1 \cup ... \cup K_{i-1})$ and we introduce
  $$\mu_\eps := \sum_{1\leq i\leq N} \mu(K_i) \delta_{x_i}.$$ To prove
  weak convergence of $\mu_\eps$ to $\mu$ as $\eps\to 0$, take
  $\phi\in \Class^0(X)$. By compactness of $X$, $\phi$ admits a
  modulus of continuity $\omega$, i.e. $\lim_{t\to 0}\omega(t) = 0$
  and $\abs{\phi(x) - \phi(y)}\leq \omega(d(x,y))$. Using that
  $\diam(K_i)\leq \eps$, we get
  \begin{align*}
    \abs{\int \phi \dd \mu - \int \phi\dd\mu_\eps}
    = \abs{\sum_{1\leq i\leq N} \int_{K_i} \phi(x)-\phi(x_i) \dd \mu}  
    \leq \omega(\eps),
  \end{align*}
 We deduce $\lim_{\eps \to 0} \sca{\phi}{\mu_\eps} = \sca{\phi}{\mu}$,
 so that $\mu_\eps$ weakly converges to $\mu$.
\end{proof}

\begin{lemma} \label{lemma:finite}
  If $(\mu,\nu)\in\Prob(X)\times\Prob(Y)$ are finitely
  supported, $\KP = \DP$.
\end{lemma}
\begin{proof} Assume that 
  $ \mu = \sum_{1\leq i\leq N} \mu_i \delta_{x_i}, \nu = \sum_{1\leq j
    \leq M} \nu_j\delta_{y_j}, $ where all the $\mu_i$ and $\mu_j$ are
  strictly positive, and consider the linear programming problem
  $$ \KP' = \min \left\{ \sum_{i,j} \gamma_{ij} c(x_i,y_j) \mid
  \gamma_{ij}\geq 0, \sum_j \gamma_{ij} = \mu_i, \sum_i \gamma_{ij} =
  \nu_j \right\}, $$ which admits a solution which we denote $\gamma$.
  By Karush-Kuhn-Tucker theorem, there exists Lagrange multipliers
  $(\phi_i)_{1\leq i\leq N}, (\psi_j)_{1\leq j\leq M}$ and
  $(\pi_{ij})_{1 \leq i\leq N, 1\leq j\leq M}$ such that
  $$ \begin{cases} \phi_i - \psi_j - c(x_i,y_j)
    = \pi_{ij} \\ \gamma_{ij} \pi_{ij} = 0 \\ \pi_{ij} \leq
    0 \end{cases} $$ In particular, $\phi_i - \psi_j \leq c(x_i,y_j)$
    with equality if $\gamma_{ij}> 0$.  To prove strong duality
    between the original problems $\KP$ and $\DP$, we construct two
    functions $\hat{\phi},\hat{\psi}$ such that
    $\hat{\phi}\ominus\hat{\psi}\leq c$ with equality on the set
    $\{(x_i,y_j) \mid \gamma_{ij}>0\}$. For this purpose, we first introduce
  $$ \psi(y) = \begin{cases} \psi_i &\hbox{ if } y = y_i \\
    + \infty &\hbox{ if not } \end{cases}
  $$ and let $\hat{\phi} = \psi^c$, $\hat{\psi} = \hat{\phi}^c$. Let $i\in \{1,\cdots, M\}$. Since $\mu_i = \sum_j \gamma_{ij}\neq 0$, 
  there exists $j\in\{1,\hdots,M\}$ such that
  $\gamma_{ij} > 0$. Using $\gamma_{ij}\pi_{ij} =0$, we deduce that so
  that $\phi_i - \psi_j = c(x_i,y_j)$, giving
  $$ \hat{\phi}(x_i) = \min_{k \in \{1,\hdots,N\}} c(x_i,y_k) + \psi_k
  = c(x_i,y_j) + \psi_j = \phi_i.$$ Similarly, one can show that
  $\hat{\psi}(y_j) = \psi_j$ for all $j\in\{1,\hdots,M\}$. Finally,
  define $\gamma = \sum_{ij} \gamma_{ij}\delta_{(x_i,y_j)} \in
  \Gamma(\mu,\nu)$. Then one can check that $\hat{\phi}\ominus\hat{\psi}\leq c$
  with equality $\gamma$-a.e., so that $\KP = \DP$ by
  Proposition~\ref{prop:complementary-slackness}.
\end{proof}

\begin{proof}[Proof of Theorem~\ref{th:Kantorovich}]
  By Lemma~\ref{lemma:density-finite}, there exists a sequence $\mu_k
  \in \Prob(X)$ (resp. $\nu_k \in \Prob(Y)$) of finitely supported
  measures which converge weakly to $\mu$ (resp. $\nu$). We denote
  $\KP_k$ and $\DP_k$ the primal and dual Kantorovich problems between
  $\mu_k$ and $\nu_k$. By Proposition~\ref{prop:dp-maximizer}, there
  exists a solution $(\phi_k,\psi_k)$ of $\DP_k$, such that $\phi_k =
  \psi_k^c$ and $\psi_k = \phi_k^c$. Moreover, since strong duality
  holds for finitely supported measures (Lemma~\ref{lemma:finite}), we
  see (Proposition~\ref{prop:complementary-slackness}) that $\gamma_k$
  is supported on the set
  $$S_k = \{ (x,y)\in X\times Y\mid \phi_k(x) - \psi_k(y) = c(x,y)
  \}.$$ Adding a constant if necessary, we can also assume that
  $\phi_k(x_0) = 0$ for some point $x_0\in X$.  As $c$-concave
  functions, $\phi_k$ and $\psi_k$ have the same modulus of continuity
  as the cost function $c$ (see Lemma~\ref{lem:cconc-modulus}), and they are uniformly bounded (using
  $\phi_k(x_0) = 0$). Using Arzelà-Ascoli theorem, we can therefore
  assume that up to subsequences, $(\phi_k)$ (resp. $(\psi_k)$)
  converges to some $\phi$ (resp $\psi$) uniformly. Then, one easily
  sees that $\phi \ominus \psi \leq c$ so that $(\phi,\psi)$ are
  admissible for the dual problem $\DP$.
  
  By compactness of $\Prob(X\times Y)$, we can assume that the
  sequence $\gamma_k \in\Gamma(\mu_k,\nu_k)$ converges to some
  $\gamma\in\Gamma(\mu,\nu)$. Moreover, by
  Lemma~\ref{lemma:Support-hausdorff}, every pair
  $(x,y)\in\spt(\gamma)$ can be approximated by a sequence of pairs
  $(x_k,y_k)\in\spt(\gamma_k)$ i.e. $\lim_{k\to\infty} (x_k,y_k) =
  (x,y)$. Since $\gamma_k$ is supported on $S_k$ one has $c(x_k,y_k) =
  \phi_k(x_k) - \psi_k(x_k)$, which gives at the limit $c(x,y) =
  \phi(x) - \psi(y)$. We have just shown that for every point pair
  $(x,y)$ in $\spt(\gamma)$, $c(x,y) = \phi(x) - \psi(y)$ where
  $\phi,\psi$ is admissible. By
  Proposition~\ref{prop:complementary-slackness}, this shows that
  $\gamma$ and $(\phi,\psi)$ are optimal for their respective problems
  and that $\KP = \DP$.
  \end{proof}
  Corollary~\ref{coro:support} is a direct consequence of
  Proposition~\ref{prop:complementary-slackness} and of the strong
  duality $\KP = \DP$.
  

\subsubsection{Solution of Monge's problem for Twisted costs}
We now show how to use Kantorovich duality to prove the existence of
optimal transport maps when the source measure is absolutely
continuous on a compact subset of $\Rsp^d$ and when the cost function
satisfies the following condition:


\begin{definition}[Twisted cost]\label{def:twist}
  Let $\Omega_X,\Omega_Y \subseteq \Rsp^d$ be open subsets, and
  $c \in\Class^1(\Omega_X\times \Omega_Y)$. The cost function
  satisfies the twist condition
  if \begin{equation} \label{eq:twist} \forall
  x_0 \in \Omega_X,~~ \hbox{ the map } y\in \Omega_Y\mapsto
  v:=\nabla_x c(x_0,y) \in \Rsp^d \hbox{ is injective,
  } \end{equation} where $\nabla_x c(x_0,y)$ denotes the gradient of
  $x\mapsto c(\cdot,y)$ at $x=x_0$.  Given $x_0 \in \Omega_X$ and
  $v\in\Rsp^d$, we denote $y_c(x_0, v)$ the unique point (if it
  exists) such that $\nabla_x c(x_0,y_c(x_0,v)) = v$.  The map
  $v\mapsto y_c(x_0, v)$ is often called the $c$-exponential map at
  $x_0$.
\end{definition}

\begin{example}[Quadratic cost] Let $c(x,y) = \nr{x-y}^2$. Then, for any $x_0\in X$, the map
  $y \mapsto \nabla_x c(x_0,y) = 2(x_0-y)$ is injective, so that $c$
  satisfies the twist condition. Moreover, given $v\in \Rsp^d$, the
  unique $y$ such that $\nabla_x c(x_0,y) = 2(x_0-y) = v$ is $y = x_0
  - \frac{1}{2} v$, implying that $y_c(x_0,v) = x_0 - \frac{1}{2}v$.
\end{example}
The following theorem is due to
Brenier~\cite{brenier1991polar} in the case of the quadratic cost
(i.e. $c(x,y) = \nr{x-y}^2$) and Gangbo-McCann in the general case of
twisted costs~\cite{gangbo1996geometry}.

Given $X\subseteq \Omega_X \subset  \Rsp^d$, we define $\Probac(X)$ as the set of
probability measures on $\Omega_X$ that are absolutely continuous with
respect to the Lebesgue measure, and with support included in $X$.

\begin{theorem}[Brenier \cite{brenier1991polar}, Gangbo-McCann \cite{gangbo1996geometry}]
  \label{th:brenier} Let $c\in\Class^1(\Omega_X\times\Omega_Y)$ be a
  twisted cost, let $X \subseteq \Omega_X, Y\subseteq \Omega_Y$ be
  compact sets, and let $(\mu,\nu)\in\Probac(X)\times \Prob(Y)$. Then,
  there exists a $c$-concave function $\phi\in\Lip(X)$ such that $\nu
  = T_\#\mu$ where $T(x) = y_c(x,\nabla \phi(x))$. Moreover, the only
  optimal transport plan between $\mu$ and $\nu$ is $\gamma_T$.
\end{theorem}

\begin{example}
  If $h \in \Class^1(\Rsp^d)$ is strictly convex, in particular if
  $h(x) = \nr{x}^p$, then the map $x\mapsto \nabla h(x)$ is
  injective. Take $c(x,y) = h(x - y)$, so that $y\mapsto \nabla_x
  c(x,y) = \nabla_x h(x-y) = \nabla h(x-y)$ is also
  injective. Moreover, given $x_0\in\Rsp^d$ and $v \in \Rsp^d$, the
  unique solution $y$ to $v = \nabla h(x_0 - y)$ is $y = y_c(x_0,v) :=
  x_0 - (\nabla h)^{-1}(v)$. As a consequence, under the hypothesis of
  the theorem above, the transport map is of the form
  $$ T(x) = x - (\nabla h)^{-1}(\nabla \phi(x))$$ where $\phi$ is a
  $c$-convex function. 
\end{example}

The following lemma shows that a transport plan is induced by a
transport map if it is concentrated on the graph of a map.
\begin{lemma}\label{lemma:tmaps-are-tplans} 
  Let $\gamma\in\Gamma(\mu,\nu)$ and $T:X\to Y$ measurable be such
  that $\gamma(\{ (x,y) \in X\times Y\mid T(x)\neq y\}) = 0$. Then,
  $\gamma = \gamma_T$.
\end{lemma}


\begin{proof} By definition of $\gamma_T$ one has  $\gamma_T(A\times B) = \mu(T^{-1}(B)\cap A)$ for all Borel sets $A\subseteq X$ and $B\subseteq Y$. On the other hand,
\begin{align*}
  \gamma(A\times B) &= \gamma(\{ (x,y) \mid x\in A, \hbox{ and } y\in
  B \})\\ &= \gamma(\{ (x,y) \mid x\in A, y\in B \hbox{ and } y=T(x)
  \}) \\ &= \gamma(\{ (x,y) \mid x\in A \cap T^{-1}(B), y=T(x) \}
  \\ &= \mu(A\cap T^{-1}(B)),
\end{align*}
thus proving the claim.
\end{proof}

\begin{proof}[Proof of Theorem~\ref{th:brenier}]
Enlarging $X$ if necessary (while keeping it compact and inside
$\Omega_X$), we may assume that $\spt(\mu)$ is contained in the
interior of $X$. First note that by compactness of $X\times Y$ and
since $c$ is $\Class^1$, the cost $c$ is Lipschitz on $X\times Y$.
Take $(\phi,\phi^{\bar{c}})$ a maximizing pair for $\DP$ with $\phi$
$c$-concave. By the formula
$\phi(x) = \min_{y \in Y} c(x,y) + \phi^{\bar{c}}(y)$ one can see that
$\phi$ is Lipschitz. By Rademacher theorem, $\phi$ is differentiable
Lebesgue almost everywhere, and by the hypothesis $\mu\in\Probac(X)$,
it is therefore differentiable on a set $B \subseteq \spt(\mu)$ with
$\mu(B)=1$. Consider an optimal transport plan $\gamma
\in\Gamma(\mu,\nu)$. For every pair of points
$(x_0,y_0)\in\spt(\gamma) \cap B\times Y$, we have
  $$ \forall x\in X, \phi(x) - c(x,y_0) \leq \phi^{\bar{c}}(y_0)$$ with
equality at $x = x_0$, so that $x_0$ maximizes the function $\phi -
c(\cdot,y_0)$. Since $x_0 \in\spt(\mu)$, $x_0$ belongs to the interior
of $X$, one necessarily has $\nabla\phi(x_0) = \nabla_x c(x_0,y_0)$. Then, by the
twist condition, one necessarily has $y_0 = y_c(x_0, \nabla
\phi(x_0))$. This shows that any optimal transport plan $\gamma$ is
supported on the graph of the map $T: x\in B\mapsto y_c(x_0,\nabla
\phi(x_0))$, and $\gamma = \gamma_T$ by the previous lemma.
\end{proof}

We finish this section with a stability result for optimal transport
maps (a more general result can be found in \cite[Chapter
  5]{villani2008optimal}).

\begin{proposition}[Stability of OT  maps] \label{prop:stab-maps}
  Let $X\subseteq \Omega_X$ and $Y\subseteq \Omega_Y$ be compact
  subsets of open sets $\Omega_X,\Omega_Y\subseteq \Rsp^d$, and
  $c\in\Class^1(\Omega_X\times\Omega_Y)$ be a twisted cost.  Let $\rho
  \in \Probac(X)$, and let $(\mu_k) \in \Prob(Y)$ be a sequence of
  measures converging weakly to $\mu\in\Prob(Y)$. Define $T_k$
  (resp. $T$) as the unique optimal transport map between $\rho$ and
  $\mu_k$ (resp. $\rho$ and $\mu$). Then,
  $\lim_{k\to+\infty}\nr{T_k-T}_{\LL^1(\rho)} = 0$.
\end{proposition}

\begin{remark}  Note that unlike the stability theorem for
optimal transport plans (Theorem~\ref{th:stab}), the convergence in
Proposition~\ref{prop:stab-maps} is for the whole sequence and not up
to subsequence. This theorem is not quantitative, and there exists
very few quantitative variants of this theorem. We are aware of two
such results. To state them, given a fixed $\rho\in\Prob(X)$ and
$\mu\in \Prob(Y)$, we denote $T_\mu$ the unique optimal transport map
between $\rho$ and $\mu$.
\begin{itemize}
\item A first result of Ambrosio, reported in an article of Gigli
  \cite[Proposition 3.3 and Corollary 3.4]{gigli2011holder}, shows
  that if $\mu_0$ is such that the optimal transport map $T_{\mu_0}$ is Lipschitz, then
  $$ \nr{T_\mu - T_{\mu_{0}}}_{\LL^2(\rho)}^2 \leq C(\mu_0)
  \Wass_2(\mu,\mu_0). $$ The theorem in \cite{gigli2011holder} holds
  for the quadratic cost on $\Rsp^d$. It was recently generalized to
  other cost functions \cite{ambrosio2019optimal}.
\item Berman \cite{berman2018convergence} proves a global estimate,
  not assuming the regularity of $T_{\mu_0}$ but with a worse Hölder
  exponent, of the form
  $$\nr{T_\mu - T_{\mu_0}}_{\LL^2(\rho)}^2 \leq C
  \Wass_1(\mu,\mu_0)^{1/2^{d-1}},$$ assuming that $\rho$ is bounded
  from below on a compact convex domain of $\Rsp^d$, when the cost is
  quadratic. The constant then $C$ only depends on $X,Y$ and $\rho$. Recently 
  a similar bound with an exponent independent on the dimension
  was obtained by M\'erigot, Delalande and Chazal~\cite{merigot2019quantitative}:
  $$
  \nr{T_\mu - T_{\mu_0}}_{\LL^2(\rho)}^2 \leq C
  \Wass_1(\mu,\mu_0)^{1/15}.
  $$
\end{itemize}
\end{remark}

\begin{proof}
  As before, without loss of generality, we assume that $\spt(\sigma)$
  lies in the interior of $X$.  Let $(\phi_k,\psi_k)$ be solutions to
  $\DP_k$, which are $c$-conjugate to each other, and such that
  $\phi_k(x_0) = 0$ for some $x_0\in X$. Then, by stability of
  Kantorovich potentials, there exists a subsequence $(\phi_k,\psi_k)$
  (which we do not relabel) which converges uniformly to
  $(\phi,\psi)$. Moreover, $(\phi,\psi)$ are Kantorovich potentials
  for $\DP$, and are also $c$-conjugate to each other.

  Since $\phi, \phi_k\in \Lip(X)$ are differentiable almost
  everywhere, there exists a subset $Z \subseteq \spt(\sigma)$ with
  $\mu(Z) = 1$ and such that for all $x\in Z$, $\nabla \phi_k$ exists
  for all $k$ and $\nabla \phi$ exists.  Let $x\in Z$. Using
  $$ \phi_k(x) - \psi_k(T_k(x)) = c(x,T_k(x)), $$ we get that for any
  cluster point $y$ of the sequence $(T_k(x))_k$,
  $$ \begin{cases} \phi(x) - \psi(y) = c(x,y),\\
    \phi(x') - \psi(y) \leq c(x',y) &\forall x'\in X\end{cases},
    $$ where the second inequality is obtained using $\phi_k \ominus
    \psi_k \leq c$. Thus, as in the proof of Brenier-McCann-Gangbo's
    theorem, $x$ is a minimizer of $c(\cdot,y) - \phi$, i.e. $\nabla_x
    c(x,y) = \nabla \phi(x)$, implying that $y = y_c(x,\nabla \phi(x))
    = T(x)$.  By compactness, this shows that the whole sequence
    $(T_k(x))_k$ converges to $S(x)$. Therefore, $T_k$ converges
    $\sigma$-almost everywhere to $T$, and $\LL^1(\sigma)$ convergence
    follows easily.
\end{proof}

\subsection{Kantorovich's functional}
\label{subsec:Kantorovich-functional}
As already mentioned in Equation~\eqref{eq:KD-unconstrained}, the
Kantorovich's dual problem $\DP$ can be expressed as an unconstrained
maximization problem:
$$
  \DP = \max_{\psi\in\Class^0(Y)} \int_{X} \psi^c \dd \mu - \int_Y \psi\dd\nu.
$$
This motivates the definition of \emph{Kantorovich's functional} as follows
%
%
\begin{definition} \label{def:Kant}
  The   Kantorovitch functional is defined on $\Class^0(Y)$ by
  \begin{equation} \label{eq:Kant}
    \Kant(\psi)= \int_{X} \psi^c \dd \mu - \int_Y \psi\dd\nu.
  \end{equation}
\end{definition}
The Kantorovitch dual problem therefore amounts to maximizing the Kantorovitch functional:
$$
  \DP = \max_{\psi\in\Class^0(Y)}   \Kant(\psi).
$$
This subsection is devoted to the general computation of the
superdifferential of Kantorovich's functional when $Y$ is finite. This
computation will be used to construct and study algorithms for
discretized optimal transport problems.
The definition, as well as basic 
properties on the superdifferential $\partial^+F$ of a function $F$ are recalled in Appendix~\ref{sec:appendixconvexanalysis}.

\begin{proposition}\label{prop:gradientkant} Let $X$ be a compact space, $Y$ be finite, 
  $c\in \Class^0(X\times Y)$ and $\mu\in \Prob(X)$ and
  $\nu\in\Prob(Y)$.  Then, for all $\psi_0\in\FYR$,
  \begin{equation}
    \partial^+ \Kant(\psi_0) = \{ \Pi_{Y\#}\gamma - \nu \mid \gamma \in \Gamma_{\psi_0}(\mu) \}.
  \end{equation}
  where $\Gamma_{\psi_0}(\mu)$ is the set of probability measures on $X\times Y$
  with first marginal $\mu$ and supported on the $c$-subdifferential $\partial^c\psi_0$ (defined in Eq.~\eqref{eq:csubdiff}), i.e.
  \begin{equation}
    \Gamma_{\psi_0}(\mu) = \{ \gamma \in \Prob(X\times Y) \mid
    \Pi_{X\#}\gamma = \mu \hbox{ and } \spt(\gamma)\subseteq
    \partial^c\psi_0 \}.
    \end{equation}
\end{proposition}

\begin{proof} Let $\gamma \in \Gamma_{\psi_0}(\mu)$. Then, for all $\psi\in\FYR$,
  \begin{align*}
    \Kant(\psi) &= \int \psi^c \dd \mu - \int \psi \dd \nu \\
    &= \int \psi^c(x) \dd\gamma(x,y) - \int \psi\dd\nu \\
    &\leq \int c(x,y) + \psi(y) \dd \gamma(x,y) - \int \psi \dd \nu,
  \end{align*}
  where we used $\Pi_{X\#} \gamma = \mu$ to get the second equality and
  $\psi^c(x) \leq c(x,y) + \psi(y)$ to get the inequality. Note
  also that equality holds if $\psi=\psi_0$, by assumption on the
  support of $\gamma$.  Hence,
\begin{align*}
  \Kant(\psi) &\leq \Kant(\psi_0) + \int (\psi(y) - \psi_0(y)) \dd \gamma(x,y) -
  \int (\psi - \psi_0) \dd\nu \\
  &= \Kant(\psi_0) + \sca{\Pi_{Y\#}\gamma - \nu}{\psi - \psi_0}.
\end{align*}
This implies by definition that $\Pi_{Y\#}\gamma - \nu$ lies in the
superdifferential $\partial^+\Kant(\psi_0)$, giving us the inclusion
$$ D(\psi_0) := \{ \Pi_{Y\#}\gamma - \nu\mid \gamma \in \Gamma_{\psi_0}(\mu)\}
\subseteq \partial^+ \Kant(\psi_0).$$ Note also that  the
superdifferential of $\Kant$ is non-empty at any $\psi_0\in
\FYR$, so that $\Kant$ is concave. 
As a concave function on the finite-dimensional space $\Rsp^Y$,
$\Kant$ is differentiable almost everywhere and one has $\partial
\Kant^+(\psi)=\{\nabla \Kant(\psi)\}$ at differentiability points.

We now show that $\partial \Kant^+(\psi_0) \subset D(\psi_0)$, using the
characterization of the subdifferential recalled in the Appendix:
$$
\partial \Kant^+(\psi_0) = \conv \left\{\lim_{n\to\infty} \nabla \Kant(\psi^n)\mid (\psi^n)_{n \in \Nsp} \in S\right\},
$$ where $S$ is the set of sequences $(\psi^n)_{n \in \Nsp}$ that
converge to $\psi_0$, such that $\nabla \Kant(\psi^n)$ exist and admit a
limit as $n\to +\infty$. Let $v=\lim_{n\to\infty} \nabla
\Kant(\psi^n)$ , where $(\psi^n)_{n \in \Nsp}$ belongs to the set
$S$. For every $n$, there exists $\gamma^n\in \Gamma_{\psi^n}(\mu)$
such that $\nabla \Kant(\psi^n)=v^n := \Pi_{Y\#}\gamma^n  - \nu$. 
By compactness of $\Prob(X\times Y)$, one can assume (taking a subsequence if necessary)
that $\gamma^n$ weakly converges  to some $\gamma$, and it is not
difficult to check that $\gamma\in\Gamma_{\psi_0}(\mu),$ ensuring that
the sequence $v^n$  converges to some $v\in D(\psi_0)$.
 Thus,
$$ \left\{\lim_{n\to\infty} \nabla \Kant(\psi^n)\mid (\psi^n)_{n \in
  \Nsp} \in S\right\} \subseteq D(\psi_0). $$ Taking the convex hull and using the convexity of $D(\psi_0)$, we get $\partial^+
\Kant(\psi_0) \subseteq D(\psi_0)$ as desired. \qedhere

\end{proof}

As a corollary of this proposition, we obtain an explicit expression
for the left and right partial deriatives of $\Kant$, and a
characterization of its differentiability. In this corollary, we use
the terminology of semi-discrete optimal transport
(Section~\ref{sec:semidiscrete}), and we will refer to the
$c$-subdifferential at $y\in Y$ as \emph{Laguerre cell associated to
$y$} and we will denote it by $\Lag_y(\psi)$.
\begin{equation} \label{eq:Lag0}
\Lag_y(\psi) := \{ x\in X\mid\forall z\in Y, c(x,y)+ \psi(y)\leq c(x,z)+\psi(z)\}.
\end{equation}
We also need to
introduce the \emph{strict Laguerre cell} $\SLag_y(\psi)$:
\begin{equation} \label{eq:SLag0}
\SLag_y(\psi) := \{ x\in X\mid\forall z\in Y, c(x,y)+ \psi(y)< c(x,z)+\psi(z)\}.
\end{equation}

\begin{corollary}[Directional derivatives of $\Kant$] \label{coro:Kant-grad}
 Let $\psi\in\FYR$, $y \in Y$ and define $\kappa(t) = \Kant(\psi^t)$
 where $\psi^t = \psi + t \one_{y}$. Then, $\kappa$ is concave and
 $$ \partial^+ \kappa(t) = [\mu(\SLag_{y}(\psi^t))
 - \nu(\{y\}), \mu(\Lag_{y}(\psi^t)) - \nu(\{y\})] $$ In particular
 $\Kant$ is differentiable at $\psi\in\FYR$ iff
 $\mu(\Lag_y(\psi)\setminus \SLag_y(\psi))=0$ for all $y\in Y$, and in
 this case
$$
\nabla \Kant(\psi) =\Big( \mu(\Lag_y(\psi)) - \nu(\{y\}) \Big)_{y\in Y}.
$$
\end{corollary}

\begin{proof} Using Hahn-Banach's  extension theorem, one can easily see that the super-differential
of $\kappa$ at $t$ is the projection of the super-differential $\Kant$
at $\psi^t$: $$\partial^+ \kappa(t)
= \left\{ \sca{\pi}{\one_{y}} \mid \pi \in \partial^+\Kant(\psi^t) \right\}. $$
Combining with the previous proposition we get
$$
\begin{aligned}
 \partial^+ \kappa(t)
& = \left\{ \sca{\Pi_{X\#\gamma}
- \nu}{\one_{y}} \mid \gamma \in \Gamma_{\psi^t}(\mu) \right\} \\
&= \left\{ \gamma(X\times \{y\})
- \nu(\{y\}) \mid \gamma \in \Gamma_{\psi^t}(\mu) \right\}
\end{aligned}
$$
To obtain the desired formula for $\partial\kappa^+(t)$, it remains to prove that
$$ \max \left\{\gamma(X\times \{y\}) \mid \gamma \in \Gamma_{\psi^t}(\mu)\right\}
= \mu(\Lag_y(\psi^t)), $$
$$ \min \left\{\gamma(X\times \{y\}) \mid \gamma \in \Gamma_{\psi^t}(\mu)\right\}
= \mu(\SLag_y(\psi^t)). $$ We only prove the first equality, the
second one being similar. Denote $Z = X \setminus \Lag_y(\psi^t)$, so that
for any $\gamma\in \Gamma_{\psi^t}(\mu)$,
$$
\begin{aligned}
 \gamma(X\times\{y\}) 
= \gamma(\Lag_y(\psi^t) \times \{y\}) + \gamma(Z\times\{y\}).
\end{aligned} $$
Moreover, by definition of $\Lag_y(\psi^t)$,
$Z\times\{y\} \cap \partial^c \psi^t = \emptyset$. Since $\gamma$
belongs to $\Gamma_{\psi^t}(\mu)$, we have
$\spt(\gamma)\subseteq \partial^c \psi^t$ so that
$\gamma(Z\times\{y\}) = 0$. This gives us
$$ \gamma(X\times\{y\})
= \gamma(\Lag_y(\psi^t) \times \{y\}) \leq \gamma(\Lag_y(\psi^t)\times Y)
= \mu(\Lag_y(\psi^t)), $$ where we used $\Pi_{X\#}\gamma = \mu$ to get
the last equality. This proves that
$$\sup \left\{\gamma(X\times \{y\}) \mid \gamma \in \Gamma_{\psi^t}(\mu)\right\}
\leq \mu(\Lag_y(\psi^t)).$$
To show equality, we consider an explicit
$\gamma\in\Gamma_{\psi^t}(\mu)$ using a map $T: X\to Y$ defined as
follows: for $x\in \Lag_y(\psi^t)$, we set $T(x) = y$ and for points
$x\not\in \Lag_y(\psi^t)$, we define $T(x)$ to be an arbitrary $z \in
Y$ such that $x\in \Lag_z(\psi^t)$. Then, one can readily check that
$\gamma = (\id,T)_\# \mu$ belongs to $\Gamma_{\psi^t}(\mu)$ and that
by construction, $\gamma(X\times\{y\}) = \mu(\Lag_y(\psi^t))$.
\end{proof}

\section{Discrete optimal transport}
In this part we present two algorithms for solving discrete optimal
transport problems, wich can be both be interpreted using
Kantorovich's duality:
\begin{itemize}
\item The first one is Bertsekas' auction algorithm,
which allows to solve optimal transport problem where the source and
targed measures are uniform over two sets with the same cardinality, a
case known as the \emph{assignment problem} in combinatorial
optimization. Bertsekas' algorithm is a
\emph{coordinate-ascent method} that iteratively modifies the
coordinates of the dual variable $\psi$ so as to reach a maximizer of
the Kantorovitch functional $\Kant(\psi)$.
\item The second algorithm is
the \emph{Sinkhorn-Knopp's algorithm} that allows to solve the
entropic regularization of (discrete) optimal transport problems. This
algorithm can be seen as a \emph{block-coordinate ascent method} since
it amounts to maximizing the dual of the regularized optimal transport
problem, denoted by $\Kant^\eta(\phi,\psi)$, by alternatively
optimizing with respect to the two dual variables $\varphi$ and
$\psi$.
\end{itemize}

\subsection{Formulation of discrete optimal transport} 
\subsubsection{Primal and dual problems}
We consider in this section that the two sets $X$ and $Y$ are finite, and
we consider two discrete probability measures $\mu = \sum_{x\in X} \mu_x
\delta_x$ and $\nu = \sum_{y\in Y} \nu_y \delta_y$. This setting
occurs frequently in applications. The set of transport plans is then
given by
\[
\Gamma(\mu,\nu)=\left\{ \gamma=\sum_{x,y} \gamma_{x,y} \delta_{(x,y)} \mid\
\gamma_{x,y}\geq 0,\
\sum_{y\in Y}\gamma_{x,y}=\mu_x,\
\sum_{x\in X} \gamma_{x,y} = \nu_y
\right\},
\]
and is often referred to as the \emph{transportation polytope}.  In
this discrete setting, we will conflate a transport plan $\gamma \in
\Gamma(\mu,\nu)$ with the matrix $(\gamma_{x,y})_{(x,y)\in X\times
  Y}$, which formally is the density of $\gamma$ with respect to the
counting measure. The constraint $\sum_{y} \gamma_{x,y} = \mu_x$
encodes the fact that all mass from $x$ is transported somewhere in
$Y$, while the constraint $\sum_{x}\gamma_{x,y} = \nu_y$ tells us that
the mass at $y$ is transported from somewhere in $X$.  The
Kantorovitch problem for a cost function $c:X\times Y \to \Rsp$ then
reads
\begin{equation}
\KP = \min_{\gamma\in \Gamma(\mu,\nu)} \sum_{x\in X,y\in Y} c(x,y) \gamma_{x,y}.
\end{equation}
As seen in Section~\ref{sec:kantorovitchduality}, the dual $\DP$ of
this linear programming problem amounts to maximizing the Kantorovitch
functional $\Kant$ \eqref{eq:Kant}, which in this setting can be expressed as 
\begin{equation} \label{eq:Kant-disc}
\Kant(\psi) = \sum_{x\in X} \min_{y\in Y}(c(x,y) + \psi(y))\mu_x  - \sum_{y\in Y} \psi(y) \nu_y,
\end{equation}
where $\psi\in\FYR$ is a function over the finite set $Y$.  Since
strong duality holds (Theorem~\ref{th:Kantorovich}), one has
\begin{equation*} \label{eq:discreteunconstrained}
\KP = \DP = \max_{\psi\in\FYR} \Kant(\psi). 
\end{equation*}

\begin{remark}
  Knowing a maximizer $\psi\in \FYR$ of $\Kant$ does not directly
allow to recover an optimal transport plan $\gamma$. However, by
Corollary~\ref{coro:support}, we know that any transport plan $\gamma
\in\Gamma(\mu,\nu)$ is optimal if and only if its support is included
in the $c$-subdifferential of $\psi$:
$$\spt(\gamma) \subset \partial^c\psi = \Big\{(x,y)\in X\times Y\mid
\psi^c(x)=c(x,y)+\psi(y)\Big\}.$$
\end{remark}

\subsection{Linear assignment via coordinate ascent}
\subsubsection{Assignment problem}
When the two sets $X$  and $Y$ have the same cardinal $N$ and when $\mu$ and $\nu$ are uniform probability measures over these sets, namely
\begin{equation} \label{eq:unif}
  \mu = \frac{1}{N} \sum_{x\in X}\delta_x, \quad \nu = \frac{1}{N}
  \sum_{y\in Y}\delta_y,
\end{equation}
then Monge's problem corresponds to the \emph{(linear) assignment
  problem} $\AP$ which is one of the most famous combinatorial
optimization problem:
\begin{equation} \label{equation} \AP = \min\left\{\frac{1}{N} \sum_{x\in X} c(x,\sigma(x))
\mid \sigma:X\to Y \mbox{ is a bijection } \right\}.
\end{equation}
This problem and its variants have generated a very important amount
of research, as demonstrated by the bibliography of the book by
Burkard, Dell'Amico and Martello on this topic
\cite{burkard2009assignment}.

Note that the set of bijections from $X$ to $Y$ has cardinal $N!$,
making it practically impossible to solve $\AP$ through direct
enumeration. Using Birkhoff's theorem on bistochastic matrices, we
will show that the assignment problem coincides with the Kantorovitch
problem.

\begin{definition}[Bistochastic matrices]
  A $N$-by-$N$ bistochastic matrix is a square matrix $M \in
  \Mat_N(\Rsp)$ with non-negative coefficients such that the sum of
  any row and any column equals one:
  $$\forall i\in\{1,\hdots,N\}, \sum_j M_{ij} = 1,\quad \forall
  j\in\{1,\hdots,N\}, \sum_i M_{ij} = 1 $$ We denote the set of
  $N$-by-$N$ bistochatic matrices as $\Bistoch_N\subseteq
  \Mat_N(\Rsp)$.
\end{definition}
\begin{definition}[Permutation matrix]
The set of permutations (bijections) from $\{1,\hdots,N\}$ to itself
is denoted $\Permut_N$. Given a permutation $\sigma \in \Permut_N, $
we associate the permutation matrix
$$ M[\sigma]_{ij} = \begin{cases} 1 \hbox{ if } \sigma(i) = j \\ 0 \hbox{ if not} \end{cases}. $$
\end{definition}

One can easily check that if $\sigma$ is a permutation, then
$M[\sigma]$ belongs to $\Bistoch_N$. Birkhoff's theorem on the other
hand asserts that the extremal points of the polyhedron $\Bistoch_N$
are permutation matrices, implying thanks to Krein-Milman theorem that
every bistochastic matrix can be obtained as a (finite) convex
combination of permutation matrices.

\begin{theorem}[Birkhoff]
  The extremal points of $\Bistoch_N$ are the permutation matrices. In
  particular, $\Bistoch_N = \conv\{ M[\sigma] \mid \sigma \in
  \Permut_N\}$.
\end{theorem}

Kantorovitch's problem $\KP$ amounts to minimizing a linear function
over the set of bistochastic matrices which is convex. Birkoff's
theorem implies that there exists a bijection that solves this problem
$\KP$, hence the following theorem:
\begin{theorem}\label{thm:assignment}
Let $\mu$ and $\nu$ be as in \eqref{eq:unif}. Then, $\AP = \KP$.
\end{theorem}

\begin{proof} Take an arbitrary ordering of the points in $X$ and $Y$, i.e. 
  $X = \{x_1,\hdots,x_N\}$ and $Y = \{y_1,\hdots,y_N\}$. Then $\gamma
  \in \Rsp^{X\times Y} \simeq \Mat_N(\Rsp)$ is a transport plan
  between $\mu$ and $\nu$ iff $N\gamma \in \Bistoch_N$. Since
  bistochastic matrices include permutation matrices, we have $\KP\leq
  \AP$, and the converse follows from the fact that the minimum in
  $\KP$ is attained at an extreme point of $\Bistoch_N$, i.e. a
  permutation matrix.
\end{proof}

\subsubsection{Dual coordinate ascent methods} 
We follow Bertsekas \cite{bertsekas1981new} by trying to solve the assignment
problem $\AP$ through the unconstrained dual problem
\eqref{eq:discreteunconstrained}.  Combining
Theorem~\ref{thm:assignment} and Corollary~\ref{coro:support}, we have
the following proposition.
\begin{proposition}The following statements are equivalent
\begin{itemize}
\item $\psi$ is a global maximizer of the Kantorovitch functional $\Kant$
\item There exists a bijection $\sigma:X \to Y$ that satisfies
$$
\forall x \in X,~  c(x,\sigma(x))+ \psi(\sigma(x)) = \min_{y \in Y} c(x,y) + \psi(y).
$$
\end{itemize}
A bijection $\sigma$ satisfying this last equation is a solution to the linear assignment problem. 
\end{proposition}
The idea of Bertsekas \cite{bertsekas1981new} is to iteratively modify
the weights $\psi\in\Rsp^Y$ so as to reach a maximizer of $\Kant$.
By
Corollary~\ref{coro:Kant-grad}, the gradient of the
Kantorovitch functional, when it exists, is given by
$$
\nabla \Kant(\psi)   = \frac{1}{N} \Big( \card(\Lag_y(\psi)) - 1 \Big)_{y\in Y}.
$$
In addition, recalling the definition of a Laguerre cell,
$$ \Lag_y(\psi)  = \{ x\in X\mid \forall y'\in Y,~ c(x,y) +
\psi(y) \leq c(x,y') + \psi(y')\} ,$$ one can see that $\card(\Lag_y(\psi))$ is obviously
decreasing when $\psi(y)$ increases.  Therefore, in order to maximize
the concave function $\Kant$, it is natural to increase the weight
$\psi(y)$ of any Laguerre cells that satisfy
$\card(\Lag_y(\psi))>1$. In the following lemma, we calculate the
optimal increment, which is known as the \emph{bid}.

\begin{lemma}[Bidding increment]
Let $\psi\in \FYR$ and $y_0\in Y$ be such that $\Lag_{y_0}(\psi) \neq
\emptyset$. Then the maximum of the function $t\to \Kant(\psi + t
\one_{y_0})$ is reached at
$$
\bid_{y_0}(\psi) = \max \{ \bid_{y_0}(\psi,x),\ x \in \Lag_{y_0}(\psi)\},
$$
where
$$
\bid_{y_0}(\psi,x) := \left(\min_{y\in Y\setminus{y_0}}c(x,y) + \psi(y)\right) - (c(x,y_0) + \psi(y_0)).
$$
\end{lemma}
\begin{proof}
Denote $\psi^t = \psi + t\one_{y_0}$. For $t>0$ one has $\Lag_{y_0}(\psi^t)
\subseteq \Lag_{y_0}(\psi)$. Remark also that for every $x\in X$, one has
$$
\begin{array}{lll}
x \in \Lag_{y_0}(\psi^t) &\Leftrightarrow &\forall z\neq y_0\quad  c(x,y_0)+\psi(y_0)+t \leq c(x,z) + \psi(z)\\
&\Leftrightarrow & t \leq \left(\min_{z\in Y\setminus{y_0}}c(x,z) + \psi(z)\right) - (c(x,y_0) + \psi(y_0))\\
&\Leftrightarrow & t \leq \bid_{y_0}(\psi,x)
\end{array}
$$ This implies that $\Lag_{y_0}(\psi^t) \ne \emptyset$ if and only if
$t \leq \bid_{y_0}(\psi)$.  By Corollary~\ref{coro:Kant-grad}, the
upper-bound of the superdifferential $\partial^+ \kappa(t)$ of the function $\kappa(t) = \Kant(\psi+t\one_{y_0})$
is $ \mu(\Lag_{y_0}(\psi^t)) - \frac{1}{N}$. It is
non-negative for $t\in[0,\bid_{y_0}(\psi)]$ and strictly negative for
$t> \bid_{y_0}(\psi)$. This directly implies (for instance by~\eqref{eq:SDcalc}) that $0\in \partial^+ \kappa(\bid_{y_0}(\psi))$, so that the largest
  maximizer of $\kappa$ is $\bid_{y_0}(\psi)$.
\end{proof}

\begin{remark}[Economic interpretation of the bidding increment.]
Assume that $Y$ is a set of houses owned by one seller and $X$ is a
set of customers that want to buy a house.  Given a set of prices
$\psi:Y\to\Rsp$, each customer $x\in X$ will make a compromise between
the location of a house $y\in Y$ (measured by $c(x,y)$) and its price
(measured by $\psi(y)$) by choosing a house among those minimizing
$c(x,y)+\psi(y)$. In other words, $x$ chooses $y$ iff
$x\in \Lag_y(\psi)$. Let $y$ be a given house. The seller of $y$ wants
to maximize his profit, hence to increase $\psi(y)$ as much as
possible while keeping (at least) one customer. Let $x\in X$ be a
customer interested in the house $y$ (i.e. $x\in \Lag_y(\psi)$). Then,
$\bid_{y,x}(\psi)$ tells us how much it is possible to increase the
price of $y$ while keeping it interesting to $x$. The best choice for
the seller is to increase the price by the maximum bid, which is the
maximum raise so that there remains at least one customer, giving the
definition of $\bid_{y}(\psi)$.
\end{remark}

\begin{remark}[Naive coordinate ascent]
A naive algorithm would be to would choose at each step a coordinate
$y\in Y$ such that $\Lag_y(\psi)\neq \emptyset$ and to increase
$\psi(y)$ by the bidding increment $\bid_y(\psi)$. In practice, such
an algorithm might get stuck at a point which is not a global
maximizer, a phenomenon which is referred to as \emph{jamming} in
\cite[\S2]{bertsekas1988dual}. In practice, this can happen when some
bidding increments $\bid_y(\psi)$ vanishes, see
Remark~\ref{rem:no-convergence} below. Note that this is a particular
case of the well known fact that coordinate ascent algorithms may
converge to points that are not maximizers, when the maximized
functional is nonsmooth.
\end{remark}

In order to tackle the problem of non-convergence of coordinate
ascent, Bertsekas and Eckstein changed the naive algorithm outlined
above to impose that the bids are at least $\eps>0$. To analyse their
algorithm, we introduce the notion of \emph{$\eps$-complemen\-ta\-ry
slackness}, where $\eps$ can be seen as a tolerance.

\begin{definition}[$\eps$-Complementary slackness.]
A \emph{partial assignment} is a couple $(\sigma,S)$ where $S \subseteq
X$ and $\sigma: S \to Y$ is an injective map. A partial assignment
$(\sigma,S)$ and a price function $\psi \in\FYR$ satisfy \emph{$\eps$-complementary
  slackness} if for every $x$ in $S$ the following inequality holds:
\begin{equation}
c(x,\sigma(x)) + \psi(\sigma(x)) \leq \min_{y \in Y} \left[c(x,y) +
  \psi(y)\right] + \eps.  \tag{CS$_\eps$}
\label{CS}
\end{equation}
\end{definition}
In the economic interpretation, a partial assignment
$\sigma:S\subseteq X\to Y$ satisfies \eqref{CS} with respect to prices
$\psi\in\FYR$ if every customer $x\in S$ is assigned to a house
$\sigma(x)$ which is ``nearly optimal'', i.e. is within $\eps$ of
minimizing $c(x,\cdot)+\psi(\cdot)$ over $Y$.


\begin{lemma} \label{lem:CS}
If $\sigma: X\to Y$ is a bijection which satisfies \eqref{CS} together
with some $\psi\in\FYR$, then
\begin{equation} \label{eq:epsoptimal}
  \KP \leq \frac{1}{N} \sum_{x \in X} c(x,\sigma(x)) \leq \KP + \eps.
\end{equation}
\end{lemma}
\begin{proof}
The first inequality just comes from the fact $\sigma$ is a particular transport plan. For the second inequality, by summing the \eqref{CS} condition, one gets
$$
\frac{1}{N} \sum_{x \in X} c(x,\sigma(x)) +\psi(\sigma(x)) \leq \frac{1}{N} \sum_{x \in X}  \min_{y\in Y}(c(x,y) + \psi(y)) + \eps.
$$
This leads to
\begin{equation*} 
\frac{1}{N} \sum_{x \in X} c(x,\sigma(x)) \leq \Kant(\psi) + \eps \leq \DP + \eps = \KP + \eps. \qedhere
\end{equation*}
\end{proof}

\subsubsection{Bertsekas' auction algorithm} 
Bertsekas' auction algorithm maintains a partial matching $(\sigma,S)$
and prices $\psi \in \FYR$ that together satisfy
$\eps$-complementary slackness. At the end of the execution, $\sigma$
is a bijection, and $(\sigma,\psi)$ satisfy the $\eps$-CS condition.

\begin{algorithm}[h!]
\caption{Bertsekas' auction algorithm \label{algo:auction-original}}
\begin{algorithmic}
  \Function{Auction}{$c, \eps, \psi = 0$}
  \State{$S\gets \emptyset$}\Comment{All points are unassigned}
  \While{$\exists x \in X \setminus S$}
  \State{$y_0 \gets \arg\min_{y\in Y}  c(x,y) + \psi(y)  $}
  \State{$y_1 \gets \arg\min_{y\in Y\setminus \{y_0\}}  c(x,y) + \psi(y)  $}
   \State{$\psi(y_0) \gets \psi(y_0) + (c(x,y_1) + \psi(y_1)) -  (c(x,y_0) + \psi(y_0))+ \eps$}
  \If{$\exists x'\in X\hbox{ s.t. } \sigma(x') = y_0$}   \Comment{$y_0$ is ``stolen'' from $x'$}
  \State{$S \gets S \setminus \{x'\}$}
  \EndIf
  \State{$S \gets S\cup \{x\}, \quad \sigma(x) \gets y_0$}
  \EndWhile
  \State{\textbf{return} $\sigma,\psi$}
  \EndFunction
  \end{algorithmic}
\end{algorithm}
\begin{remark}[Non-convergence when $\eps=0$]\label{rem:no-convergence}
Consider for instance $X = \{x_1,x_2,x_3\}$, $Y = \{y_1,y_2,y_3\}$, $c(x,y) = \nr{x - y}$ and 
$$ y_1=(0,1), \quad y_2=(0,-1),\quad y_3=(10,0),$$
$$ x_1=(-1,0),\quad x_2 = (-2, 0), \quad x_3 = (-3,0).$$ The points
$x_1$, $x_2$ and $x_3$ are equidistant to $y_1$ and $y_2$ and ``far''
from $y_3$. Implementing auction's algorithm with $\eps=0$ then leads
to an infinite loop. Indeed, at every steps, the customers
$x_1,x_2,x_3$ pick one of the houses $y_1$ or $y_2$, but do not raise
the prices, as the second best house is equally interesting. This
``bidding war'' goes on forever.
\end{remark}

\begin{remark}[Lower bound on the number of steps] Consider the same setting as before, but with
  $\eps>0$. At the beginning of the algorithm, the customers $x_1$,
  $x_2$ and $x_3$ pick alternatively $y_1$ or $y_2$. As long as $y_3$
  has never been selected, the difference of prices between $y_1$ and
  $y_2$ is either $0$ or $\eps$, so that the bid is always $\eps$ or
  $2\eps$.  After $n$ iterations, the price of the houses $y_1,y_2$ is
  at most equal to $2n \eps$. This means that the third house $y_3$
  will never be chosen until $2n\eps > \min_{i} \nr{x_i - y_3} :=
  C$. As a consequence, the number of iterations is at least $C/(2\eps)$.
\end{remark}

The lower bound in the previous remark has a matching upper bound.

\begin{theorem} \label{th:steps-auction}
  If one starts the auction algorithm  with $\psi = 0$, then
  \begin{itemize}
  \item the number of steps in the auction algorithm is at most $N
    (C/\eps+1)$, where $C := \max_{X\times Y} c(x,y)$.
  \item the number of operations is at most $N^2 (C/\eps+1)$.
  \end{itemize}
  Moreover, the bijection $\sigma$ and the prices $\psi$ returned by the
  algorithm satisfy \eqref{CS}, so that in particular $\sigma$ is $\eps$-optimal \eqref{eq:epsoptimal}.
\end{theorem}

\begin{remark} Note that the computational complexity of this algorithm
  is very \emph{bad}.  Indeed, if $\eps = 10^{-k}$ and if $C=1$, the
  number of steps in the worst-case complexity is $10^k N^2$. It would
  be highly desirable to replace the factor $1/\eps$ by
  $\log(1/\eps)$. In the next paragraph, we see how this can be
  achieved using a scaling technique.
\end{remark}

The proof of Theorem~\ref{th:steps-auction} relies on the following
lemma, whose proof is straightforward.

\begin{lemma} \label{lem:auction:raise}
Over the course of the auction algorithm,
\begin{itemize}
\item[(i)] the set of selected ``houses'' $\sigma(S)$ is increasing
           w.r.t inclusion;
\item[(ii)] $(\psi, \sigma)$ always satisfy the $\eps$-complementary slackness
  condition ;
\item[(iii)] the price increments are by at least $\eps$.
\end{itemize}
\end{lemma}

\begin{proof}[Proof of theorem~\ref{th:steps-auction}]
  Suppose that after $i$ steps the algorithm hasn't stopped. Then,
  there exists a point $y_0$ in $Y$ that does not belong to $\sigma(S)$,
  i.e whose price hasn't increased since the beginning of the
  algorithm, i.e. $\psi(y_0) = 0$.  
 
 Suppose now that there exists $y_1$ whose price has been raised more than $n> C/\epsilon+1$. Then, by Lemma~\ref{lem:auction:raise}.(iii), one has for every $x \in X$ 
$$
\psi_i(y_0)+c(x,y_0) = c(x,y_0) \leq C < n \epsilon -\epsilon \leq \psi_i(y_1) -\epsilon\leq \psi_i(y_1) + c(x,y_1) -\epsilon
$$
This contradicts the fact that  $y_1$ was chosen at a former step. From this, we deduce that there is
  no point in $Y$ whose price has been raised $n$ times with $n >C/\epsilon+1$.
With at most $C/\eps+1$ price rise
  for each of the $N$ objects, and every step costing $N$ (finding the
  minimum among $N$) we deduce the desired bound.
\end{proof}

\subsubsection{Auction algorithm with $\eps$-scaling}
Following \cite{edmonds1972theoretical}, Bertsekas and Eckstein
\cite{bertsekas1988dual} modified
Algorithm~\ref{algo:auction-original} using a scaling technique which
improves dramatically both the running time and worst-case complexity
of the algorithm. Note that similar scaling techniques have also been
applied to improve other algorithms for the assignment problem, see
e.g. \cite{goldberg1987efficient,gabow1989faster}.

The modified algorithm can be described as follows: define $\psi_0 =
0$ and recursively, let $\psi_{k+1}$ be the prices returned by
$\mathrm{Auction}(\psi_k, \eps_k)$, where $\eps_k =
\frac{C}{2^k}$. One stops when $\eps_k < \eps$, so that the number of
runs of the unscaled auction algorithm is bounded by
$\log_2(C/\eps)$. Bounding carefully the complexity of each auction
run, one gets:

\begin{algorithm}[h!]
\caption{Bertsekas-Eckstein  auction algorithm with $\eps$-scaling \label{algo:auction-scaling}}
\begin{algorithmic}
  \Function{AuctionScaling}{$c,\eta$}
  \State{$\eps \gets C,~\psi \gets 0$}
  \While{$\eps > \eta$}
  \State{$\sigma,\psi\gets \Call{Auction}{c,\eps,\psi}$}
  \State{$\eps \gets \eps/2$}
  \EndWhile
  \State{\Return $\sigma,\psi$}
  \EndFunction
  \end{algorithmic}
\end{algorithm}
\begin{theorem} The  auction algorithm with scaling constructs an $\eta$-optimal assignement
  in time $\BigO(N^3 \log(C/\eta))$. \label{th:compauction}
\end{theorem}
\begin{lemma} Consider a bijection $\sigma_0:X \to Y$, an injective map
$\sigma: S\subseteq X\to Y$, and two price vector
  $\psi_0,\psi:Y\to\Rsp$. Assume that $(\sigma_0, \psi_0)$ and
  $(\sigma,\psi)$ satisfy respectively the $\lambda$- and
  $\eps$-complementary slackness conditions, with $\eps \leq
  \lambda$. Moreover, suppose that $S \neq X$ and that $\psi_0$ and
  $\psi$ agree on the set $Y \setminus \sigma(S)$. Then,
\label{lem:scaling}
$$ \forall y \in Y, ~\psi(y) \leq \psi_0(y) + N(\lambda+\eps) $$
\end{lemma}
\begin{proof}
Consider a point $y_0$ in $Y$, and define $y_{k+1}$ as follows: (a) if
$y_k \in \sigma(S)$, let $x_k := \sigma^{-1}(y_k)$, and $y_{k+1} =
\sigma_0(y_k)$ (b) if $y_k \not\in \sigma(S)$, then stop.  The
$\eps$-complementary slackness for $(\psi, \sigma)$ at $(x_k, y_k)$
implies
\begin{equation}
c(x_k, y_k) + \psi(y_k) \leq \min_{y\in Y} c(x_k, y) + \psi(y) +
\eps 
\leq \psi(y_{k+1}) + c(x_k, y_{k+1})  + \eps.
\label{eq:cs:1}
\end{equation}
Similarly, $\lambda$-CS for $(\psi_0, \sigma_0)$ at $(x_k, y_{k+1})$ with
$y = y_k$ implies
\begin{equation}
\label{eq:cs:2}
\psi_0(y_{k+1}) +  c(x_k, y_{k+1})  \leq c(x_k, y_{k}) + \psi_0(y_{k})
+ \lambda.
\end{equation}
Summing the inequalities \eqref{eq:cs:1} and \eqref{eq:cs:2} for $k=0$
to $k = K -1$ gives
$$ \psi_0(y_K) - \psi_0(y_0) + \psi(y_0) - \psi(y_K) \leq
K\times(\lambda + \eps)$$ By assumption, the point $y_K$ does not
belong to $\sigma(S)$ and $\psi(y_K) = \psi_0(y_K)$. This gives us
$\psi(y_0) \leq \psi_0(y_0) + K (\lambda + \eps)$, and we conclude by
remarking that the path $(y_0,x_0,\hdots,y_K)$ is simple, i.e. $K
\leq N$.
\end{proof}

\begin{proof}[Proof of Theorem~\ref{th:compauction}]
Lemma~\ref{lem:scaling} implies that during the run $k+1$  of the (unscaled) auction
algorithm, the price vector never grows larger than $\psi_0 + (\eps +
\lambda) N = \psi_0 + 3 \eps N$, with $\lambda := \eps_k$ and $\eps :=\frac{1}{2}\lambda$.
Since at each step, the price grows by at least $\eps$, there are at
most $3 N^2$ steps in the run $k$. Taking into account the cost of
finding $\min_{y\in Y} c(x,y)+\psi(y)$ at each step, the computational
complexity of each auction run is therefore $\BigO(N^3)$. Since the
the number of runs is $\BigO(\log(C/\eta))$, we get the claimed
estimate.
\end{proof}

\subsubsection{Implementations of auction's algorithm}
\label{subsec:auction-implementation}
One of the most expensive phase of Auction's algorithm is the
computation of the bid. Computing the bid for a certain customer $x\in
X$ requires one to browse through all the houses $y\in Y$ in order to
determine the smallest values of $c(x,y) + \psi(y)$, $y\in Y$. The
cost of determining the bid accounts for a factor $N=\Card(Y)$ in the
computational complexity of auction's algorithm in
Theorems~\ref{th:steps-auction} and Theorem~\ref{th:compauction}. We
mention two possible ways to overcome this difficulty.

\subsubsection*{Exploiting the geometry of the cost}
The first idea is to exploit the geometry of the space in order to
reduce the cost of finding the minimum of $c(x,y)+\psi(y)$, $y\in Y$,
which accounts for a cost of $N$ in the complexity analysis of
Theorem~\ref{th:compauction}.  The computation of this minimimum is
similar to the \emph{nearest neighbor problem} in computational
geometry, and nearest neighbors can sometimes be found in $\log(N)$
time, after some preprocessing. For instance, in the case of $c(x,y)
= \nr{x-y}^2$ on $\Rsp^d$, and for $\psi\geq0$ one can rewrite
  $$ c(x,y) + \psi(y) = \nr{x-y}^2 + (\sqrt{\psi(y)} - 0)^2 =
\nr{(x,0) - (y, \sqrt{\psi(y)})}^2, $$ thus showing that finding the
smallest value of $c(x,y)+\psi(y)$ over $Y$ amounts to finding the
closest point to $(x,0)$ in the set $\{(y,\sqrt{\psi(y)}) \mid y\in
Y\}\subseteq \Rsp^{d+1}$. This idea and variants thereof leads to
practical and theoretical improvements, both for auction's algorithm
and for other algorithms for the assignment problem. We refer to
\cite{kerber2017geometry,agarwal2014approximation} and references
therein.

\subsubsection*{Exploiting the graph structure of solutions} When the cost
satisfies the Twist conditition \eqref{eq:twist} on $\Rsp^d$ and the
source measure is absolutely continuous, Theorem~\ref{th:brenier}
guarantees that the solution to the Kantorovich's problem is
concentrated on a \emph{graph}, i.e. $\dim(\spt(\gamma)) = d$ while a
priori, the dimension of $\spt(\gamma) \subseteq \Rsp^{2d}$ could be
as high as $2d$. It is natural, in view of the stability of the
optimal transport plans (Theorem~\ref{th:stab}), to hope that this
feature remains true at the discrete level, meaning that one expects
that the support of the discrete solution concentrates on a lower
dimensional graph $G$. One could then try to use this phenomenom to
prune the search space, i.e. taking the minimum in $c(x,y) + \psi(y)$
not over the whole space but over points $y$ such that $(x,y)$ lie
``close'' to $G$. In practice, $G$ is unknown but can estimated in a
coarse-to-fine way. This idea or variants thereof has been used as a
heuristic in several works
\cite{merigot2011multiscale,oberman2015efficient,benamou2015iterative},
and has been analyzed more precisely by Bernhard Schmitzer
\cite{schmitzer2016sparse,schmitzer2019stabilized}.

\subsection{Discrete optimal transport via entropic regularization}
\label{sec:dot-entropic}
We now turn to another method to construct approximate solutions to
optimal transport problems between probability measures on two finite
sets $X$ and $Y$. Here, the measures are not supposed uniform any
more, and we set
$$ \mu = \sum_{x\in X} \mu_x \delta_x \qquad \nu = \sum_{y \in Y}
\nu_y \delta_y.$$ For simplicity, we assume throughout that all the
points in $X$ and $Y$ carry some mass, that is $\min(\min_{x \in X}
\mu_x, \min_{y\in Y}\nu_y) > 0.$ As before, we conflate a transport plan
$\gamma\in \Gamma(\mu,\nu)$ with its density $(\gamma_{x,y})_{(x,y)\in
  X\times Y}$.

\subsubsection{Entropic regularization problem}

We start from the primal formulation of the optimal transport problem,
but instead of imposing the non-negativity constraints
$\gamma_{x,y}\geq 0$, we add a term to the transport cost, which
penalizes (minus) the entropy of the transport plan and acts as a
barrier for the non-negativity constraint:
\begin{equation}\label{eq:entropy}
  \begin{aligned}
    &\qquad H(\gamma) = \sum_{x\in X,y\in Y} h(\gamma_{x,y}), \\
    &\hbox{ where }
    h(t)= \begin{cases}
      t (\log(t) - 1) &\mbox{if } t>0\\
0  &\mbox{if } t=0\\
+\infty &\mbox{if } t\leq 0\\
\end{cases}
  \end{aligned}
\end{equation}
The regularized problem is the following minimization problem:
\begin{equation}\label{eq:regularizationPb}
  \begin{aligned}
    &\qquad \KPeta :=\min_{\gamma \in \overline{\Gamma}(\mu,\nu)} \sca{c}{\gamma}  + \eta H(\gamma),\\
    &\hbox{ where }  \overline{\Gamma}(\mu,\nu)=\left\{ \gamma=(\gamma_{x,y}) \mid\
\sum_{y\in Y}\gamma_{x,y}=\mu_x,\
\sum_{x\in X} \gamma_{x,y} = \nu_y
\right\}.
\end{aligned}
\end{equation}

\begin{theorem}\label{th:entropic:existence}
  The problem $\KPeta$ has a unique solution $\gamma$, which belongs
  to $\Gamma(\mu,\nu)$. Moreover, if $\min_{x\in X} \mu_x > 0$ and
  $\min_{y\in Y} \mu_y > 0$, then
$$ \forall (x,y) \in X\times Y,~ \gamma_{x,y} >0.
$$ 
\end{theorem}

\begin{lemma}\label{lemma:stronglyconvex}$H:\gamma \in (\Rsp_+^*)^{X\times Y} \mapsto \sum_{x,y}  h(\gamma_{x,y})$ is $1$-strongly convex. 
\end{lemma}
\begin{proof}
From $h''(t)=1/t$, one sees that the Hessian $\D^2 H(\gamma)$ is
diagonal with diagonal coefficients $1/\gamma_{x,y}\geq 1$ since
$\gamma_{x,y}\in ]0,1]$.
\end{proof}

\begin{proof}
 The regularized problem $\KPeta$ amounts to minimizing a continuous and coercive function over a closed
  convex set, thus showing existence. 
Let us denote by
  $\gamma^*$ a solution of $\KPeta$. Then, $\gamma^*$ has a finite
  entropy, so that it satisfies the constraint $\gamma^*_{x,y} \geq
  0$. This implies that $\gamma^*$ is a transport map between $\mu$
  and $\nu$.  We now prove by contradiction that the set $Z
  :=\{(x,y)\mid \gamma^*_{x,y}=0\}$ is empty. For this purpose, we
  define a new transport map $\gamma^\eps \in \Gamma(\mu,\nu)$ by $
  \gamma^\eps = (1-\eps)\gamma^*+ \eps \mu \otimes \nu$, and we give
  an upper bound on the energy of $\gamma^\eps$.  We first observe
  that by convexity of $h: r\mapsto r(\log r - 1)$, one has
  $$ h(\gamma^\eps_{x,y}) \leq (1-\eps) h(\gamma^*_{x,y}) + \eps
  h(\mu_x \nu_y) \leq h(\gamma^*_{x,y}) + O(\eps). $$We consider some
  $(x,y)\in Z$. Introducing $C =
  \min_{x,y} \mu_x\nu_y$, which is strictly positive by assumption, we
  have
\begin{align*}
  h(\gamma^\eps_{x,y}) = h(\eps \mu_x \nu_y) &= \mu_x\nu_y\eps(\log\eps + \log(\mu_x\mu_y)) - \mu_x\nu_y\eps \\
  &\leq C \eps \log\eps + O(\eps),
\end{align*}
Summing the two previous estimates over $Z$ and $(X\times Y)\setminus
Z$, and setting $n=\Card(Z)$, we get
$$ H(\gamma^\eps) \leq H(\gamma^*) + Cn \eps\log\eps + O(\eps). $$
Since in addition we have by linearity
$\sca{c}{\gamma^\eps}  \leq \sca{c}{\gamma^*} + O(\eps),$ we get 
$$ \sca{c}{\gamma^*} + H(\gamma^*) \leq \sca{c}{\gamma^\eps} +
H(\gamma^\eps) \leq \sca{c}{\gamma^*} + H(\gamma^*) + Cn \eps\log\eps
+ O(\eps),$$ where the lower bound comes from the optimality of
$\gamma^*$. Thus, $Cn \eps\log\eps + O(\eps)\geq 0$, which is possible
if and only if $n=\Card(Z)$ vanishes, implying the strict positivity of $\gamma^*$.

By continuity of the function minimized in $\KPeta$, the set of solutions $\gamma^*$ is closed and therefore included in $[\delta,+\infty)^{X \times Y}$ for some $\delta >0$. Therefore, by Lemma~\ref{lemma:stronglyconvex}, the regularized problem $\KPeta$ amounts to
minimizing a coercive and strictly convex function over a closed convex set, thus showing uniqueness of the solution.  
\end{proof}

\subsubsection{Dual formulation}
We start by deriving (formally) the dual problem and  first introduce the Lagragian of $\KPeta$
\begin{equation}
  \begin{aligned}
  L(\gamma,\varphi,\psi) := 
  \sum_{x,y} \gamma_{x,y} c(x,y) +\eta h(\gamma_{x,y}) &+
  \sum_{x\in X} \phi(x) \left(\mu_x - \sum_{y\in Y} \gamma_{x,y}\right)\\
  &+ \sum_{y\in Y} \psi(y) \left(\sum_{y\in Y} \gamma_{x,y} - \nu_y\right),
  \end{aligned}
\end{equation}
where $\phi:X\to\Rsp$ and $\psi:Y\to\Rsp$ are the Lagrange multipliers. Then,
$$
\KPeta =  \min_{\gamma} \sup_{\phi,\psi} L(\gamma,\varphi,\psi).
$$ As always, the dual problem is obtained by inverting the infimum
and the supremum. We also simplify slightly the expressions:
\begin{align}
  \sup_{\phi,\psi} \min_{\gamma}   L(\gamma,\varphi,\psi)
  = \sup_{\phi,\psi} \min_{\gamma}     &\sum_{x,y} \gamma_{x,y} (c(x,y) + \psi(y) - \phi(x) + \eta(\log(\gamma_{x,y}) - 1)) \notag \\
  \label{eq:dual-reg}
  &+ \sum_{x \in X}\phi(x)\mu_x - \sum_{y\in Y}\psi(y)\nu_y.
\end{align}
Taking the derivative with respect to $\gamma_{x,y}$, we find that for
a given $\phi,\psi$, the optimal $\gamma$ must satisfy:
\begin{equation}
  \label{eq:rel-phipsi-gamma}
  \begin{aligned}
 &c(x,y) + \psi(y) - \phi(x) + \eta \log(\gamma_{x,y}) = 0\\
 & \hbox{ i.e. } \gamma_{x,y} = e^{\frac{1}{\eta}(\phi(x) - \psi(y) - c(x,y))}
  \end{aligned}
\end{equation}
Putting these values in the Equation \eqref{eq:dual-reg} gives the following definition:
\begin{definition}[Dual regularized problem] The dual of the regularized optimal transport problem is defined by 
\begin{equation}\label{eq:dual-eta}
\DPeta = \sup_{\phi,\psi} \Kant^\eta(\phi,\psi)
\end{equation}
where
\begin{equation}\label{eq:Kant-eta}
\Kant^\eta(\phi,\psi) := -\sum_{(x,y)\in X\times Y} \eta e^{\frac{1}{\eta}(\phi(x) - \psi(y) - c(x,y))} + \sum_{x \in X}\phi(x)\mu_x - \sum_{y\in Y}\psi(y)\nu_y.
\end{equation}
\end{definition}
We can now state the strong duality result
\begin{theorem}[Strong duality]\label{thm:reg-duality}
Strong duality holds and the maximum in the dual problem is reached, i.e.  there exist $\phi \in\Rsp^X$ an $\psi \in \FYR$ such that 
$$
\KPeta = \DPeta = \Kant^\eta(\phi,\psi).
$$
\end{theorem}
\begin{corollary}\label{coro:reg-dual-recovery}
If $\phi,\psi$ is the solution to the dual problem $\DPeta$,  then the solution $\gamma$ of $\KPeta$  is given by
$$
\gamma_{x,y} = e^{\frac{\phi(x) - \psi(y) - c(x,y)}{\eta}}.
$$
\end{corollary}

Corollary~\ref{coro:reg-dual-recovery} is a direct consequence of the
relation \eqref{eq:rel-phipsi-gamma}.  This holds because, unlike the
original linear programming formulation of optimal transport, the
regularized problem $\KPeta$ is smooth and strictly convex.
\begin{proof}[Proof of Theorem~\ref{thm:reg-duality}] Weak duality $\KPeta \geq \DPeta$ always hold.
  To prove the strong duality, we denote by $\gamma^*$ the solution to
  $\KPeta$, and we note that by Theorem~\ref{th:entropic:existence},
  $\gamma^*_{xy} > 0$ for all $(x,y)\in X\times Y$. This implies that
  the optimized functional $\gamma\mapsto \sca{c}{\gamma} + \eta
  H(\gamma)$ is $\Class^1$ in a neighborhood of $\gamma^*$. Thus,
  there exists Lagrange multipliers for the equality constrained
  problem, i.e. $\tilde{\phi}\in \Rsp^X$ and $\tilde{\psi} \in\FYR$ such that
  $$ \nabla_\gamma L(\gamma^*,\tilde{\phi},\tilde{\psi}) = 0. $$ Since
  the function $L(\cdot,\tilde{\phi},\tilde{\psi})$ is convex, this
  implies that $\gamma^* = \argmin_\gamma
  L(\gamma,\tilde{\phi},\tilde{\psi})$. Hence
$$
\DPeta = \sup_{\phi,\psi} \min_\gamma L(\gamma,\phi,\psi)
\geq \min_\gamma L(\gamma,\tilde{\phi},\tilde{\psi}) 
= L(\gamma^*,\tilde{\phi},\tilde{\psi}) 
= \KPeta.
$$ The last equality follows from the fact that $\gamma^*$ satisfies
the constraints and is a solution to $\KPeta$. Thus $\DPeta = \KPeta$.
\end{proof}

\subsubsection{Regularized $c$-transform}
A natural way to maximize $\Kant^\eta(\phi,\psi)$ is to maximize
alternatively in $\phi$ and $\psi$. In the case of entropy-regularized
optimal transport, each of the partial maximization problems
($\max_\phi \Kant^\eta(\phi,\psi)$ and $\max_\psi
\Kant^\eta(\phi,\psi)$) have explicit solutions, which are connected
to the notion of $c$-transform in (non-regularized) optimal transport:

\begin{proposition} The following holds
  \begin{enumi}
  \item Given $\psi \in\FYR$, the maximizer of
    $\Kant^\eta(\cdot,\psi)$ is attained at a unique point in  $\Rsp^X$, denoted $\psi^{c,\eta}$, and defined by
    \begin{equation}
      \psi^{c,\eta}(x) = \eta\log(\mu_x) -  \eta\log\left(\sum_{y \in Y} e^{\frac{1}{\eta}(-c(x,y) - \psi(y))}\right). \label{eq:phi-eta}
    \end{equation}
  \item Given, $\phi \in \Rsp^X$, the maximizer of
    $\Kant^\eta(\phi,\cdot)$ is attained at a unique point in  $\FYR$, denoted $\phi^{\bar{c},\eta}$,
    and defined by
    \begin{equation}  \phi^{\bar{c},\eta}(y) = -\eta\log(\mu_y) + \eta\log\left(\sum_{x \in X} e^{\frac{1}{\eta}(- c(x,y) + \phi(x))}\right). \label{eq:psi-eta}
    \end{equation}
  \end{enumi}
\end{proposition}

\begin{proof}
  To prove (i), consider $\phi\in\Rsp^X$ the maximizer of
  $\Kant^\eta(\cdot,\psi)$. Taking the derivative of $\Kant^\eta$ with
  respect to the variable $\phi(x)$ gives us
  $$ \mu_x = e^{\frac{\phi(x)}{\eta}} \sum_{y \in Y}
  e^{-\frac{1}{\eta}(\psi(y) + c(x,y))}, $$ implying the desired
  formula. The second formula is proven similarly.
\end{proof}

\begin{definition}[Regularized $c$-transform]
  Given $\psi\in\Rsp^Y$, we will call the function $\psi^{c,\eta}$
  defined by \eqref{eq:phi-eta} its \emph{regularized
    $c$-transform}. Similarly, given $\phi\in\Rsp^X$, we call the
  function $\phi^{\bar{c},\eta}$   defined by \eqref{eq:psi-eta} its \emph{regularized
    $\bar{c}$-transform}
\end{definition}

\begin{remark}[Relation to the $c$-transform]
  As the notation indicates, $\psi^{c,\eta}$ is related to the
  $c$-transform used in optimal transport
  (Def.~\ref{def:ctransform}). Indeed, when $\eta$ tends to zero, one
  has
  $$\begin{aligned} \lim_{\eta\to 0} \psi^{c,\eta}(x) &= \lim_{\eta\to
      0} \eta\left(\log(\mu_x) - \log\left(\sum_{y \in Y}
    e^{\frac{1}{\eta}(-c(x,y) - \psi(y))}\right)\right)\\ &=
    \min_{y\in Y} c(x,y)+\psi(y) = \psi^{c}(x).
  \end{aligned}$$
  This explains the choice of notation: $\psi^{c,\eta}$ is a smoothed
  version of the $c$-transform introduced in
  Definition~\ref{def:ctransform}. 
\end{remark}

The following two properties are very similar to some properties
holding for the standard $c$-transform. In the following, we denote
$\nr{\cdot}_{o,\infty}$ the pseudo-norm of uniform convergence up to
addition of a constant:
$$ \nr{f}_{o,\infty} = \inf_{a \in \Rsp}\nr{f+a}_\infty
= \frac{1}{2}(\sup f - \inf f). $$ This pseudo-norm will be very
useful to state convergence results for Sinkhorn-Knopp's algorithm for
solving the regularized optimal transport problem.

\begin{proposition} \label{prop:ceta-transform} Let $\psi,\bar{\psi}\in\Rsp^Y$. Then,
  \begin{itemize}
  \item[(i)] for $a\in\Rsp$, $(\psi+a)^{c,\eta} = \psi^{c,\eta} + a$.
  \item[(ii)] $\nr{\psi^{c,\eta}}_{o,\infty} \leq \eta \nr{\log(\nu)}_{o,\infty} + \nr{c}_{o,\infty}$,
  \item[(iii)]  $\nr{\psi^{c,\eta} - \bar{\psi}^{c,\eta}}_{o,\infty} \leq \nr{\psi - \bar{\psi}}_{o,\infty}$.
  \end{itemize}
  Similar properties hold for the map $\phi\in\Rsp^X\mapsto \phi^{\bar{c},\eta}$.
\end{proposition}

\begin{proof}  (ii) Using the formula \eqref{eq:phi-eta}, and $c(x,y) - c(x',y) \leq \sup c - \inf c$,
  $$
  \begin{aligned}
    &\psi^{c,\eta}(x) - \psi^{c,\eta}(x') \\
    &\quad= \eta(\log(\mu_x) - \log(\mu_{x'})) \\
    \quad &\qquad+ \eta\left(\log\left(\sum_{y \in Y} e^{\frac{1}{\eta}(-c(x',y) - \psi(y))}\right) - \log\left(\sum_{y \in Y} e^{\frac{1}{\eta}(-c(x,y) - \psi(y))}\right)\right). \\
     &\quad\leq \eta (\sup \log(\mu) - \inf \log(\mu)) + \sup c - \inf c,
  \end{aligned}$$
  implying the first inequality.

(iii) If we show that $\nr{\psi^{c,\eta} - \bar{\psi}^{c,\eta}}_\infty
  \leq \nr{\psi - \bar{\psi}}_\infty$, the same inequality with
  $\nr{\cdot}_{o,\infty}$ will follow easily using (i). Using $\psi(y)
  \leq \bar{\psi}(y) + \nr{\psi - \bar{\psi}}_\infty, $ we have
  \begin{align*}
    &\psi^{c,\eta}(x) - \bar{\psi}^{c,\eta}(x) \\&\qquad=
    -  \eta\log\left(\sum_{y \in Y} e^{\frac{1}{\eta}(-c(x,y) - \psi(y))}\right) + \eta\log\left(\sum_{y \in Y} e^{\frac{1}{\eta}(-c(x,y) - \bar{\psi}(y))}\right)  \\
    &\qquad\leq \nr{\psi - \bar{\psi}}_{\infty} \qedhere
  \end{align*}
\end{proof}

\subsubsection{Regularized Kantorovitch functional}
As in standard optimal transport (see
\S\ref{subsec:Kantorovich-functional}) and following Cuturi and Peyré
\cite{cuturi2018semidual}, we can express the regularized dual
maximization problem \eqref{eq:dual-eta} using only the variable $\psi
\in \FYR$.

\begin{definition}[Regularized Kantorovitch functional] The regularized Kantorovitch functional $\Kant^\eta:\FYR\to \Rsp$ is defined by
  \begin{equation} \label{eq:Kantsmooth} \Kant^{\eta}(\psi)
      = \max_{\phi\in\Rsp^X} \Kant^\eta(\phi,\psi)
      = \sca{\psi^{c,\eta}}{\mu}
      - \sca{\psi}{\nu} \end{equation} \end{definition} Since
      $\psi^{c,\eta}$ has a closed-form expression, the functional
      $\Kant^\eta$ can be computed explicitely. This explicit
      expression is a special feature of the choice of the entropy as
      the regularization. In the next formula,
      $H(\mu)=\sum_{x\in X} \mu_x \log(\mu_x)$:
$$ \Kant^\eta(\psi) =  -\eta \sum_{x\in X}  \mu_x 
\left(\log\sum_{y \in Y} e^{\frac{-c(x,y) - \psi(y)}{\eta}} \right)
+  \eta H(\mu )- \sum_{y\in Y}\psi(y)\nu_y,$$
\begin{remark}
  Note the similarity between the formula for Kantorovich functional
  derived from regularized transport \eqref{eq:Kantsmooth} and the
  formula for the Kantorovich functional without
  regularization \eqref{eq:Kant}. Note also that $\Kant^\eta$ is also
  invariant by addition of a constant, namely $\Kant^\eta(\psi
  + \lambda \one_Y) = \Kant^\eta(\psi)$ for any $\lambda \in \Rsp$ and
  $\one_Y = \sum_{y\in Y}\one_y$ the constant function equal to one.
\end{remark}

In order to express the gradient and the Hessian of $\Kant^\eta$, we
introduce the notion of \emph{smoothed laguerre cells}.
\begin{definition}[Smoothed Laguerre cells]
Given $\psi\in\Rsp^Y$, we define
\begin{equation}\label{eq:RLag}
\RLag^\eta_{y}(\psi) = \frac{e^{- \frac{c(\cdot,y) + \psi(y)}{\eta}}}{\sum_{z\in Y} e^{- \frac{c(\cdot,z) + \psi(z)}{\eta}}}.
\end{equation}
\end{definition}
  Unlike the standard Laguerre cell $\Lag_y(\psi)$ defined
  in \eqref{eq:Lag0}, which is a set, $\RLag^\eta_y(\psi)$ is
  a \emph{function}.  The family $(\RLag^\eta_{y}(\psi))_{y\in Y}$ is
  a partition of unity, meaning that the sum over $y$ of
  $\RLag^\eta_{y}(\psi)$ equals one.  One can loosely think of the
  regularized Laguerre cells as smoothed indicator functions of the
  (standard) Laguerre cells. In particular,
  $$ \lim_{\eta\to 0} \RLag_{y}^{\eta}(\psi)(x) = \begin{cases} 0 & \hbox{ if } x\not\in \Lag_y(\psi) \\
1 & \hbox{ if } x \in \SLag_y(\psi),
  \end{cases} $$
  where $\SLag_y(\psi)$ is the strict Laguerre cell introduced in \eqref{eq:SLag0}.
  We also introduce the two quantities
 \begin{align*}
 &G_y^\eta(\psi) = \sca{\RLag^\eta_{y}(\psi)}{\mu} \\
 &G_{yz}^\eta(\psi) = \begin{cases}
    \frac{1}{\eta} \sca{\RLag^\eta_{y}(\psi)\RLag^\eta_{z}(\psi)}{\mu}
    & \hbox{ if } z\neq y \\
   -\sum_{z\neq y} G_{yz}^\eta(\psi) & \hbox{ if } z=y.
  \end{cases}
   \end{align*}
   Informally, $G_y^\eta(\psi)$ measures the quantity of mass of $\mu$ within the regularized
  Laguerre cell $\RLag^\eta_{y}(\psi)$.
\begin{theorem}\label{thm:hessian_regKant} \mbox{}\\
$\bullet$ The regularized Kantorovitch functional $\Kant^\eta$ is
$\Class^\infty$, concave, with first and second-order partial
derivatives given by
\begin{align*}
&\forall y\in Y,~ \frac{\partial \Kant^\eta}{\partial \one_y}(\psi)
= G_y^\eta(\psi) - \nu_y, \\
&\forall y\neq z\in Y,~\frac{\partial^2 \Kant^\eta}{\partial \one_z\partial\one_y}(\psi)
= G_{yz}^\eta(\psi).
\end{align*}
$\bullet$ The function $\Kant^\eta$  is strictly concave on the orthogonal of the set of constant functions. More precisely,  for every $\psi \in \FYR$ one has
$$
\forall v\in \FYR\ s.t. \sum_{y\in Y} v(y) =0,~
\D^2\Kant^\eta(\psi)(v,v)<0.
$$
$\bullet$ If $\psi$ is a maximizer in $\DPeta$, then the solution to $\KPeta$ is
given by
$$\gamma = \sum_{x,y} \gamma_{x,y} \delta_{(x,y)}, \hbox{ with } \gamma_{x,y}
= \RLag^\eta_{y}(\psi)(x)\mu_x.$$
\end{theorem}
\begin{proof}
For every $y\in Y$, the derivative is given by
$$
\frac{\partial \Kant^\eta}{\partial \one_y}(\psi)
=  \sum_{x\in X}  \mu_x 
\frac{ e^{\frac{- \psi(y) - c(x,y)}{\eta}}}{\sum_{z \in Y} e^{\frac{-c(x,z) - \psi(z)}{\eta}}}
-\nu_y
= G_y^\eta(\psi) - \nu_y.
$$
The second order derivative is given for $z\neq y$ by
$$
\frac{\partial^2 \Kant^\eta}{\partial \one_z\partial\one_y}(\psi)
=  \sum_{x\in X}  \mu_x e^{\frac{- \psi(y) - c(x,y)}{\eta}}
\frac{\frac{1}{\eta} e^{\frac{- \psi(z) - c(x,z)}{\eta}}}{\left(\sum_{z \in Y} e^{\frac{-c(x,z) - \psi(z)}{\eta}}\right)^2}
= G_{yz}^\eta(\psi) .
$$
The relation
$$
\sum_{y\in Y} \frac{\partial \Kant^\eta}{\partial \one_y}(\psi) = 1
$$
gives the desired formula for the second order derivatives when $z=y$.
The hessian of $\Kant^n$ is therefore symmetric with dominant
diagonal, with negative diagonal coefficients. This implies that the
Hessian is negative, hence that $\Kant^\eta$ is concave.  Let us now
show that $\ker H = \Rsp \one_Y$, where
$H=\D^2\Kant^\eta(\psi)$. Consider $v \in \ker H$ and let $y_0\in Y$
be the point where $v$ attains its maximum.  Then using $Hv=0$, and in
particular $(H v)(y_0) =0$, one has
$$\begin{aligned}
  0 &= \left(\sum_{y \neq y_0} H_{y, y_0} v(y)\right) + H_{y_0, y_0} v(y_0)\\
   &= \sum_{y \neq y_0} H_{y, y_0} (v(y) -  v(y_0)).
\end{aligned}
$$ This follows from $H_{y_0, y_0} = -\sum_{y \neq y_0}
   H_{y,y_0}$. Since for every $y\neq y_0$, one has $ H_{y,y_0} > 0 $
   and $v({y_0})-v(y) \geq 0 $, this implies that
   $v(y)=v(y_0)$. Therefore $\ker H \subseteq \Rsp \one_Y$. The reverse
   inclusion is obvious and therefore $\Kant^\eta$ is strictly concave
   on the orthogonal of the set of constant functions.

To prove the last claim we note that if $\psi$ maximizes
$\Kant^\eta(\cdot)$, then $(\psi^{c,\eta},\psi)$ maximizes
$\Kant^\eta(\cdot,\cdot)$. By Corollary~\ref{coro:reg-dual-recovery},
the optimal transport map $\gamma$ is 
\begin{equation*}
\gamma_{x,y}= e^{\frac{\phi(x)-\psi(y)-c(x,y)}{\eta}} =  \frac{e^{-\frac{\psi(y)+c(x,y)}{\eta} }\mu_x}{\sum_{z\in Y} e^{-\frac{\psi(z)+c(x,z)}{\eta}}}
= \RLag^\eta_{y}(\psi)(x)  \mu_x. \qedhere
\end{equation*}
\end{proof}

\subsubsection{Sinkhorn-Knopp as block coordinate ascent}
We present here the Sinkhorn-Knopp algorithm that consists in
computing a maximizer to the dual problem $\DPeta$ by optimizing the
functional $\Kant^\eta$ alternatively in $\phi$ and $\psi$.  The
iterations are defined by
\begin{equation}\label{eq:sinkhorn}
  \begin{cases}
  \phi^{(k+1)} = (\psi^{(k)})^{c,\eta}\\
  \psi^{(k+1)} = (\phi^{(k+1)})^{\bar{c},\eta},
  \end{cases}
\end{equation}
or equivalently $\psi^{(k+1)} = S(\psi^{(k)})$ where 
\begin{equation}\label{eq:sinkhorn-S}
  S(\psi) = (\psi^{c,\eta})^{\bar{c},\eta}.
\end{equation}
\begin{remark}[Relation to matrix factorization]
  This algorithm is in fact a reformulation, using a logarithmic
  change of variable, of Sinkhorn-Knopp's algorithm
  \cite{sinkhorn1967concerning} for finding a factorization of
  non-negative matrices introduced by Sinkhorn
  \cite{sinkhorn1964relationship}. We therefore refer to the
  iterations \eqref{eq:sinkhorn}--\eqref{eq:sinkhorn-S} as
  Sinkhorn-Knopp's algorithm.
\end{remark}

\paragraph{Correctness}
We first show the correctness of Sinkhorn--Knopp's algorithm, using a
simple expression for $S(\psi)$ which can be found in an article of
Robert Berman~\cite{berman2017}.

\begin{proposition}[Correctness of Sinkhorn-Knopp]
Let $\psi \in \FYR$ be a potential. The following assertions are equivalent:
\begin{itemize}
\item[(i)] $\psi$ is a fixed point of $S$;
\item[(ii)] 
for every $y\in Y$ $\sca{\mu}{\RLag^\eta_y(\psi)} = \nu_y$;
\item[(iii)] 
  $\psi$ is a maximizer of the regularized Kantorovich function $\Kant^\eta$
  \end{itemize}
\end{proposition}
This proposition follows at once from the next lemma, and from the
computation of $\nabla \Kant^\eta$ in
Theorem~\ref{thm:hessian_regKant}.
\begin{lemma} $
\frac{S(\psi)(y)-\psi(y)}{\eta} 
= - \log(\nu_y) +  \log \sca{\mu}{\RLag^\eta_{y}(\psi)}.
$
  \end{lemma} 
\begin{proof}
A calculation shows that
$$
\begin{aligned}
S(\psi)(y) &=  - \eta \left(  \log(\nu_y) -  \log\sum_{x \in X} e^{\frac{-c(x,y)  +\eta \left( \log(\mu_x) -  \log\sum_{z \in Z} e^{\frac{- c(z,y) - \psi(z)}{\eta}} \right)}{\eta}} \right)\\
&= - \eta \left(  \log(\nu_y) -  \log\sum_{x \in X}   \mu_x \frac{e^{\frac{-c(x,y)}{\eta}}}{\sum_{z\in Y} e^{\frac{- c(z,y) - \psi(z)}{\eta}}} \right)\\
&= - \eta \left(  \log(\nu_y) -  \log e^{\frac{\psi(y)}{\eta}}\sum_{x \in X}   \mu_x \RLag^\eta_{y}(\psi)(x) \right)\\
&= - \eta \left(  \log(\nu_y) -  \log e^{\frac{\psi(y)}{\eta}} \sca{\mu}{\RLag^\eta_{y}(\psi)} \right),
\end{aligned}
$$
which implies the equation.
\end{proof}

\paragraph{Convergence}
In order to prove convergence, we need to strengthen the $1$-Lipschitz
estimation from Proposition~\ref{prop:ceta-transform}. This allows to
apply Picard's fixed point theorem to get the contraction of the
Sinkhorn-Knopp iteration \eqref{eq:sinkhorn-S}.  The proof we present
in this chapter has been first introduced in course notes of Vialard
\cite{vialard:hal-02303456}.

\begin{theorem}[Convergence of Sinkhorn, \cite{vialard:hal-02303456}]
  \label{th:cv-sinkhorn-vialard}
  The map $S$ is a contraction for $\nr{\cdot}_{o,\infty}$. More precisely,
  $$ \nr{S(\psi^0) - S(\psi^1)}_{o,\infty} \leq \left(1 -
  e^{-2\frac{\nr{c}_{o,\infty}}{\eta}}\right) \nr{\psi^0
  - \psi^1}_{o,\infty}.$$ In particular, the iterates
  $(\phi^{(k)},\psi^{(k)})$ of Sinkhorn-Knopp's
  algorithm \eqref{eq:sinkhorn} converge with linear rate to the
  unique (up to constant) maximizer the regularized dual
  problem \eqref{eq:dual-eta}
\end{theorem}

\begin{remark}[Other convergence proofs]
The convergence of Sinkhorn-Knopp's algorithm is usually proven
(e.g. in \cite{sinkhorn1967concerning}) using a theorem of Birkhoff
\cite{birkhoff1946tres}. We refer to the recent book by Peyré and
Cuturi \cite{peyre2019computational} for this point of view.  Other
convergence proofs exist, see for instance Berman~\cite{berman2017}
(in the continuous case), and Altschuler, Weed and Rigolet
\cite{altschuler2017near}.
\end{remark}

\begin{remark}[Convergence speed]
  This theorem shows that the Sinkhorn-Knopp algorithm converges with
  linear speed, but the contraction constant has a bad dependency in
  $\eta$.  Denoting $C = \nr{c}_{o,\infty}$, to get an error of $\eps$
  one needs
   $$ (1 - e^{-2C/\eta})^k \leq \eps $$
  $$ \hbox{ i.e. } k \gtrsim e^{2C/\eta}  \log(1/\eps), $$
  where the second inequality holds for small values of $\eta$.
  This bad dependency in $\eta$ seems to be a practical obstacle to
  choosing a very small smoothing parameter. This calls for scaling
  techniques, as for the auction's algorithm, and was considered by
  Schmitzer~\cite{schmitzer2016sparse,schmitzer2019stabilized}.
\end{remark}

\begin{remark}[Implementation]
  The numerical implementation of Sinkhorn-Knopp's algorithm is more
  complicated than it seems:
  \begin{itemize}
    \item In a naive implementation, the computation of the smoothed
      $c$-transforms \eqref{eq:phi-eta}--\eqref{eq:psi-eta} has a cost
      proportional to $\Card(X)\Card(Y)$. This can be alleviated for
      instance when $X=Y$ are grids and when the cost is a
      $\nr{\cdot}_p$ norm, using fast convolution techniques (see
      e.g. \cite{solomon2015convolutional} or \cite[Remark
        4.17]{peyre2019computational}), or when the cost is the
      squared geodesic distance on a Riemannian manifold
      \cite{crane2013geodesics,solomon2015convolutional}.
    \item The convergence speed can be slow when the supports of the
      data $X,Y$ are ``far'' from each other, and when $\eta$ is
      small. This difficulty is cirvumvented using the $\eta$-scaling
      techniques mentioned above, often combined with multi-scale
      (coarse-to-fine) strategies, studied in this context by Benamou,
      Carlier and Nenna \cite{benamou2016numerical} and Schmitzer
      \cite{schmitzer2016sparse}.
    \item Finally, some numerical difficulties (divisions by zero) can
      occur when $\eta$ is small and the potential $\psi$ is far from
      the solution.
  \end{itemize}
  The book of Cuturi and Peyré present these difficulties in more
  details and explain how to circumvent them
  \cite{peyre2019computational}. In addition to the works already
  cited, we refer to the PhD work of Feydy
  \cite{charlier2017efficient,feydy2019fast}, and especially to the
  implementation of regularized optimal transport in the library
  GeomLoss\footnote{\url{https://www.kernel-operations.io/geomloss/}}.
\end{remark}

In order to prove this theorem, we will make use of the following
elementary lemma, giving an upper bound on the $\LL^1$ distance
between two Gibbs kernels $e^{u_i}/Z_i$ for $i\in\{0,1\}$ as a
function of $\nr{u_1 - u_0}_{o,\infty}$.

\begin{lemma} \label{lemma:gibbs}
  Let $u_0,u_1$ be two functions on $Y$ and denote $g_i = e^{u_i}/Z_i$
  where $Z_i = \sum_{y\in Y}e^{u_i(y)}$. Then,
  $$\sum_{y\in Y} \abs{g_1(y) - g_0(y)} \leq 2(1- e^{-2\nr{u_0 - u_1}_{o,\infty}}). $$
\end{lemma}

\begin{proof}
  Note that by definition the Gibbs kernel $g_i$ does not change if a
  constant is added to $u_i$, so that we can assume that
  $$\eps := \nr{u_0 - u_1}_{o,\infty} = \nr{u_0 - u_1}_{\infty}. $$
  Using the inequality 
  $u_0 - \eps \leq u_1 \leq u_0 + \eps,$
  one easily shows that 
  $$ e^{-2\eps} \frac{e^{u_0}}{Z_0} \leq \frac{u_1}{Z_1} \leq e^{2\eps} \frac{e^{u_0}}{Z_0},$$
  thus implying $e^{-2\eps} g_0 \leq g_1 \leq e^{2\eps} g_0$.
  This gives
  $$ \begin{cases}
    (e^{-2\eps}-1) g_0 \leq g_1 - g_0\\
    (e^{-2\eps}-1) g_1 \leq g_0 - g_1,
    \end{cases}
  $$
  thus implying
  $$ \abs{g_1 - g_0} \leq (1 - e^{-2\eps}) \max(g_0, g_1) \leq (1 -
  e^{-2\eps})(g_0+g_1). $$ Summing this inequality over $Y$ and using
  $\sum_Y g_i = 1$, we obtain the desired inequality.
\end{proof}

\begin{proof}[Proof of Theorem~\ref{th:cv-sinkhorn-vialard}]
   Consider $\psi_0,\psi_1\in\Rsp^Y$ and $\psi_t = \psi_0 + t v$ with
   $v=\psi_1-\psi_0$. Without loss of generality, we assume that the
   functions $\psi_0,\psi_1$ are translated by a constant so that
   $\nr{\psi_0 - \psi_1}_\infty = \nr{\psi_0 - \psi_1}_{o,\infty}$. We
   will first give an upper bound on $\nr{\psi_1^{c,\eta} -
     \psi_0^{c,\eta}}_{o,\infty}$, and to do that we will give an upper bound on
  $$ A(x,x') = (\psi_1^{c,\eta}(x) - \psi_0^{c,\eta}(x)) -
   (\psi_1^{c,\eta}(x') - \psi_0^{c,\eta}(x')) $$ which is independent
   of $x,x' \in X$. For this purpose, we introduce
  $$B(t,x,x') = - \eta\log\left(\sum_{y \in Y} e^{\frac{1}{\eta}(-c(x,y) - \psi_t(y))}\right)
+ \eta \log\left(\sum_{y \in Y} e^{\frac{1}{\eta}(-c(x',y) - \psi_t(y))}\right),$$
and
  $$ g_{x,t}(y) = \frac{e^{\frac{1}{\eta}(-c(x,y) -
    \psi_t(y))}}{\sum_{z \in Y} e^{\frac{1}{\eta}(-c(x,z) -
    \psi_t(z))}}.$$ Then, recalling the definition of
$\psi_t^{c,\eta}$ in Eq.~\eqref{eq:psi-eta},
  \begin{align*}
    A(x,x') &= B(1,x,x') - B(0,x,x')\\
    &= \int_0^1 \partial_t B(t,x,x')\dd t 
    = \int_{0}^1 \sca{v}{g_{x,t} - g_{x',t}}_{\Rsp^Y} \dd t, \\
    &\leq \nr{v}_\infty \int_{0}^1 \sum_{y \in Y} \abs{g_{x,t}(y) - g_{x',t}(y)} \dd t 
  \end{align*}
 Then, by the previous lemma (Lemma~\ref{lemma:gibbs}) and setting
 $u_{x,t}(y) = -\frac{1}{\eta}(c(x,y) + \psi_t(y))$, so that $g_{x,t}
 = e^{u_{x,t}}/Z_{x,t}$ with $Z_{x,t} =\sum_Y g_{x,t}$, we obtain
 $$
 \begin{aligned}
   A(x,x') &\leq 2\nr{v}_\infty \int_{0}^1 1-e^{-2\nr{u_{x,t} - u_{x',t}}_{o,\infty}}\dd t\\
   &\leq 2 \nr{\psi_1 - \psi_0}_{o,\infty} (1-e^{-2\nr{u_{x,t} - u_{x',t}}_{o,\infty}})
   \end{aligned} $$
  In addition,
  $$ \nr{u_{x,t} - u_{x',t}}_{o,\infty} \leq \frac{\nr{c}_{o,\infty}}{\eta}. $$
  We therefore obtain
  $$
  \begin{aligned}
    \nr{\psi^{c,\eta}_1 - \psi^{c,\eta}_0}_{o,\infty}
    &\leq \frac{1}{2} \sup_{x,x'\in X} A(x,x') \\
    &\leq  \nr{\psi_1 - \psi_0}_{o,\infty} \left(1-e^{-2\frac{\nr{c}_{o,\infty}}{\eta}}\right).
  \end{aligned}$$
  We conclude the proof of the contraction inequality by remarking
  that the map $\phi \mapsto \phi^{\bar{c},\eta}$ is $1$-Lipschitz,
  thanks to Proposition~\ref{prop:ceta-transform}.(iii).
\end{proof}

\section{Semi-discrete optimal transport}
In this part, we consider the semi-discrete optimal transport problem,
where the source measure is a probability density and the target is a
finitely supported measure.  We start by introducing in
Section~\ref{sec:semidiscrete} the framework of semi-discrete optimal
transport, showing its connection with the notion of Laguerre tessellation
in discrete geometry. We study in detail the regularity of
Kantorovitch functional $\Kant$ in this setting, in connection with algorithms for solving the semi-discrete optimal transport problem:
\begin{itemize}
\item In Section~\ref{sec:olikerprussner}, we show convergence of the \emph{coordinate-wise
increment} algorithm introduced by Oliker and Pr\"{u}ssner using  the
Lipschiz-continuity of the gradient $\nabla \Kant$.
\item In Section~\ref{sec:newton} a damped Newton method and prove its
convergence from a $C^1$-regularity and monotonicity property
of the gradient $\nabla \Kant$.
\item Finally, we consider in Section~\ref{sec:rot_semi-discrete} the
entropic regularization of the semi-discrete optimal transport problem
and its relation to unregularized semi-discrete optimal transport.

\end{itemize}


\subsection{Formulation of semi-discrete optimal transport}
\label{sec:semidiscrete}
Our working assumptions for this section are the following:
\begin{itemize} 
\item $\Omega_X,\Omega_Y$ are two open subsets of $\Rsp^d$. The cost
function $c\in \Class^1(\Omega_X\times \Omega_Y)$ satisfies the twist condition
introduced in Definition~\ref{def:twist}.
\item the source measure $\rho$ is absolutely continuous with respect
  to the Lebesgue measure on $\Omega_X$ and its support is contained
  in a compact subset $X$ of $\Omega_X$. \textbf{When writing $\rho\in
    \Probac(X)$ we always mean that $\rho$ belongs to
    $\Probac(\Omega_X)$ with $\spt(\rho)\subseteq X$.}
\item the target space $Y$ is finite so that $\nu \in \Prob(Y)$ can be
  written under the form $\nu = \sum_{y\in Y} \nu_y\delta_y$. For
  simplicity, we assume that $\min_y \nu_y > 0$.
\end{itemize}
Note that by an abuse of notation, we will often conflate $\rho$ with
its density with respect to the Lebesgue measure. 

\subsubsection{Laguerre tessellation}\label{sec:laguerre}
In the semi-discrete setting, the dual of Kantorovich's relaxation can
be conveniently phrased using the notion of Laguerre tessellation, a
variant of the Voronoi tesselation. This connection was already known
and used in the 1980s and 1990s, see for instance Cullen--Purser
\cite{cullen1984extended}, Aurenhammer--Hoffman--Aronov
\cite{aurenhammer1998minkowski} or
Gangbo-McCann~\cite{gangbo1996geometry}, Caffarelli--Kochengin--Oliker
\cite{caffarelli1999problem}. Large-scale numerical implementations
are more recent, starting in the 2010s, see
e.g. \cite{merigot2011multiscale,de2012blue,gu2013variational,kitagawa2014iterative,
  levy2015numerical,kitagawa2016newton,deleo2017numerical,hartmann2017semi,
  degournay2018discrete,degournay2019differentiation}.
%
%
To explain the connection, we start with an economic metaphor. Assume
that the probability density $\rho$ describes the population
distribution over a large city $\Omega_X$, and that the finite set $Y$
describes the location of bakeries in the city. Customers living at a
location $x$ in $\Omega_X$ try to minimize the walking cost $c(x,y)$,
resulting in a decomposition of the space called a Voronoi
tessellation. The number of customers received by a bakery $y\in Y$ is
equal to the integral of $\rho$ over its Voronoi cell,
$$\Vor_y := \{ x\in \Omega_X \mid\forall z \in Y, c(x,y) \leq c(x,z)
\}. $$ If the price of bread is given by a function $\psi: Y\to\Rsp$,
customers living at location $x$ in $X$ make a compromise between
walking cost and price by minimizing the sum $c(x,y) + \psi(y)$. This
leads to the notion of Laguerre tessellation.

\begin{definition}[Laguerre tessellation]
The Laguerre tessellation associated to a set of prices $\psi:
Y\to \Rsp$ is a decomposition of the space into \emph{Laguerre cells}
defined by
\begin{equation}
  \Lag_y(\psi) := \{ x\in \Omega_X \mid \forall z \in Y, c(x,y)
  + \psi(y) \leq c(x,z) + \psi(z) \}.
\end{equation}
More generally, for any distinct $y_1,\hdots,y_\ell \in Y$, we denote the common facet between
the Laguerre cells $\Lag_{y_i}(\psi)$ by
\begin{equation} \label{eq:Lagyz}
\Lag_{y_1\hdots y_\ell}(\psi) = \bigcap_{1\leq i\leq \ell} \Lag_{y_i}(\psi)
\end{equation}
\end{definition}

\begin{figure}
    \begin{center}
      \includegraphics[width=\textwidth]{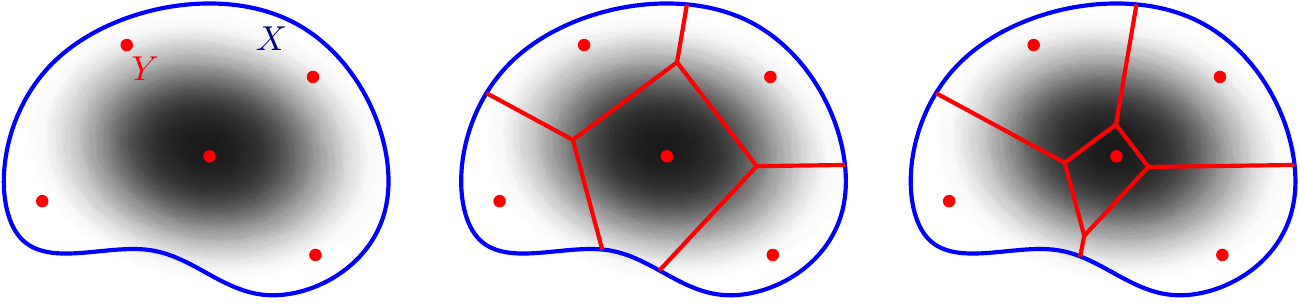}
      \caption{(Left) The domain $X$ (with boundary in blue) is
        endowed with a probability density pictured in grayscale
        representing the density of population in a city. The set $Y$
        (in red) represents the location of bakeries. Here,
        $X,Y\subseteq \Rsp^2$ and $c(x,y) = |x-y|^2$ (Middle) The
        Voronoi tessellation induced by the bakeries (Right) The
        Laguerre tessellation: the price of bread the bakery near the 
        center of $X$ is higher than at the other bakeries,
        effectively shrinking its Laguerre cell.\label{fig:laguerre}}
      \end{center}
\end{figure}

We will also frequently consider the following hypersurfaces/halfspaces:
\begin{equation}
\begin{aligned} \label{eq:Hyz}
H_{yz}(\psi) &= \{x \in \Omega_X \mid c(x,y) + \psi(y) = c(x,z) + \psi(z) \},\\
H^\leq_{yz}(\psi) &= \{x \in \Omega_X \mid c(x,y) + \psi(y) \leq c(x,z) + \psi(z) \},
\end{aligned}
\end{equation}
which are defined so that
\begin{equation}\label{eq:LagH}
\begin{aligned}
&\forall y\in Y, \Lag_y(\psi) = \cap_{z \in Y\setminus\{y\}} H_{yz}^\leq(\psi)\\
&\forall y,z\in Y, \Lag_{yz}(\psi) \subseteq H_{yz}(\psi)
\end{aligned}
\end{equation}

\begin{remark} For the quadratic cost  $c(x,y) = \nr{x-y}^2$, one has
\begin{align*}
&c(x,y) + \psi(y) \leq c(x,z) + \psi(z) \\
&\Longleftrightarrow \sca{x}{z-y} \leq \frac{1}{2}(\psi(z) + \nr{z}^2 - (\psi(y) - \nr{y}^2)),
\end{align*}
which easily implies that the Laguerre cells are convex polyhedra
intersected with the domain $\Omega_X$. Introducing $\tilde{\psi}(z) = \frac{1}{2}(\psi(z) + \nr{z}^2)$, one has 
$$ \Lag_{y}(\psi) = \{ x\in \Omega_X \mid \forall z\in Y,
~\sca{x}{z-y} \leq \tilde{\psi}(z) - \tilde{\psi}(y) \}.$$ 
As a direct
consequence, the intersection of two distinct Laguerre cells is
contained in an hyperplane and is therefore Lebesgue negligible. If in
addition $\psi \equiv 0$, then the Laguerre tessellation coincides
with the Voronoi tessellation.  The shape of the Voronoi and Laguerre
tessellations is depicted in Figure~\ref{fig:laguerre}.
\end{remark}

The following proposition shows that Laguerre tessellations can be used to
build optimal transport maps.

\begin{proposition}   \label{prop:OT}
Under the twist condition (Def.~\ref{def:twist}), the intersection of
two distinct Laguerre cells $\Lag_y(\psi)\cap \Lag_z(\psi)$ ($y\neq
z$) is Lebesgue-negligible, and the map
$$ T_\psi: x\in \Omega_X\mapsto \arg\min_{y\in Y} c(x,y)+ \psi(y)$$ is
well-defined Lebesgue almost-everywhere. In addition for any
$\psi\in\Rsp^Y$ and any $\rho \in \Probac(X)$, $T_\psi$ is an optimal
transport map for the cost $c$ between $\rho$ and the measure
\begin{equation}\label{eq:nupsi}
\nu_\psi := T_{\psi \#} \rho = \sum_{y\in Y} \rho(\Lag_y(\psi)) \delta_y.
\end{equation}
\end{proposition}

\begin{proof}  By Equation \eqref{eq:LagH}, one has 
$$ \Lag_y(\psi) \cap \Lag_z(\psi) = \Lag_{yz}(\psi) \subseteq  f^{-1}(\{0\}),$$ where we
have set $f(x) = c(x,y) - c(x,z) + \psi(y) - \psi(z)$. By the twist
condition, $\nabla f(x) \neq 0$ for all $x\in \Omega_X$, implying that
the set $f^{-1}(\{0\})$ is a $(d-1)$-submanifold and is in
particular Lebesgue-negligible. This easily implies that $T_\psi$ is
well-defined.

Let us now prove optimality of $T_\psi$ in the optimal transport
problem between $\rho$ and $T_{\psi\#}\rho$. By definition of
$T_\psi$, one has
$$ \forall (x,y) \in X,~~ c(x,T_\psi(x)) + \psi(T_\psi(x)) \leq c(x,y)
  + \psi(y).$$ Let $\gamma$ be a transport plan between $\rho$ and
  $\nu_\psi$. Integrating the above inequality with respect to $\gamma$
  gives
$$ \int_{X} (c(x,T_\psi(x))
  + \psi(T_\psi(x)))  \rho(x)\dd x \leq \int_{X\times Y} (c(x,y)
  + \psi(y)) \dd \gamma(x,y),$$ where we have used $\Pi_{X\#} \gamma
  = \rho$ to simplify the left-hand side. Since $\nu = \Pi_{Y\#}\gamma
  = T_{\psi\#} \rho$,  applying change of variable formulas we get
$$ \int_{X \times Y} \psi(y) \dd\gamma(x,y) = \int_Y \psi(y) \dd \nu
  = \int_X \psi(T_\psi(x))  \rho(x) \dd x $$ Substracting this equality
  from the inequality above shows that the map $T_\psi$ is optimal in
  the optimal transport :
\begin{equation*} \int_X c(x,T_\psi(x)) \rho(x) \dd  x \leq
\int_{X\times Y} c(x,y)\dd \gamma(x,y)\qedhere
  \end{equation*}
\end{proof}

\subsubsection{Monge-Ampère equation}
Proposition~\ref{prop:OT} implies that any map $T_\psi$ induced by a Laguerre
tessellation of the domain solves the optimal transport between $\rho$
and the image measure $\nu_\psi = T_{\psi\#}\rho$. From now on, we
will denote
\begin{equation} \label{eq:G}
\begin{aligned}
G_y: \Rsp^Y\to \Rsp, \psi \mapsto \rho(\Lag_y(\psi)) \\
G: \Rsp^Y\to \Rsp^Y, \psi \mapsto (y\mapsto G_y(\psi)).
\end{aligned}
\end{equation}
In the bakery analogy, the function $G_y(\psi)$ measures the number of
customers for the bakery $y$ given a family of prices $\psi\in\Rsp^Y$,
and $G:\Rsp^Y\to\Rsp^Y$ maps a family of prices to a distribution of
customers among the bakeries. By \eqref{eq:nupsi}, one has
$$\nu_\psi = \sum_{y\in Y}
G_y(\psi) \delta_y.$$
For simplicity, we consider $\Prob(Y)$ as a
subset of $\Rsp^Y$, conflating a probability measure $\nu = \sum_{y\in
Y} \nu_y \delta_y$ with the function $\nu: y\mapsto \nu_y$. Then,
$T_\psi$ is an optimal transport map between $\rho$ and $\nu$ iff
$T_{\psi\#}\rho = \nu$ iff
\begin{equation}
\begin{aligned}
\label{eq:DMA}
G(\psi) = \nu.
\end{aligned}
\end{equation}
In other words, we have transformed the optimal transport problem into
a finite-dimensional non-linear system of equations \eqref{eq:DMA}.

\begin{remark}[Relation to subdifferential and Monge-Ampère equation]
Assume that $X=\Omega_X = \Rsp^d$ and that $c(x,y) =
-\sca{x}{y}$. Then,
$$
\begin{aligned}
  \Lag_y(\psi)
  &= \{ x \in \Omega_X\mid \forall z\in Y,~-\sca{x}{y} + \psi(y) \leq -\sca{x}{z} + \psi(z) \} \\
  &= \{ x \in \Omega_X \mid \forall z\in Y,~ \psi(z) \geq \sca{x}{z-y} + \psi(y)  \}
\end{aligned}
$$ Denote $\hat{\psi}$ the convex envelope of $\psi$, which can be
defined using the double Legendre-Fenchel transform by
$$\phi(x) = \max_{y\in Y} \sca{x}{y} - \psi(y), $$
$$\hat{\psi}(z) = \max_{x\in X} \sca{x}{y} - \phi(x). $$ Then the
Laguerre cells defined above agree with the subdifferential of
$\hat{\psi}$, i.e. $\Lag_y(\psi) =  \partial \hat{\psi}(y).$
Moreover, in the context of Monge-Ampère equations,  the (infinite) measure
$$ \sum_{y\in Y} \Leb(\Lag_{y}(\psi)) \delta_y = \sum_{y\in Y}
\Leb(\partial \hat{\psi}(y)) \delta_y$$ is called the
\emph{Monge-Ampère measure} of the function $\hat{\psi}$
\cite{gutierrez2001monge}. Semi-discrete techniques can also be
applied to the numerical resolution of Monge-Ampère equations (with
e.g. Dirichlet boundary conditions). We refer the reader to the
pioneering work of Oliker-Prussner \cite{oliker1989numerical} and  to the
survey by Neilan, Salgado and Zhang \cite{neilan2019monge}.
\end{remark}

\begin{remark}[Lack of uniqueness]
The solution $\psi$ to $G(\psi) = \nu$ is never unique, because $G$ is
invariant under addition of a constant (see
Proposition~\ref{prop:G}-\ref{prop:G:constant}). When $\spt(\rho)$ is
disconnected there might also exist two solutions $\psi^0,\psi^1$ to
$G(\psi^i) = \nu$ such that $\psi^0 - \psi^1$ is not constant. Take $X
= [-1,1]$, $Y = \{-1,1\}$, choose $c(x,y)= (x-y)^2$ and
$$ \rho = \one_{[-1,-\frac{1}{2}] \cup [\frac{1}{2},1]} \qquad \nu
= \frac{1}{2}(\delta_{-1} + \delta_{1}), $$ A computation shows that
if $\abs{\psi(1) - \psi(-1)} \leq
2$, then $G(\psi) = \nu$.
\end{remark}

The existence of solutions to \eqref{eq:DMA} and the algorithms that
one can use to solve this system depend crucially on the properties of
the function $G$. In the next proposition, we denote $(\one_y)_{y\in
Y}$ the canonical basis of $\Rsp^Y$, i.e. $\one_y(z) = 1$ if $y=z$ and
$0$ if not. We also denote $\one_Y$ the constant function on $Y$ equal to $1$. 
On $\Rsp^Y$ we consider two norms:
$$ \nr{\psi} = \sqrt{\sum_{y\in Y} \abs{\psi(y)}^2} \quad \hbox{ and }
\quad \nr{\psi}_\infty = \max_{y\in Y} \abs{\psi(y)}.$$ We will often
use the notation $R$, which measures the oscillation of the cost
function:
\begin{equation} \label{eq:def-R}
R:= \max_{X\times Y} c - \min_{X\times Y} c,
\end{equation}

\begin{proposition}\label{prop:G}
Assume  $c$ is twisted (Def.~\ref{def:twist}) and $\rho\in\Probac(X)$. Then,
\begin{enumi}
\item \label{prop:G:monotone1}
$\forall y\in Y,\forall t\geq 0$, $G_y(\psi + t \one_y) \leq
G_y(\psi)$,
\item \label{prop:G:monotone2}
$\forall y\neq z\in Y,\forall t\geq 0$, $G_y(\psi + t \one_z) \geq G_y(\psi)$,
\item \label{prop:G:constant} $\forall \psi\in\Rsp^Y,\forall t\in \Rsp,~G(\psi + t\one_Y) = G(\psi)$,
\item \label{prop:G:proba} $\forall \psi\in\Rsp^Y, G(\psi)\in \Prob(Y)$,
\item \label{prop:G:empty}
if $\psi\in \Rsp^Y$ is such that $G_{y_0}(\psi) > 0$, then $\psi(y_0) \leq \min_Y \psi + R$,
\item \label{prop:G:allnonempty}
if $\psi\in \Rsp^Y$ is such that $G_{y}(\psi) > 0$ for every $y\in Y$, then\\ $\max_Y\psi - \min_Y \psi \leq R$,
\item $G$ is continuous,
\end{enumi}
where $R=\max_{X\times Y} c - \min_{X\times Y} c$. 
\end{proposition}


\begin{proof}
The properties \ref{prop:G:monotone1},
\ref{prop:G:monotone2}, \ref{prop:G:constant} are straightforward consequences of the definition of Laguerre cells. 
Property \ref{prop:G:proba} is a consequence of Proposition~\ref{prop:OT} 
and of the assumption $\rho\in\Probac(X)$. To
prove \ref{prop:G:empty}, take $\psi$ such that $G_{y_0}(\psi) > 0$,
implying in particular that the Laguerre cell $\Lag_{y_0}(\psi)$ is
non-empty and contains a point $x\in X$. Then, by definition of the
cell one has for all $y\in Y\setminus\{y_0\}$, $c(x,y_0)
+ \psi(y_0) \leq c(x,y) + \psi(y),$ thus showing that $\psi(y_0)
\leq \min_{Y} \psi + R$. Point \ref{prop:G:empty} is a consequence of Point~\ref{prop:G:allnonempty}.

It remains to establish that each of the maps $G_y$ is
continuous. For this purpose, we consider a sequence 
$(\psi_n)_{n\in \Nsp} \in \Rsp^Y$ converging to some
$\psi_\infty \in \Rsp^Y$. We first note that as in the proof of
Proposition~\ref{prop:OT}, the set
$$ S = \{x \in X \mid \exists y\neq z \in
Y \hbox{ s.t. } c(x,y) + \psi(y) = c(x,y)+\psi(z) \}.$$
is
Lebesgue-negligible and therefore also $\rho$-negligible. Defining
$\chi = \one_{\Lag_y(\psi)}$ and $\chi_n =
\one_{\Lag_y(\psi_n)}$, 
$$ G_y(\psi_n) = \int \chi_n \dd\rho, \hbox{ and } G(\psi) = \int \chi
\dd\rho.$$ To prove that $\lim_{n\to +\infty} G_y(\psi_n) = G_y(\psi)$
it suffices to establish that $\chi_n$ converges to $\chi$ on
$X\setminus S$, which is straightforward (because the inequalities
defining the set $X\setminus S$ are strict), and to apply Lebesgue's
dominated convergence theorem.
\end{proof}

From these properties of $G$, we can deduce the existence of a
solution to the equation $G(\psi) = \nu$. The strategy used to prove
this proposition is borrowed from \cite{caffarelli2008weak} and is
also reminiscent of Perron's method to prove existence to Monge-Ampère
equations, see e.g. \cite{gutierrez2001monge}.

\begin{corollary} \label{coro:existence-DMA} Let $G: \Rsp^Y\to \Rsp^Y$ satisfying (i)-- (vi) in
Proposition~\ref{prop:G} and let $\nu\in \Prob(Y)$. Then, there exists
$\psi \in \Rsp^Y$ such that $G(\psi) = \nu$.
\end{corollary}
\begin{proof}
Fix some $y_0\in Y$ such that $\nu_{y_0} \neq 0$, and consider the set
$$ K = \{ \psi \in \Rsp^Y \mid \psi(y_0) = 0 \hbox{ and } \forall y\in
Y\setminus\{y_0\}, G_{y_0}(\psi) \leq \nu_y \hbox{ and } \psi(y) \leq
R \}, $$ where $R$ is defined as in \eqref{eq:def-R}. Given $\psi\in K$, one has
$$ G_{y_0}(\psi) = 1 - \sum_{y\neq
y_{0}} G_y(\psi) \geq \nu_{y_0} > 0,$$ 
implying
by \ref{prop:G:empty} that $\min_{y\in Y} \psi \geq \psi(y_0) -
R$.  The set $K$ is therefore bounded and closed (by continuity of the
functions $G_y$) and therefore compact. We consider $\psi^*$ a
minimizer over the set $K$ of the function
$J(\psi) =  \sum_{y\in Y} \psi_y$. Assume that
$G_y(\psi^*) < \nu_y$ for some $y\in Y\setminus\{y_0\}$. Then, by
continuity of $G_y$, there exists some $t>0$ such that $G_y(\psi^* -
t \one_y) < \nu_y$.  Then, by property \ref{prop:G:monotone2}, we have
$$ \forall z\neq y, G_z(\psi^* - t\one_y) \leq
G_z(\psi^*) \leq \nu_z,$$ thus showing that $\psi^* - t\one_y \in
K$. Since $J(\psi^* - t \one_y) = J(\psi^*) - t < J(\psi^*)$, we get a
contradiction. We thus have showed that $\forall y\in Y\setminus\{y_0\},$
$G_y(\psi^*) = \nu_y$, and using \ref{prop:G:proba} and $\nu\in \Prob(Y)$ we
obtain
$$ G_{y_0}(\psi^*) = 1 - \sum_{y\in Y\setminus \{y_0\}} G_y(\psi^*) = 
1 - \sum_{y\in Y\setminus \{y_0\}} \nu_y = 
\nu_{y_0},$$
so that $G(\psi^*) = \nu$ and $\psi^*$ is a solution
to \eqref{eq:DMA}.
\end{proof}


\subsubsection{Kantorovich's functional}

We now show that Equation \eqref{eq:DMA} is the optimality condition
of the Kantorovitch functional, and can thus be recast as a smooth  unconstrained optimization problem. 
We recall that 
$$
\KP = \max_{\psi\in \Rsp^Y} \Kant(\psi),
$$
where $\Kant$ is the Kantorovich functional given by 
\begin{align*}
\Kant(\psi) &=  \int_{X} \psi^c \dd \mu - \int_Y \psi\dd\nu\\
&= \sum_{y\in Y} \int_{\Lag_y(\psi)} (c(x,y) + \psi(y)) \dd \rho(x) - \sum_{y\in Y}\psi(y)\nu_y.\\
\end{align*}

\begin{theorem}[Aurenhammer, Hoffman, Aronov]
  \label{th:Aurenhammer} Assume that $\rho \in \Probac(X)$,
  that $c$ is twisted (Def.~\ref{def:twist}), and consider 
  $\Kant$ defined in \eqref{eq:Kant}. Then:
  \begin{itemize}
  \item $\Kant$ is concave and $\Class^1$-smooth and its gradient is 
  \begin{equation}
  \nabla \Kant(\psi) = G(\psi) - \nu 
\label{eq:GradKant}
\end{equation}
where $G$ is defined in \eqref{eq:G}.
  \item $\forall \psi\in\Rsp^Y,\forall t\in \Rsp,~\Kant(\psi + t\one_Y) = \Kant(\psi)$,
  \item $\Kant$ attains its maximum over $\Rsp^Y$, and $\nabla \Kant(\psi) = 0$ iff $\psi$ solves \eqref{eq:DMA}.
  \end{itemize}
\end{theorem}

\begin{remark}
This theorem could be deduced from the computation of directional
derivatives of $\Kant$ given in Corollary~\ref{coro:Kant-grad},
however we prefer to give a simple and self-contained proof due to
Aurenhammer, Hoffman, Aronov \cite{aurenhammer1998minkowski}.
\end{remark}




\begin{proof}[Proof of Theorem~\ref{th:Aurenhammer}] We simultaneously  show that
  the functional is concave and compute its gradient.  For any
  function $\psi$ on $Y$ and any measurable map $T: X\to Y$, one has
  $$\min_{y\in Y} (c(x,y) + \psi(y)) \leq c(x,T(x)) + \psi(T(x)),$$
  which by integration against $\rho$ gives
\begin{equation}
 \Kant(\psi) \leq \int_{X} (c(x,T(x)) + \psi(T(x)))\rho(x) \dd x - \sum_{y\in Y}\psi(y)\nu_y. 
\label{eq:AurSD}
\end{equation}
Moreover, equality holds when $T=T_\psi$. Taking another function
$\psi' \in \Rsp^Y$ and setting $T = T_{\psi'}$ in Equation~\eqref{eq:AurSD}
gives
\begin{align*}
\Kant(\psi) &\leq
\int_{X} (c(x,T_{\psi'}(x)) + \psi(T_{\psi'}(x))) \rho(x) \dd x- \sum_{y\in Y}\psi(y)\nu_y\\
&= \sum_{y\in Y} \int_{\Lag_y(\psi')} (c(x,y) + \psi(y)) \rho(x) \dd x - \sum_{y\in Y}\psi(y)\nu_y \\
&= \sum_{y\in Y} \int_{\Lag_y(\psi')} (c(x,y) + \psi'(y)) \rho(x) \dd x + \\
&\phantom{= } \sum_{y\in Y} \rho(\Lag_y(\psi')) (\psi(y) - \psi'(y)) - \sum_{y\in Y}\psi(y)\nu_y\\
&= \Kant(\psi') + \sca{G(y) - \nu}{\psi - \psi'}
\end{align*}
By definition, this shows that $G(\psi)-\nu$ belongs to the
superdifferential to $\Kant$ (Definition~\ref{def:superdifferential})
at $\psi$, i.e. $G(\psi)-\nu \in \partial^+ \Kant(\psi)$, thus proving
by Proposition~\ref{prop:superdifferential} that $\Kant$ is concave.

We now prove that $\Kant$ belongs to $\Class^1(\Rsp^Y)$. Consider
$\psi\in \Rsp^Y$ and let $(\psi_n)_{n\in\Nsp}$ be a sequence
converging to $\psi$ and such that $\nabla \Kant(\psi_n)$ exists for
every $n\in \Nsp$. Since $G(\psi_n) - \nu \in \partial^+\Kant(\psi_n)
= \{ \nabla \Kant(\psi_n) \},$ we obtain $\nabla \Kant(\psi_n) =
G(\psi_n) - \nu$. Thus, by the continuity of $G$
(Proposition~\ref{prop:G}),
$$ \lim_{n\to +\infty} \nabla \Kant(\psi_n) = \lim_{n\to +\infty}
G(\psi_n) - \nu = G(\psi) - \nu,$$ ensuring by \eqref{eq:SDcalc} that
$\partial^+ G(\psi) = \{ G(\psi) - \nu \}$, so that
$\nabla \Kant(\psi) = G(\psi) - \nu$ for all $\psi\in\Rsp^Y$.  By
continuity of $G$ we get $\Kant\in\Class^1$ as announced, and
Equation~\eqref{eq:G} holds for all $\psi\in \Rsp^Y$, so that one
trivially has $\nabla \Kant(\psi) = 0$ iff $G(\psi) = \nu$.  Finally,
we note that thanks to Corollary~\ref{coro:existence-DMA}, there
exists $\psi\in\Rsp^Y$ such that $G(\psi) = \nu$, which automatically
is a maximizer of $\Kant$ because $\Kant$ is concave and
$\nabla \Kant(\psi) = 0$.
\end{proof}

\subsection{Semi-discrete optimal transport via coordinate decrements}\label{sec:olikerprussner}
As before, we assume that $X \subseteq \Omega_X$ is compact, that
$Y \subseteq \Omega_Y$ is finite and that
$\Omega_X,\Omega_Y\subseteq \Rsp^d$ are open sets.  We recall the
notation $G_y(\psi) := \rho(\Lag_y(\psi)).$ Oliker-Prussner's
algorithm for solving $G(\psi) = \nu$ is described in
Algorithm~\ref{algo:oliker-prussner}, and bears strong resemblance
with Bertsekas' auction algorithm, in that the ``prices'' are evolved
in a monotonic way.

\begin{algorithm}
\begin{description}
  \item[Input] A tolerence parameter $\delta>0$.
  \item[Initialization] Fix some $y_0 \in Y$ once for all. Set 
  $$ \psi^{(0)}(y) := 
  \begin{cases}
    0 &\hbox{ if } y = y_0\\
    R  & \hbox{ if not}.
  \end{cases}
  $$
\item[While] $\exists y \neq y_0$ such that $G_{y}(\psi^{(k)})) \leq
  \nu_y - \frac{\delta}{N}$
    \begin{description}
\item[Step 1] Compute 
  \begin{equation}
    t_y = \min \{ t \geq 0 \mid G_y(\psi^{(k)} -  t\one_y) \geq \nu_y \}.  \label{eq:ty}
  \end{equation}
  \item[Step 2] Set $\psi^{(k+1)} = \psi^{(k)} - t\one_y$.
\end{description}
  \item[Output] A vector $\psi^{(k)}$ that satisfies $\nr{G(\psi^{(k)}) - \nu}_\infty  \leq \delta$.
\end{description}
\caption{Oliker-Prussner algorithm} \label{algo:oliker-prussner}
\end{algorithm}

This algorithm can be described in words using the bakery analogy of
Section~\ref{sec:laguerre}. We choose once and for all a bakery
$y_0\in Y$ whose price will be set to zero. Initially, the price of
bread $\psi^{(0)}$ is zero at this bakery $y_0$ and set to the
prohibitively large value $R$, defined in Equation \eqref{eq:def-R},
at any other location. This choice guarantees that the bakery $y_0$
initially gets all the customers. The prices $\psi^{(k)} \in\Rsp^y$
are then constructed iteratively by performing a sort of reverse
auction: at step $k$, start by finding some bakery $y = y^{(k)} \in
Y\setminus \{y_0\}$ which sells less bread than its production
capacity, i.e.
$$G_{y}(\psi^{(k)})\leq \nu_{y} - \frac{\delta}{N}.$$
The
price of bread at $y$ is then decreased so that the amount of
bread sold equals the production capacity of $y$, i.e. one finds
$t_y \geq 0$ such that
$$ G_{y}(\psi^{(k)} - t_y\one_{y}) = \nu_y $$
and then updates $\psi^{(k+1)} = \psi^{(k)} - t_y \one_{y}$.

\begin{remark}[Origin and extensions]
This algorithm was introduced by Oliker and Prussner, for the purpose
of solving Monge-Ampère equations with Dirichlet boundary conditions
in \cite{oliker1989numerical}. In the context of optimal transport,
the first use of Algorithm~\ref{algo:oliker-prussner} seems to be in
an article of Caffarelli, Kochengin and Oliker
\cite{caffarelli1999problem} (see also \cite{caffarelli2008weak}), in
the setting of the reflector problem, namely $c(x,y) =
-\log(1-\sca{x}{y})$ on $X=Y=\Sph^{d-1}$. Since then, the convergence
of this algorithm has been generalized to more other costs and/or more
general assumptions on the probability density $\rho$, we refer the
reader to \cite{kitagawa2014iterative,deleo2017numerical} and to
references therein.
\end{remark}

\subsubsection{$\Class^{1,1}$ estimates for Kantorovich functional}
The proof of convergence of Oliker-Prussner's algorithm relies on the
Lipschitz regularity of the map $G$ when $\rho$ is bounded, proven in
the next proposition. (Since $\nabla \Kant = G-\nu$, this proposition
also implies that Kantorovich's functional $\Kant$ has Lipschitz
gradient, improving from the $\Class^1$ estimate of
Theorem~\ref{th:Aurenhammer}.)

\begin{proposition} \label{prop:Glip}
Assume that $c\in \Class^2(\Omega_X \times \Omega_Y)$  satisfies the
twist condition, and assume also that $\rho
\in\Probac(X)\cap\LL^\infty(X)$. Then for every $y\in Y$, the map
$G_y: \Rsp^Y\to\Rsp$ defined in \eqref{eq:G} is globally Lipschitz.
\end{proposition}

\begin{remark}
The proof of this proposition comes with an estimation of the
Lipschitz constant: namely it shows  $\abs{G_y(\psi) -
G_y(\phi)} \leq L_G \nr{\phi - \psi}_\infty$ with
\begin{equation} \label{eq:def-LG}
\begin{aligned}
&L_G =  c(d) N \nr{\rho}_\infty \frac{1}{\kappa} \left(1+\frac{M}{\kappa} \diam(X)\right)\diam(X)^{d-1},\\
&\kappa = \min_{y \neq z \in Y} \min_{X} \nr{\nabla_x c(\cdot,y) - \nabla_x c(\cdot,z)}, \\
&M = \max_{y \neq z \in Y} \max_{X} \nr{\D^2_{xx} c(\cdot,y) - \D^2_{xx} c(\cdot,z)}.
\end{aligned}
\end{equation}
In the estimation of the Lipschitz constant $L_G$ \eqref{eq:def-LG},
it is possible that the term in $N$ is not tight, but the other terms
cannot be improved without adding assumptions on the cost. 
\end{remark}
\begin{example}
With $c(x,y) = \frac{1}{2}\nr{x-y}^2$, one has $\nabla_x c(x,y) = (x-y)$ and
$\D^2_{xx} c(x,y) =  \id$, so that $M=0$ and $\kappa$ is 
the minimal distance between two distinct points in $Y$:
$\kappa = \min_{y\neq z\in Y} \nr{y-z}.$
\end{example}

The proof relies on the following lemma, which allows to estimate the
variations of $G_y$ in the direction $\one_z$, $z\neq y$.
\begin{lemma} \label{lemma:Gy:pd} Let
 $c\in\Class^1(\Omega_X\times\Omega_Y)$ be a twisted cost and
 $\rho\in\Probac(X)$.  For every $y \neq z\in Y$ and
 $\psi\in \Rsp^Y$,
\begin{equation} \label{eq:Gy:pd}
G_{y}(\psi + t\one_z) - G_y(\psi) = \int_{0}^t G_{yz}(\psi + s\one_z) \dd s.
\end{equation}
where
$$ G_{yz}(\psi) = 
\int_{\Lag_{yz}(\psi)} \frac{\rho(x)}{\nr{\nabla_x c(y,x) - \nabla_x c(y,z)}} \d\Haus^{d-1}(x). $$
\end{lemma}
\begin{proof} This is a consequence of the  coarea formula, Equation~\eqref{eq:coarea2}.
In order to see this, we first note that  
$$\Lag_y(\psi+t\one_z) = E \cap H^{\leq}_{yz}(\psi) \quad \hbox{ where } E = \bigcap_{w \in Y\setminus\{y,z\}} H_{yw}^{\leq}(\psi).$$ In particular, for $t\geq 0$, setting   $c_{yz} = c(\cdot,y) - c(\cdot,z)$ and $a = \psi(z)
- \psi(y)$,
$$ \Lag_y(\psi+t\one_z) \setminus\Lag_y(\psi) = E\cap c_{yz}^{-1}((a,a+t]) $$
 Thus, by the coarea formula,
\begin{align*}
G_y(\psi+t\one_z) - G_y(\psi) &= \rho(\Lag_y(\psi + t\one_z) \setminus \Lag_y(\psi)) \\
&= \int_{E \cap c_{yz}^{-1}((a,a+t])} \rho(x) \dd \vol^d(x) \\
&= \int_0^{t} \int_{E\cap c_{yz}^{-1}(a+s)} \frac{\rho(x)}{\nr{\nabla c_{yz}(x)}} \d\Haus^{d-1}(x)\dd s
\end{align*}
One concludes by remarking that
\begin{align*}
x\in \Lag_{yz}(\psi + s\one_z) &\Longleftrightarrow  x\in E \hbox{ and }c(x,y) + \psi(y) = c(x,z) + \psi(z) + s\\
&\Longleftrightarrow x \in E \cap c_{yz}^{-1}(a + s).
\end{align*}
This establishes \eqref{eq:Gy:pd} in the case $t\geq 0$, and the case
$t\leq 0$ can be treated similarly.
\end{proof} 

The second ingredient to prove Proposition~\ref{prop:Glip} is an uniform upper
 bound on the $(d-1)$--Hausdorff measure of the level set of a
 $\Class^2$ function $f$ with non-vanishing gradient.

\begin{lemma}\label{lemma:bnd-area}
Let $X\subseteq \Omega_X \subseteq \Rsp^d$ with $\Omega_X$ open and
$X$ compact, and let $f \in \Class^2(\Omega_X)$ such that $\forall
x\in \Omega_X, \nr{\nabla f(x)}> 0$.  Then,
$$ \Haus^{d-1}(f^{-1}(0) \cap X) \leq c(d) \left(1+\frac{M}{\kappa}\diam(X)\right)\diam(X)^{d-1}. $$
where $\kappa = \min_X \nr{\nabla f}$ and $M = \max_X \nr{\D^2 f}$.
\end{lemma}


\begin{proof} By compactness, there exists a finite number of unit
vectors $u_1,\hdots u_n$ and $V_1,\hdots,V_n$ an open covering of the
unit sphere $\Sph^{d-1}$ such that if $u\in V_i$, then $\sca{u}{u_i}
\geq 3/4$, implying in particular, $\nr{u - u_i}^2 = 2 - 2\sca{u}{u_i}
\leq \frac{1}{2}$. Moreover $n$ depends only on the dimension $d$.
Let $S = f^{-1}(0)\cap X$. This set $S$ can be covered by patches
$S_i$, i.e. $S = \cup_{i} S_i$ where $$S_i = \{ x \in S \mid \nabla
f(x) \in V_i \}.$$ We will now estimate the volume of each patch $S_i$
using the coarea formula recalled in Theorem~\ref{th:coarea} of the
appendix. To apply this formula, we consider $\Pi_i: S_i \subseteq
\Rsp^d\to \{u_i\}^\perp$ the orthogonal projection onto the hyperplane
$H_i = \{u_i\}^\perp$. We need to estimate the Jacobian $J_{\Pi_i}(x)$
(see \eqref{eq:J}).
Since $\Pi_i$ is linear, we have $\D \Pi_i= \Pi_i$.
Moreover, for any tangent vector $v$ at $x\in S_i$, one has
$\sca{v}{\nabla f(x)} = 0$. Setting $u = \nabla f(x)\in V_i$, we get
\begin{align*}
\nr{\Pi_i v}^2 &= \nr{v}^2 - \sca{v}{u_i}^2 \\
&= \nr{v}^2 - \sca{v}{u_i-u}^2\\
&\geq \nr{v}^2(1-\nr{u_i - u}^2) \geq \frac{1}{2}\nr{v}^2.
\end{align*}
This directly shows that the restriction of $\D \Pi_i(x)$ to the
tangent space $T_xS_i$ at $S_i$ is injective and that its inverse is
$\frac{1}{\sqrt{2}}$-Lipschitz. This  implies that
$$J_{\Pi_i}(x) \geq
c(d) = \left(\frac{1}{\sqrt{2}}\right)^{d-1}.$$  We now apply the co-area
formula \eqref{eq:coarea2} to the manifold $M = f^{-1}(0)$,
$E=S_i\subseteq N$, $N=H_i$, $n=m=d-1$, $\Phi=\Pi_i$, and $u\equiv 1$:
\begin{align*}
\Haus^{d-1}(S_i) &= \int_{H_i} \int_{\Pi_i^{-1}(y)} \frac{1}{\J_{\Pi_i}(x)}\d\Haus^{0}(x) \dd \Haus^{d-1}(y)\\
&\leq c(d) \int_{\{u_i\}^\perp} \Card(S_i \cap (y + \Rsp u_i)) \dd \Haus^{d-1}(y).
\end{align*}
We now give an upper bound on $\Card(S_i \cap (y + \Rsp u_i))$. Let
$x\in S_i$ and use Taylor's formula to get
$$ f(x + t u_i) \geq f(x) + t \sca{\nabla f(x)}{u_i} - \frac{M}{2}
t^2 \geq \frac{3}{4} \kappa t - \frac{M}{2} t^2 $$ so that $f(x + t
u_i) > 0$ as long as $t \in (0, t^*)$ with $t^*
= \frac{3\kappa}{2M}$. One has a similar bound for negative $t$. This
directly implies that the number of intersection points between $S_i$
and $y+\Rsp u_i$ is at most $1+\diam(X)/t^*$. Since the number $n$ of directions $u_i$  only depends on the dimension $d$, we have
\begin{align*}
\Haus^{d-1}(S) &\leq \sum_{1\leq i\leq n} \Haus^{d-1}(S_i) \\
&\leq c(d) \sum_{1\leq n} \Haus^{d-1}(H_i \cap \Pi_i(X)) \left(1+\frac{M}{\kappa} \diam(X)\right) \\
&\leq c(d) \left(1+\frac{M}{\kappa}\diam(X)\right)\diam(X)^{d-1}.\qedhere
\end{align*}
\end{proof}

\begin{proof}[Proof of Proposition~\ref{prop:Glip}]
Let $y \in Y$. Applying Lemma~\ref{lemma:Gy:pd}, we have
\begin{align*}
\abs{G_y(\psi+t\one_y) - G_y(\psi)} &\leq \abs{\int_{0}^t \int_{\Lag_{yz}(\psi + s\one_y)} \frac{\rho(x)}{\nr{\nabla c_{yz}(x)}}
 \dd\Haus^{d-1}(x) \dd
 s} \\ &\leq \frac{\nr{\rho}_{\infty}}{\kappa} \max_{a \leq s \leq
 a+t} \Haus^{d-1}(c_{yz}^{-1}(s)\cap X) \abs{t},
\end{align*}
where we used the bound $\nr{\nabla c_{yz}(x)} \geq \kappa$, which
comes from the twist assumption and the inclusion
$\Lag_{yz}(\psi+s\one_z)\subseteq c_{yz}^{-1}(a+s)$ with $a= \psi(z)
- \psi(y)$, as in the proof of the previous lemma. 
Applying Lemma~\ref{lemma:bnd-area} to the function $f=c_{yz} -s$, we
get a uniform upper bound on the $(d-1)$-volume of the level set
$c_{yz}^{-1}(s)$:
$$ \Haus^{d-1}(c_{yz}^{-1}(s) \cap X) \leq
c(d)\left(1+ \frac{M}{\kappa}\diam(X)\right) \diam(X)^{d-1}, $$
which  yields
\begin{equation} \label{eq:Gylip}
\begin{aligned}
&\abs{G_y(\psi+t\one_z) - G_y(\psi)} \leq \hat{L}_G \abs{t},\\
&\hbox{  with }
 \hat{L}_G = \frac{c(d)}{\kappa}\left(1+ \frac{M}{\kappa}\diam(X)\right) \diam(X)^{d-1}\ \|\rho\|_\infty.
\end{aligned}
\end{equation}

Take $\psi,\tilde{\psi}\in\Rsp^Y$. 
Order the points in $Y$, i.e. let $Y
= \{y_1,\hdots,y_N\}$ 
and define recursively
$$
\begin{cases}
\psi^0 = \psi \\
\psi^{k+1} = \psi^k + (\tilde{\psi}(y_k)
- \psi(y_k)) \one_{y_k}
\end{cases}
$$
Then, $\psi^{N} = \tilde{\psi}$ and for $k\geq 1$, $\psi^{k+1}$ and
$\psi^k$ differ only by the value at $y_k$
Thus, applying \eqref{eq:Gylip},
\begin{align*}
\abs{G_y(\tilde{\psi}) - G_y(\psi)} &= \sum_{1\leq k\leq N} \abs{G_y(\psi^{k+1}) - G_y(\psi^k)} \\
&\leq \sum_{1\leq k\leq N} \hat{L}_G \abs{\psi(y_k) - \tilde{\psi}(y_k)} \\
&\leq L_G \nr{\psi - \tilde{\psi}}_\infty \hbox{ with } L_G = N \hat{L}_G \qedhere
\end{align*}
\end{proof}



\subsubsection{Convergence of Oliker-Prussner's algorithm}

Now that we have established the Lipschitz continuity of $G_y$, the
convergence of Algorithm~\ref{algo:oliker-prussner} follows easily, using arguments
similar to those used to establish the convergence of Auction's
algorithm.

\begin{theorem}[Oliker-Prussner] \label{th:OlikerPrussner}
Assume that the cost $c\in \Class^2(\Omega_X\times \Omega_Y)$ is twisted (Def.~\ref{def:twist}) and that
$\rho \in \Probac(X) \cap \LL^\infty(X)$. Then,
\begin{itemize}
\item  Oliker-Prussner's algorithm converges in a finite number of steps $k \leq \mathrm{C} N^3/\delta$, where  $\mathrm{C}$ is a constant that depends on $X$, $Y$, $\rho$ and $c$. 
\item Furthermore, at step $k$, one has
$$ \forall 1\leq i\leq N, \abs{G_i(\psi^{(k)})
- \nu_i} \leq \delta. $$ 
\end{itemize}
\end{theorem}
\begin{remark}[Computational complexity]
The computational complexity is actually much higher than the number
of steps of the algorithm, since:
\begin{itemize}
  \item at each iteration, one needs to compute $t_y$ (this could be
    done using for instance a binary search or more clever
    techniques).
\item each time the map $G_y$ is evaluated, one needs to compute the
  Laguerre cell $\Lag_y(\psi)$, which, if done naively, requires to
  compute the intersection of $N-1$ half-spaces
  $H^\leq_{yz}(\psi)$.
\end{itemize}
Overall, this leads to an upper bound on computational complexity of
at least $\BigO(\frac{N^4}{\delta}\log(N))$, assuming that one can
compute $\Lag_y(\psi)$ in time $N$. To the best of our knowledge,
there exists no lower bound on the number of iterations of
Algorithm~\ref{algo:oliker-prussner}, i.e. specific instances of the
problem for which one can count the number of iterations.
\end{remark}

\begin{remark}[$\delta$-Scaling]
It is tempting to perform $\delta$-scaling as in the case of Auction's
algorithm (see Algorithm~\ref{algo:auction-scaling}). In practice, one
could start with a rather large $\delta^{(0)} \in (0,1)$, to get a
first estimation of the prices using Oliker-Prussner's algorithm. Then
one would iteratively replace $\delta^{(\ell)}$ by $\delta^{(\ell+1)}
= \frac{1}{2} \delta^{(\ell)}$ and run again the algorithm starting
from the prices found at the previous iteration. Doing so, one could
hope to get rid of the $\frac{1}{\delta}$ term in the number of
iterations, and to replace it by e.g. $\log\left(\frac{1}{\delta}\right)$.
\end{remark}

\begin{proof}[Proof of Theorem~\ref{th:OlikerPrussner}]\mbox{}\\
\textbf{Step 1} \emph{(Correctness)} When
Algorithm~\ref{algo:oliker-prussner} terminates with $\psi :=
\psi^{(k)}$, one has for any $y \neq y_0$, $\rho(\Lag_y(\psi)) \leq
\nu_y$. When it stops, it also means that one has $\rho(\Lag_y(\psi))
\geq \nu_y - \frac{\delta}{N}$. Then, as desired, we get
\begin{equation*}
 \rho(\Lag_{y_0}(\psi)) = 1 - \sum_{y\neq y_0} \rho(\Lag_{y_0}(\psi)) \in [\nu_{y_0}, \nu_{y_0} + \delta]. 
\end{equation*}

\noindent\textbf{Step 2} \emph{(A priori bound on $\psi_k$)} By 
construction one has $\rho(\Lag_y(\psi^{(k)})) \leq \nu_y$, which also
imply that
$$\rho(\Lag_{y_0}(\psi^{(k)})) = 1 - \sum_{y \in Y \setminus \{y_0\}} \rho(\Lag_{y}(\psi^{(k)})) \geq \nu_{y_0} > 0.$$
By Proposition~\ref{prop:G}--\ref{prop:G:empty}, we get
$0 = \psi^{k}(y_0) \leq \min_{Y} \psi^{(k)} + R$. Since the price of $y_0$ is never changed, $\psi^{(k)}(y_0) = 0$ and  $R\geq \psi^{(k)}  \geq -R$. 

\noindent \textbf{Step 3} \emph{(Minimum decrease and termination)} In
the second step of the algorithm, when $\psi^{(k)}$ is updated one has
$G_y(\psi^{(k)} - t_t \one_y) \geq G_y(\psi^{(k)}) + \frac{\delta}{N}$. Since $G_y$ is
Lipschitz with some constant $L_G$, this implies that $\abs{t_y} \geq
\frac{\delta}{N L_G}$. Then, since $\psi_0(y) = R$ and for any $k$, $\psi_k(y)
\geq -R$, the number of times $k_y$ the price of a point $y \in Y$ has been updated cannot
be too large:
$$ k_y \delta/(NL_G) \leq 2R,$$ i.e. $k_y \leq (2 R N L_G)/\delta$. Since
this bound on the number of steps is for a single point, it needs to
be multiplied by $N$ to get the total number of steps.
Using the bound on $L_G$ given in \eqref{eq:def-LG}, we get an upper
bound of $\BigO(\frac{N^3}{\delta})$ on the number of iterations of
the algorithm. \end{proof}

\subsection{Semi-discrete optimal transport via Newton's method}\label{sec:newton}

We consider a simple damped Newton's algorithm to solve semi-discrete
optimal transport problem introduced in \cite{kitagawa2016newton}, and
adapted from a similar algorithm for solving Monge-Ampère equations
with Dirichlet boundary conditions \cite{mirebeau2015discretization}.





\subsubsection{Hessian of Kantorovich's functional}
In order to write the Newton's algorithm, we first show that $G$ is
$\Class^1$ (or equivalently $\Kant$ is $\Class^2$) and we compute its
derivatives under a genericity assumption, which depends on the cost
and on the choice of points $Y$. This condition is a bit technical,
but is for instance satisfied for the quadratic cost on $\Rsp^d$ (see
Remark~\ref{rem:generic:quad} below).
\label{subsec:hessian:kantorovich} 

\begin{definition}[Genericity assumption] \label{def:genericity}
Let $\Omega_X,\Omega_Y\subseteq \Rsp^d$ open, $c\in
\Class^1(\Omega_X\times \Omega_Y)$, and $X\subseteq \Omega_X$,
$Y\subseteq \Omega_Y$, with $X$ compact and $Y$ finite.
\begin{itemize} 
\item We call $Y$ \emph{generic with respect to $c$} if
for all distinct $y_0,y_2,y_2\in Y$ and all $t\in \Rsp^2$, one has
$$ \Haus^{d-1}\left(\{x \in \Omega_X \mid (c(x,y_1) - c(x,y_0),
c(x,y_2) - c(x,y_0)) = t \}\right) = 0. $$
\item We call $Y$ \emph{generic with respect to $\partial X$} if for all distinct $y_0,y_1\in Y$ and
all $t\in\Rsp$, one has
$$ \Haus^{d-1}\left(\{x \in \Omega_X \mid c(x,y_1) - c(x,y_0) = t \} \cap \partial X\right) = 0.$$
\end{itemize}
\end{definition}

\begin{remark}[d=1] The genericity  assumption is \emph{never} satisfied in dimension $d=1$, because
it requires that the intersection of the $0$-dimensional sets $\{
c(\cdot,y_i) - c(\cdot,y_0)=t_i\}, $ with $i=1,2$ is empty for all
$t\in\Rsp^2$.  Nonetheless, quasi-Newton methods seem to be quite
efficient in this case as well \cite{degournay2018discrete}.
\end{remark}


\begin{remark}[Sufficient genericity condition]
Assume for all distinct points $y_0,y_1,y_2 \in Y$ and for every $x\in
X$, the vectors $\nabla_x c(x,y_1) - \nabla_x c(x,y_0)$ and $\nabla_x
c(x,y_2) - \nabla_x c(x,y_0)$ are independent. Then, the implicit
function theorem guarantees that for every $t\in \Rsp^2$ the set
$$(c(\cdot,y_1) - c(\cdot,y_0), c(\cdot,y_2) - c(\cdot,y_0))^{-1}(t)$$
is a $(d-2)$ dimensional submanifold, and therefore has zero
$(d-1)$--volume. In particular, the set $Y$ is generic with respect to
$c$ (but not necessarily with respect to $\partial X$).
\end{remark}

\begin{remark}[Quadratic cost]
  \label{rem:generic:quad}
For the quadratic cost $c(x,y) = \frac{1}{2} \nr{x-y}^2$, we have
$\nabla_x c(x,y_i) - \nabla_x c(x,y_0) = y_0 - y_i$. Using the
previous remark, we see that the set $Y$ is generic with respect to
$c$ if it does not include three aligned points.

Genericity with respect to the boundary $\partial X$ requires more
assumptions.  For instance, if $X$ is a strictly convex set (or more
generally if the Gaussian curvature is nonzero at any point on
$\partial X$) and if the cost is quadratic, then $Y$ is automatically
generic with respect to $\partial X$.
As a second example, we assume that $X$ is a compact convex
polyhedron, e.g.  $$X = \{ x\in \Rsp^d\mid \forall 1\leq j\leq M,
\sca{x}{w_j} \leq 1 \},$$ where $w_1,\hdots,w_M \in \Rsp^d$. Then $Y$
is generic with respect to $\partial X$ if for all distinct $y_0,y_1
\in Y$ and any $1\leq i\leq M$, the vectors $y_1-y_0$ and $w_j$ are
independent.
\end{remark}

\begin{example}[Non-differentiability of $G$]
When the set $Y$ isn't generic, the map $G$ might be
non-differentiable. Consider for instance $Y = \{ y_{-1}, y_0,
y_1 \}\subseteq \Rsp^2$ with $y_i = (i,0)$, $X= [-1,1]^2$ and $\rho
= \frac{1}{4}\restr{\Haus^{2}}_{X}$. Define a one-parameter family of
prices $\psi_t: Y\to\Rsp$ by $\psi_t(y_0) = t$ and $\psi_t(y_{\pm 1})
= 0$. Then, for $t\geq 0$,
$$ \Lag_{y_0}(\psi_t) = \{ x = (x_1,x_2) \in \Rsp^2\mid \abs{x_1} \leq \frac{1}{2} \abs{1-t} \}. $$
$$ G_{y_0}(\psi_t) = \rho(\Lag_{y_0}(\psi_t))
= \begin{cases} \frac{1}{2}\abs{1-t} & \hbox{ if } t\leq 1\\ 0 &\hbox{
if not.}\end{cases}, $$ showing that the function $G_{y_0}$ is
non-differentiable at $t=1$.
\end{example}

\begin{theorem} \label{th:Hessian}
If $c\in \Class^2(\Omega_X\times \Omega_Y)$ satisfies the twist
condition (Def. \ref{def:twist}), $Y$ is generic with respect to $c$
and $\partial X$ (Def. \ref{def:genericity}), and the restriction $\rho_{\mid X}$ of $\rho$ to $X$ 
is continuous ($\rho_{\mid X} \in \Class^0(X)$), then the map $G:\Rsp^Y\to\Rsp^Y$ is $\Class^1$,
and
\begin{equation} \label{eq:DG}
\begin{aligned}
 &\forall z\neq y,~~ \frac{\partial G_y}{\partial \one_z}(\psi) = G_{yz}(\psi) := 
\int_{\Lag_{yz}(\psi)} \frac{\rho(x)}{\nr{\nabla_x c(x,y) - \nabla_x c(x,z)}} \dd x,\\
 &\forall y\in Y,~~  \frac{\partial G_y}{\partial \one_y}(\psi) = G_{yy}(\psi) :=  -\sum_{z \in Y\setminus\{y\}} G_{yz}(\psi)
\end{aligned}
\end{equation}
where we denote
$\Lag_{yz}(\psi) =
\Lag_y(\psi)\cap \Lag_z(\psi)$ for $y\neq z$.
\end{theorem}

The formula that one should expect for the partial derivative of $G_y$
with respect to $\one_z$ ($z\neq y$) is already quite clear from
Lemma~\ref{lemma:Gy:pd}. The main difficulty in order to establish
Theorem~\ref{th:Hessian} is to prove that the function $G_{yz}$
defined in \eqref{eq:DG} is continuous.

\begin{lemma} \label{lemma:Gyz:cont}
Assume that $Y$ is generic with respect to $c$ and $\partial X$. Then,
for any $y\neq z \in Y$, the function $G_{yz}$ defined
in \eqref{eq:DG} is continuous.
\end{lemma}
\begin{proof} Let $f = c_{yz} = c(\cdot,y) - c(\cdot,z)$. By Cauchy-Lipschitz's theory, one
can construct a flow $\Phi: [-\eps,\eps]\times \Omega_X^1 \to \Omega_X$, where
$\eps>0$, such that
\begin{equation} \label{eq:flow:grad}
\begin{cases} \Phi(0,x) = x \\
\dot{\Phi}(t,x) = \frac{\nabla f(\Phi(t,x))}{\nr{\nabla f(\Phi(t,x))}^2},
\end{cases}
\end{equation}
where $\dot{\Phi}$ is the derivative with respect to $t$ and $\Omega_X^1\subset  \Omega_X$ is an open set containing $X$.
A simple calculation shows that $\frac{\d}{\d t}f(\Phi(t,x))=1$, which implies  that $f(\Phi(t,x)) = f(\Phi(0,x)) +
t$. Moreover, since $\nabla f / \nr{\nabla f}$ is of class $C^1$ on $\Omega_X$, then $F_t := \Phi(t,\cdot)$ converges pointwise in a $\Class^1$ sense to the identity 
 as $t\to 0$.

Let $(\psi_n)$ be a sequence in $\Rsp^Y$ converging to some
$\psi_\infty \in\Rsp^Y$. We put $a_n = \psi_n(z) - \psi_n(y)$, $a = \psi_\infty(z) - \psi_\infty(y)$ and $ t_n = a_n-a$ and define 
$$ L_n = \Phi(-t_n, \Lag_{yz}(\psi_n)) \quad \mbox{and}\quad L_\infty
= \Lag_{yz}(\psi_\infty).$$ 
By definition, one has $f(\Lag_{yz}(\psi_n))=a_n$ and  $f(\Lag_{yz}(\psi_\infty))=a$. Using the flow property, one gets that  
 both $L_n$ and $L_\infty$ are subsets
of the hypersurface $H = f^{-1}(a)$. Denoting $F_n$ the restriction
of $\Phi(t_n,\cdot)$ to $H$, one has $\Lag_{yz}(\psi_n)=F_n(L_n)$. 

We now need to consider a continuous extension $\bar{\rho}$ of $\rho_{\mid X}$ onto $\Omega_X$, since $L_n$ may not be included in $X$. By a change of variable (see \eqref{eq:changeofvariable} for instance), one gets
\begin{align*}
g_n := G_{yz}(\psi_n) &= \int_{\Lag_{yz}(\psi_n)} \frac{\rho(y)}{\nr{\nabla c_{yz}(y)}} \dd \Haus^{d-1}(y) \\
&= \int_{\Lag_{yz}(\psi_n)} \frac{\bar{\rho}(y)}{\nr{\nabla c_{yz}(y)}} \chi_X(y)\dd \Haus^{d-1}(y) \\
&= \int_{H} \frac{\bar{\rho}(F_n(x))}{\nr{\nabla c_{yz}(F_n(x))}} JF_n(x) \chi_{L_n}(x)  \chi_X(F_n(x)) \dd \Haus^{d-1}(x),
\end{align*}
where $\chi_A$ is the indicator function of $A$. Moreover,
$$ g_\infty := G_{yz}(\psi) = \int_{H} \frac{\rho(x)}{\nr{\nabla c_{yz}(x)}} \chi_{L_\infty\cap X}(x)
\dd \Haus^{d-1}(x).$$
By Lebesgue's dominated convergence theorem, to prove that $(g_n)_{n\geq 0}$
converges to $g_\infty$, it suffices to prove that the 
integrand of $g_n$ (seen as a function on $H$) tends to the integrand of $g_\infty$ $\Haus^{d-1}$-almost everywhere. 
Since  $F_n$ converges to the identity in a $C^1$ sense and $\bar{\rho}$ is continuous, it remains to show that  $\lim_{n\to \infty}\chi_{L_n}(x)\chi_X(F_n(x))=\chi_{L_\infty\cap X}(x)$ for almost every $x\in H$ (for the $(d-1)$ Hausdorff measure).

We first prove that
$\lim\sup_{n\to \infty} \chi_{L_n}(x))\chi_X(F_n(x))\leq \chi_{L_\infty\cap X}(x)$ for every $x\in H$. The
limsup is non-zero if and only if there exists
a subsequence $\sigma(n)$ such that $x \in L_{\sigma(n)}$ and $F_n(x)\in X$. Then, since
$F_{\sigma(n)}(L_{\sigma(n)}) = \Lag_{yz}(\psi_{\sigma(n)})$ we get 
$$
\begin{cases}
c(F_{\sigma(n)}(x),y) + \psi_{\sigma(n)}(y) \leq c(F_{\sigma(n)}(x),w) + \psi_{\sigma(n)}(w) \\ c(F_{\sigma(n)}(x),y)
+ \psi_{\sigma(n)}(y) = c(F_{\sigma(n)}(x),z) + \psi_{\sigma(n)}(z). \end{cases}$$ Passing to the
limit $n\to +\infty$, we see that $x$ belongs to
$\Lag_{yz}(\psi_\infty)=L_\infty$ and to $X$, thus ensuring
$$ \lim\sup_{n\to +\infty} \chi_{L_n}(x) \chi_X(F_n(x))\leq \chi_{L_\infty\cap X}(x). $$

We now pass to the liminf inequality. Denote
$$ S = \left(\bigcup_{w \in Y\setminus \{y,z\}} H_{yzw}(\psi_\infty)\right)
\cup \left(H_{yz}(\psi)\cap \partial X\right), $$
where $H_{yz}$ is defined in Equation~\eqref{eq:Hyz} and $ H_{yzw}(\psi_\infty) :=  H_{yz}(\psi_\infty) \cap  H_{zw}(\psi_\infty)$
which by assumption has zero $(d-1)$ Hausdorff measure. We now prove
that $\lim\inf_{n\to \infty} (x)\chi_{L_n}\chi_X(F_n(x)) \geq \chi_{L_\infty \cap X}$ on
$H\setminus S$. If $x\not\in L_\infty \cap X$, $\chi_{L_\infty\cap X}(x) = 0$ and
there is nothing to prove. We therefore consider $x \in
(L_\infty\cap X) \setminus S$, meaning by definition of $S$ that $x$ belongs
to the interior $\inte{X}$ and that 
$$
\forall w \in Y \setminus\{z,y\}, c(x,y) + \psi_\infty(y) < c(x,w) + \psi_\infty(w).
$$
Since $F_n(x)$ converges to $x$, this implies that for $n$ large
enough one has $F_n(x)\in \inte{X}$ and
$$ \forall w \in Y \setminus\{z,y\}, c(F_n(x),y) + \psi_\infty(y) <
c(F_n(x),w) + \psi_\infty(w). $$ By definition, this means that
$F_n(x)$ belongs to $\Lag_{yz}(\psi_n)$, and therefore $x\in L_n$ by
definition of $L_n$. Thus 
$$
\lim\inf_{n\to +\infty} \chi_{L_n}(x)\chi_X(F_n(x)) = 1\geq \chi_{L_\infty\cap X}(x)
$$ 

%
\end{proof}

\begin{proof}[Proof of Theorem~\ref{th:Hessian}]
Lemma~\ref{lemma:Gy:pd} shows that for any distinct point $y\neq z\in Y$ and any
$\psi\in\Rsp^Y$ one has
\begin{align*}
G_y(\psi + t\one_z) = G_y(\psi) +  \int_0^t G_{yz}(\psi + s\one_z)\dd s.
\end{align*}
Moreover, by Lemma~\ref{lemma:Gyz:cont}, we know that the function
$G_{yz}$ is continuous. The fundamental theorem of calculus implies
that $f: t\mapsto G_y(\psi+t\one_z)$ is differentiable, and that
$f'(0) = \frac{\partial G_y}{\partial \one_z}(\psi) = G_{yz}(\psi).$
To compute the partial derivative of $G_y$ with respect to $\one_y$, we note
that by invariance of $G_y$ under addition of a constant,
$$G_y(\psi + t\one_y) = G_y(\psi - \sum_{z\neq y} t\one_z)$$ The
right-hand side of this expression is differentiable with respect to
$t$, so that the left-hand side is also differentiable, and the chain rule gives
$$ \frac{\partial G_y}{\partial \one_y}(\psi) = - \sum_{z \in
Y\setminus\{y\}} G_{yz}(\psi). $$ Using again the continuity of
$G_{yz}$ on $\Rsp^Y$, we obtain $G\in\Class^1(\Rsp^Y)$.
\end{proof}

\subsubsection{Strong concavity of Kantorovich's functional}
We show here a strict monotonicity property of $G$, which corresponds to  a concavity property on the Kantorovitch functional $\Kant$, since we have $D^2 \Kant  = D G$. 
\begin{theorem} \label{th:StrongConcavity}
Assume that $c,X,Y,\rho$ are as in Theorem~\ref{th:Hessian}, and in
addition that $\rho(\partial X) = 0$ and that the set $\{\rho>0\} \cap \inte{X}$ is connected. Define
$$ \Ccv_+ := \{ \psi \in \Rsp^d\mid \forall y\in Y, G_y(\psi) > 0 \}. $$
$$ \Ccv_\epsilon := \{ \psi \in \Rsp^d\mid \forall y\in Y, G_y(\psi) \geq \epsilon \}. $$
\begin{itemize}
\item 
Kantorovich's functional is locally strongly concave on $\Ccv_+ \cap \{\one_Y\}^\perp$:
$$\forall \psi \in \Ccv_+,~ \forall v \in \{\one_Y\}^\perp \setminus
            \{0\},~ \sca{DG(\psi) v}{v} < 0
            $$
\item For every $\epsilon >0$, the set of functions $\Ccv_\epsilon \cap \{\one_Y\}^\perp$ is compact.
\end{itemize}
\end{theorem}
\begin{definition}[Irreducible matrix]
A square matrix $H$ is called \emph{irreducible} if and only if the
graph induced by $H$ is connected\footnote{the graph induced by the $N\times N$ matrix $H$
is the graph with vertices $\{1,\hdots,N\}$, and where $i,j$ are
linked by an edge if $H_{ij} \neq 0$}, i.e.
$$ \forall (a,b) \in \{1,\hdots,N\}, \exists i_1 = a,\hdots,i_k = b \hbox{ s.t. } \forall j\in\{1,\hdots,k-1\}, H_{i_j, i_{j+1}} \neq 0. $$
\end{definition}
      
\begin{lemma} \label{lemma:irreducible}
Let $H$ be a symmetric irreducible matrix such that  $H_{ij} \geq 0$ if $i\neq j$ and $H_{ii} =
- \sum_{j\neq i} H_{ij}$. Then, $H$ is non-positive and $\ker H = \Rsp
(1,\hdots,1) = \Rsp \one_Y$.
\end{lemma}
\begin{proof}
The non-positivity follows from Gershgorin's circle theorem.  The lemma
will be established if we prove that any vector in the kernel of $H$
is constant.  Consider $v \in \ker H$ and let $i_0$ be an index where
$v$ attains its maximum, i.e. $i_0 \in \arg\max_{1 \leq i \leq n} v_i
$.  Then using $Hv=0$, and in particular $(H v)_{i_0} =0$, one has
$$
  0 = \sum_{i \neq i_0} H_{i, i_0} v_i + H_{i_0, i_0} v_{i_0}
   = \sum_{i \neq i_0} H_{i, i_0} v_i  -\sum_{i \neq i_0} H_{i, i_0} v_{i_0} 
   = \sum_{i \neq i_0} H_{i, i_0} (v_i -  v_{i_0}).
$$  
This follows from $H_{i_0, i_0} = -\sum_{i \neq i_0} H_{i,
i_0}$. Since for every $i\neq i_0$, one has $ H_{i,i_0} \geq 0 $ and
$v_{i_0}-v_i \geq 0 $, this implies that $v_i=v_{i_0}$ for every $i$
such that $ H_{i,i_0} \neq 0 $. By induction and using the
connectedness of the graph induced by $H$, this shows that $v$ has to
be constant.
\end{proof}

\begin{lemma} \label{lem:path-connectedness} Let $U \subseteq \Rsp^d$ be a connected open set,
and $S\subseteq\Rsp^d$ be a closed set such that $\Haus^{d-1}(S) =
0$. Then, $U\setminus S$ is path-connected.
\end{lemma}

\begin{proof}
It suffices to treat the case where $U$ is an open ball, the general
case will follow by standard connectedness arguments. Let $x,y \in
U\setminus S$ be distinct points. Since $U\setminus S$ is open, there
exists $r>0$ such that $\B(x,r)$ and $B(y,r)$ are included in
$U\setminus S$. Consider $H$ the hyperplane orthogonal to the segment
$[x,y]$, and $\Pi_H$ the projection on $H$. Then, since $\Pi_H$ is
$1$-Lipschitz, $\Haus^{d-1}(\Pi_H S) \leq \Haus^{d-1}(S) = 0$, so that
$H\setminus \Pi_H S$ is dense in the hyperplane $H$. In particular,
there exists a point $$z \in \Pi_H(B(x,r))\setminus S
= \Pi_H(B(y,r))\setminus S. $$ By construction the line $z + \Rsp
(y-x)$ avoids $S$ and passes through the balls $\B(x,r)\subseteq U\setminus S$
and $\B(y,r) \subseteq U\setminus S$. This shows that the points $x,y$
can be connected in $U\setminus S$.
\end{proof}

\begin{proof}[Proof of Theorem \ref{th:StrongConcavity}]
Let $Y = \{y_1,\hdots,y_N\}$. Fix $\psi \in \Ccv_+$ and
define $$H_{ij} := \frac{\partial
G_{y_i}}{\partial \one_{y_j}}(\psi).$$ By
Lemma~\ref{lemma:irreducible}, the first claim will hold if we prove
that the matrix $H$ is irreducible. We define $Z = \inte{X}\cap\{\rho
> 0\}$, which by assumption is a connected open set. 

\noindent\emph{Step 1:} We show here that for all $i\in\{1,\hdots,N\}$, $\inte{\Lag_{y_i}(\psi)} \cap Z$ contains
at least a point which we denote $x_i$.  Indeed, since
$\psi\in \Ccv_+$, we know that $\rho(\Lag_{y_i}(\psi)) > 0$. In
addition, by Proposition~\ref{prop:OT},
$\rho(\Lag_{y_i}(\psi)\cap\Lag_{y_j}(\psi)) = 0$ for all $j \neq i$,
and $\rho(\partial X) = 0$ by assumption. This implies that $\rho(L_i)
= \rho(\Lag_{y_i}(\psi)) > 0$, where
$$ L_{i} = \{x\in Z \mid \forall j\neq i, c(x,y_i) + \psi(y_i) <
c(x,y_j) + \psi(y_j)\} \subseteq \inte{\Lag_{y_i}(\psi)}. $$ We
conclude by remarking that $L_i$ is contained in
$\inte{\Lag_{y_i}(\psi)}\cap Z$, which therefore has to be nonempty.

\noindent\emph{Step 2:}
Let $S$ the union of facets that are common to at least three distinct
Laguerre cells, i.e.
 $$ S = \bigcup_{y_1,y_2,y_3 \mathrm{ distinct}} \Lag_{y_1, y_2,
 y_3}(\psi). $$ Then, $Z \setminus S$ is open and
 path-connected. Indeed, by the genericity assumption
 (Def~\ref{def:genericity}), we already know that $\Haus^{d-1}(S) =
 0$, and Lemma~\ref{lem:path-connectedness} then implies that
 $Z\setminus S$ is path-connected.

\noindent\emph{Step 3:} Let $x\in Z\setminus S$ be such that
$x \in \Lag_{y_i}(\psi)\cap\Lag_{y_j}(\psi)$ for $i\neq j$. Then,
$H_{ij} > 0$. To see this, we note that since $x$ belongs to the complement of $S$,
$$ \begin{cases}
c(x,y_i) + \psi(y_i)  = c(x,y_j) + \psi(y_j),\\
\forall k \not\in\{i,j\},  c(x,y_i) + \psi(y_i)  < c(x,y_k) + \psi(y_k).
\end{cases} $$
This implies that there exists a ball with radius $r>0$ around $x$ such that
$$ \forall x'\in \B(x,r),
\forall k\not\in\{i,j\}, c(x,y_i) + \psi(y_i)  < c(x,y_k) + \psi(y_k), $$
directly implying that
$$ H_{yy'}(\psi) \cap \B(x,r) \subseteq \Lag_{y_i y_j}(\psi). $$ By
the twist hypothesis and the inverse function theorem, $H_{yy'}(\psi)$
is a $(d-1)$-dimensional submanifold. In addition, $\rho(x)>0$ because
$x$ belongs to $Z$. This implies that
\begin{align*}
H_{ij} &= \int_{\Lag_{y_iy_j}(\psi)} \frac{\rho(x')}{\nr{\nabla_x c(x',y_i) - \nabla_x c(x',y_j)}} \dd \Haus^{d-1}(x') \\
&\geq  \int_{H_{y_iy_j}(\psi) \cap \B(x,r)} \frac{\rho(x')}{\nr{\nabla_x c(x',y_i) - \nabla_x c(x',y_j)}} \dd \Haus^{d-1}(x') > 0.
\end{align*}

\noindent\emph{Step 4:} We now fix $i\neq j \in\{1,\hdots,N\}$ and the points $x_i,x_j$ 
 whose existence is established in Step~1:
$$x_i \in\inte{\Lag_{y_i}(\psi)} \cap Z,\quad
x_j \in \inte{\Lag_{y_j}(\psi)} \cap Z,$$ so that in particular
$x_i,x_j$ belongs to $Z\setminus S$.  By Step 2, we get the existence
of a continuous path $\gamma\in\Class^0([0,1], Z\setminus S)$ such
that $\gamma(0) = x_i$ and $\gamma(1) = x_j$. We define a sequence
$i_{k} \in \{1,\hdots,N\}$ of indices by induction, starting from
$i_0=i$. For $k\geq 0$ we define $t_k = \max \{ t\in
[0,1] \mid \gamma(t)\in\Lag_{y_{i_k}}\}.$ If $t_k = 1$ we are done. If
not, $\gamma(t_k)$ belongs to exactly two distinct Laguerre cells, and
we define $i_{k+1} \neq i_k$ so that
$\gamma(t_k) \in \Lag_{y_{i_k}}(\psi)\cap \Lag_{y_{i_{k+1}}}(\psi)$. By
definition of $t_i$ as a maximum, the points $y_1,\hdots, y_k$ must be
distinct, so that $t_\ell = 1$ after a finite number of iterations and
then $i_\ell = j$. By Step 3, we get that $H_{y_{i_k} y_{i_{k+1}}} >
0$ for any $k \in \{0,\ell-1\}$, proving that the matrix $H$ is
irreducible, thus $\ker \DD G(\psi) = \Rsp \one_Y$ by Lemma~\ref{lemma:irreducible}, implying the strict concavity property.\\

\noindent\emph{Compactness of $\Ccv_\epsilon \cap \{\one_Y\}^\perp$:} By continuity of the function $G$, this set is closed. 
By Proposition~\ref{prop:G}-\ref{prop:G:allnonempty}, $\max_Y \psi - \min_Y$ is bounded on the set $\Ccv_\epsilon$. This implies that $\Ccv_\epsilon \cap \{\one_Y\}^\perp$ is bounded since every function of $\{\one_Y\}^\perp$ has a mean value equal to zero, and is thus compact.

\end{proof}

\subsubsection{Damped Newton algorithm and its convergence}
\begin{algorithm}
\begin{description}
  \item[Input] A tolerance $\eta > 0$ and an initial $\psi^{(0)}\in
    \Rsp^Y$ such that
\begin{equation}\label{eq:nonzerocells}
  \eps :=  \frac{1}{2} 
  \min\left[\min_{y\in Y} G_y(\psi^{(0)}),~ \min_{y\in Y} \nu_y\right] >  0.
\end{equation}
\item[While] $\nr{G(\psi^{(k)}) - \nu}_\infty  \geq \eta$
  \begin{description}
\item[Step 1] Compute $v^{(k)}$ satisfying
$$  \begin{cases}
\D G(\psi^{(k)}) v^{(k)} = \nu - G(\psi^{(k)}) \\
\sum_{y\in Y} v^{(k)}(y) = 0
\end{cases}$$
\item[Step 2] Determine the minimum $\ell \in \mathbb{N}$ such that $\psi^{(k,\ell)} :=
  \psi^{(k)} + 2^{-\ell} v^{(k)}$ satisfies
\begin{equation*}
\left\{
\begin{aligned}
&\forall y\in Y, G_y(\psi^{(k,\ell)}) \geq \eps \\
&\nr{G(\psi^{(k,\ell)}) - \nu} \leq (1-2^{-(\ell+1)}) \nr{G(\psi^{(k)}) - \nu}
\end{aligned}
\right.
\end{equation*}
\item[Step 3] Set $\psi^{(k+1)} = \psi^{(k)} + 2^{-\ell}  v^{(k)}$ and $k\gets k+1$.
  \end{description}
    \item[Output] A vector $\psi^{(k)}$ that satisfies $\nr{G(\psi^{(k)}) - \nu}_\infty  \leq \eta$.
\end{description}
\caption{Damped Newton algorithm\label{algo:newton}}
\end{algorithm}

\begin{proposition}\label{prop:convergence-newton}
Let $G = (G_1,\hdots,G_N) \in \Class^1(\Rsp^N,\Rsp^N)$ be a function satisfying the following properties:
\begin{enumerate}
\item (Invariance and image) $G$ is invariant under the addition of a constant, $G_i(\psi)\geq 0$ 
and $\sum_i G_i(\psi) = 1$ for all $\psi\in\Rsp^N$.
\item (Compactness) For any $\eps>0$ the set $\Ccv_\eps \cap \{\one\}^\perp$ is compact, where
        \begin{align*}
         &\Ccv_{\eps} :=  \left\{ \psi \in \Rsp^N \mid
            \forall i,~ G_i(\psi) \geq \epsilon \right\} \\
         &\one = (1,\hdots,1)\in \Rsp^N
        \end{align*}
\item (Strict monotonicity) The matrix $\D G(\psi)$ is symmetric nonpositive, and 
$$\forall \psi \in \Ccv_{\eps},~ \forall v \in \{\one\}^\perp \setminus
            \{0\},~ \sca{DG(\psi) v}{v} < 0.$$
\end{enumerate}
Then Algorithm~\ref{algo:newton} terminates in a finite number of
steps. More precisely, the iterates $(\psi^{(k)})$ of
Algorithm~\ref{algo:newton} satisfy, for some
$\tau^*>0$,
    $$ \nr{G(\psi^{k+1}) - \nu} \leq \left( 1 - \frac{\tau^\star}{2} \right)
    \nr{G(\psi^{k}) - \nu}. $$
\end{proposition}

\begin{proof}\mbox{}\\
\textbf{Estimates.} 
Let $\nu \in \Rsp^N$ be such that $\sum_i \nu_i = 1$. We assume that
$ \psi^{(0)} \in \Rsp^N \cap \{\one\}^\perp$ is chosen so that 
$$\epsilon = \frac{1}{2} \min \left( \min_i
G_{i}(\psi^0), \min_i \nu_i \right) >0,$$ and we let $\Ccv
:= \Ccv_\eps \cap \{\one\}^\perp$. Let $\psi\in\Ccv$. By
Theorem~\ref{th:StrongConcavity}, the matrix $\D G(\psi)$ is symmetric
non-positive, and its kernel is the one-dimensional space $\Rsp
\one$. Thus, the equation
$$\begin{cases}
\D G(\psi) v = \nu - G(\psi),\\
 \sum_i v_i = 0, \end{cases}$$ has a unique solution, which we denote
$v(\psi)$, and we let $\psi_\tau = \psi + \tau v(\psi)$. By continuity of 
$\D G(\psi)$ over the compact domain $\Ccv$, the non-zero eigenvalues
of $- \D G(\psi)$ lie in $[a,A]$ for some $0< a\leq A < +\infty$. In
particular, there exists a constant $M>0$ such that for all $\psi \in \Ccv$
\begin{equation} \label{eq:newt:psi}
\frac{\nr{G(\psi) - \nu}}{A} \leq \nr{v(\psi)} \leq \frac{\nr{G(\psi)
- \nu}}{a} \leq M. 
\end{equation}
In particular, the function $F:(\psi,\tau) \in \Ccv \times
[0,1] \mapsto \psi_\tau$ is continuous. Since $\Ccv \times [0,1]$ is
compact, $K := F(\Ccv \times [0,1])$ is also compact. Then, by uniform
continuity of $\D G$ over $K$, we see that there exists an increasing
function $\omega$ such that $\lim_{t\to 0}\omega(t) = 0$ and $\nr{\D
G(\psi) - \D G(\psi')} \leq \omega(\nr{\psi - \psi'})$ for all
$\psi, \psi' \in K$.  Since $G$ is of class $\Class^1$, a Taylor
expansion in $\tau$ gives
\begin{equation}
        \label{eq:taylor-g}
        G(\psi_\tau) = G(\psi + \tau v(\psi)) = (1 - \tau)
        G(\psi) + \tau \nu + R(\tau)
\end{equation}
where $ R(\tau) = \int_0^\tau (\D G(\psi_t) - \D G(\psi)) v(\psi) \dd
t $ is the integral remainder. Then, we can bound the norm of $
R(\tau) $ for $\tau \in [0,1]$:
\begin{align}
        \nr{R(\tau)} &= \nr{\int_0^\tau (\D G(\psi_t) - \D G(\psi)) v(\psi) \dd t} \notag \\
                     &\leq \nr{v(\psi)} \int_0^\tau \omega(\nr{\psi_t - \psi})
                     \dd t \notag \\
                     &\leq  \nr{v(\psi)} \tau\ \omega(\tau \nr{v(\psi)}). \label{ineq:Rtau}
\end{align}
To establish the first  inequality, we used that $\psi$ and $\psi_t$ belong to the compact set $K$ 
 and for the second one that $\omega$ is increasing and that $t\in [0,\tau]$. \\

\noindent \textbf{Linear convergence.}
We first show the existence of $\tau^*_1>0$ such that for all
$\psi\in\Ccv$ and $\tau \in (0,\tau^*_1)$, one has
$\psi_\tau \in \Ccv$.  By definition of $\eps$, for every
$i \in \{1, \ldots, N\}$ one has $ \nu_i \geq 2 \epsilon $ and $
G_i(\psi) \geq \epsilon $. Using \eqref{eq:taylor-g}
and \eqref{ineq:Rtau}, one deduces a lower bound on $G_i(\psi_\tau)$:
\begin{align*}
G_i(\psi_\tau) &\geq (1 - \tau) G_i(\psi) + \tau \nu_i +
    R_i(\tau)\\ 
    &\geq (1 + \tau) \epsilon - \nr{R(\tau)} \\
    &\geq \eps + \tau(\eps - M \omega(\tau M)).
\end{align*}
If we choose $\tau^*_1>0$ small enough so that $M\omega(\tau_1^*
M) \leq \eps$, this implies that $\psi_\tau \in \Ccv$ for all
$\psi\in\Ccv$ and $\tau \in [0,\tau_1^*]$.

We now prove that there exists $\tau_2^*>0$ such that for $\tau \in
[0,\tau_2^*]$, one has $\nr{G(\psi_\tau) - \nu} \leq
(1-\tau/2) \nr{G(\psi) - \nu}$.  From Equation~\eqref{eq:taylor-g}, we have
$ G(\psi_\tau) - \nu = (1 - \tau) (G(\psi) - \nu) + R(\tau),$ and it
is therefore sufficient to prove 
$$ \nr{R(\tau)} \leq
\frac{\tau}{2}\nr{G(\psi) - \nu}.$$
With the upper bound on $R(\tau)$ given in Equation~\eqref{ineq:Rtau}
combined with the two bounds on $\nr{v(\psi)}$
of Equation~\eqref{eq:newt:psi}, this condition will hold provided that $\tau$
is such that $\omega(\tau M)/a \leq 1/2$.

These two bounds directly imply that the $\tau^{(k)}$ chosen in
Algorithm~\ref{algo:newton} always satisfy $\tau^{(k)}\geq \tau^*$ with $\tau^* = \frac{1}{2}\min(\tau_1^*,\tau_2^*)$, so that
$$ \nr{G(\psi^{(k+1)}) - \nu} \leq \left(1-\frac{\tau^*}{2}\right)\nr{G(\psi^{(k+1)}) - \nu}. $$
This establishes the linear convergence of Algorithm~\ref{algo:newton}.
\end{proof}

\subsubsection{Application to optimal transport}
The damped Newton algorithm allows to solve the semi-discrete optimal transport problem when applied to the function $G$ given by $G_y(\psi)=\rho(\Lag_\psi(y))$. 
The function $G$ satisfies the assumptions of
Proposition~\ref{prop:convergence-newton}: it is of class $C^1$ by
Theorem~\ref{th:Hessian}, satisfies the compactness and strict
monotonicity property by Theorem~\ref{th:StrongConcavity} and clearly
also satisfies Assumption \emph{(1)}. We therefore have the following
theorem:

\begin{theorem} We make the following assumptions:
  \begin{itemize}
  \item  $c\in \Class^2(\Omega_X\times \Omega_Y)$ satisfies the twist
    condition (Def. \ref{def:twist}),
  \item  $Y$ is generic with respect to $c$
    and $\partial X$ (Def. \ref{def:genericity}),
  \item  $\rho(\partial X) = 0$ and 
    $\rho_{\mid X}\in\Class^0(X)$ is such that the set $\{ \rho >0\} \cap \inte{X}$ is connected.
  \end{itemize}
  Then Algorithm~\ref{algo:newton} terminates in a finite number of
  steps. More precisely, the iterates $(\psi^{(k)})$ of
  Algorithm~\ref{algo:newton} satisfy, for some
$\tau^*>0$,
    $$ \nr{G(\psi^{k+1}) - \nu} \leq \left( 1 - \frac{\tau^\star}{2} \right)
    \nr{G(\psi^{k}) - \nu}. $$
\end{theorem}

\begin{remark}[Quadratic convergence]
The above theorem shows that the convergence of the damped Newton algorithm is globally linear. When the cost $c$ satisfies the \emph{Ma-Trudinger-Wang} (MTW) condition that appears in the regularity theory of optimal transport,  and when the density function $\rho$ is Lipschitz-continuous, the convergence is even locally quadratic~\cite{kitagawa2016newton}. 
\end{remark}

\begin{remark}[Implementation] 
The most difficult part in the implementation of both
Oliker--Prussner's algorithm and the damped Newton algorithm is the
computation of the Laguerre tessellation.  In several interesting
cases, Laguerre cells can be obtained by intersecting Power diagram
with surfaces, such as planes, spheres or triangulated
surfaces. Recall that the Power diagram of a weighted point cloud $P =
(p_1,\hdots, p_N) \in (\Rsp^d)^N$ with weights
$(\omega_1,\hdots,\omega_N)\in\Rsp$, is defined by the cells
$$
\mathrm{Pow}_i:=\{x\in \Rsp^3\mid \|x-p_i\|^2 + \omega_i \leq \|x-p_j\|^2 + \omega_j\quad  \forall j\}.
$$ This diagram can be efficiently computed by using libraries, such
as for instance \textsc{Cgal} or \textsc{Geogram}.

When $c(x,y)=\|x-y\|^2$ is the quadratic cost and $X$ is a
triangulated surface in $\Rsp^3$, the Laguerre cells can be obtained
by intersecting power cells with the triangulated surface
$X$~\cite{merigot2018algorithm}. This approach is also used in several
inverse problems arising in nonimaging optics that correspond to
optimal transport problems. For instance, when
$c(x,y)=-\log(1-\sca{x}{y})$ is the reflector cost on the unit sphere,
the Laguerre cells are obtained by intersecting power cells with the
unit sphere~\cite{meyron2019light,de2016far}.
\end{remark}

\subsection{Semi-discrete entropic transport}\label{sec:rot_semi-discrete}
The semi-discrete entropic transport problem was introduced by
Genevay, Cuturi, Peyré and Bach \cite{genevay2016stochastic}, as a
regularization of high-dimensional optimal transport problems, see
also \cite{cuturi2018semidual}. Such high-dimensional problems occur
for instance in image generation, we refer for instance to the work of
Galerne, Leclaire and Rabin \cite{galerne2018texture}. Our goal here
is to investigate briefly the relation between the semi-discrete
Kantorovich functional $\Kant$ and its entropically regularized
variant $\Kant^\eta$.

Let $X \subseteq \Omega_X$ be compact and $Y \subseteq \Omega_Y$ be
finite. We recall that the entropy of a probability measure is
\begin{equation}
  \Ent(\rho) = \begin{cases} \int_{\Omega_X} \rho \log \rho &\hbox{ if } \rho \in\Probac(X) \\
    +\infty &\hbox{ if not }
    \end{cases}
\end{equation}
We also recall that if $\gamma$ is a transport plan between a
probability density $\rho$ in $\Probac(X)$ and a finitely supported
measure $\nu = \sum_{y\in Y} \nu_y\delta_y$, then there exists
probability densities $\rho_y\in \Meas^+(X)\cap \LL^1(X)$ such that
$\gamma = \sum_y \rho_y \otimes \delta_y$, and which satisfy the two marginal
conditions
\begin{equation}
\label{eq:sde:marginal} \sum_y \rho_y = \rho \qquad \hbox{ and } \qquad \int \rho_y
= \nu_y.
\end{equation}
Then, the entropy of $\gamma$, with respect to
$\Haus^d\otimes \Haus^0$,  is the sum of the entropies of the
$\rho_y$. This leads to the following definition.

\begin{definition}[Semi-discrete entropic transport]
The entropy-regularized semi-discrete optimal transport problem
between a density $\rho \in\Probac(X)$ and a finitely supported
measure $\nu = \sum_{y \in Y}\nu_y \delta_y$ is defined for any
$\eta>0$ by
$$
\begin{aligned}
\KP^\eta 
&= \min \left\{ \sca{c}{\gamma} + \eta\sum_{y\in Y} \Ent(\rho_y) \mid \rho_y \in \LL^1(X), \hbox{s.t. } \sum_{y\in Y} \rho_y = \rho, \int_X \rho_y = \nu_y \right\}.
\end{aligned}
$$\end{definition}

\subsubsection{Dual problem}
The dual problem is constructed, as always, by introducing Lagrange
multipliers $\phi,\psi$ for the marginal constraints
\eqref{eq:sde:marginal}. We skip the derivation of the dual problem,
which is very similar to the one presented in the discrete case
(Section~\ref{sec:dot-entropic}), and we directly state it:
\begin{equation}
\DP^\eta = \sup_{(\phi,\psi)\in\LL^1(X)\times \Rsp^Y} \sca{\phi}{\rho} - \sca{\psi}{\nu} - \eta \sum_{y\in Y} \int_X e^{- \frac{c(x,y) + \psi(y) - \phi(x)}{\eta}} \dd x.
\end{equation}
Maximizing with respect to $\phi$ for a given $\psi \in \Rsp^Y$, we
obtain a second formulation as a finite-dimensional optimization
problem involving a regularized Kantorovich functional, exactly as in
\S\ref{sec:dot-entropic}:
\begin{align}
&\DP^{\eta'} = \sup_{\psi \in \Rsp^Y} \Kant^\eta(\psi)\\
\hbox{ where }
&\Kant^\eta(\psi) := -\eta \int_X \log\left(\sum_{y\in Y} e^{- \frac{c(x,y) + \psi(y)}{\eta}}\right)\rho(x) \dd x  - \sca{\psi}{\nu} + \eta \Ent(\rho) \notag
\end{align}

In order to express the gradient and the Hessian of $\Kant^\eta$, we
also need the notion of smoothed laguerre cells, introduced in the
discrete case (see Equation~\eqref{eq:RLag}) and defined by
\begin{equation}
\RLag^{\eta}_{y}(\psi) = \frac{e^{- \frac{c(\cdot,y) + \psi(y)}{\eta}}}{\sum_{z\in Y} e^{- \frac{c(\cdot,z) + \psi(z)}{\eta}}}.
\end{equation}

\begin{theorem} Assume that $c\in \Class^1(X\times Y)$ is twisted. Then,
$\Kant^\eta$ is a $\Class^2$ strictly concave function over $\Rsp^Y$,
with first-order partial derivatives
\begin{equation}
\forall y\in Y,~ \frac{\partial \Kant^\eta}{\partial \one_y}(\psi)
= G_y^\eta(\psi) - \nu_y \hbox{ with } G_y^\eta(\psi) := \sca{\RLag^{\eta}_{y}(\psi)}{\rho}.
\end{equation}
and second-order partial derivatives
\begin{equation} \label{eq:d2-Keta}
\begin{aligned}
&\forall y\neq z\in Y,~\frac{\partial^2 \Kant^\eta}{\partial \one_z\partial\one_y}(\psi)
= G_{yz}^\eta(\psi) := \frac{1}{\eta} \sca{\RLag^{\eta}_{y}(\psi)\RLag_{z}^{\eta}(\psi)}{\rho}\\
&\forall y\in Y,~\frac{\partial^2 \Kant^\eta}{\partial\one_y^2}(\psi)
=  G_{yy}^\eta(\psi) := -\sum_{z\neq y} G_{yz}^\eta(\psi)
\end{aligned}
\end{equation}
If $\psi$
is a maximizer in $\DP^{\eta'}$, then the solution to $\KP^\eta$ is
given by
$$\gamma = \sum_y \rho_y \otimes \delta_y, \hbox{ with } \rho_y
= \RLag_y^{\eta}(\psi) \rho. $$
\end{theorem}
We skip the proof of this theorem which follows closely the one of
Theorem~\ref{thm:hessian_regKant} in the discrete case.

\subsubsection{Strong convergence of $\Kant^\eta$ to $\Kant$}
The next proposition show that for twisted costs, $G^\eta$ converges
to $G$ locally uniformly (i.e. $\Kant^\eta$ converges to $\Kant$ in
$\Class^1$). Its proof follows closely the proof of the Lipschitz
estimate for $G^\eta$ in Proposition~\ref{prop:Glip}.

\begin{proposition}\label{prop:regularityGeta} Assume that  $c\in \Class^2(\Omega_X\times \Omega_Y)$ is
  twisted (Def~\ref{def:twist}), that $X\subseteq \Omega_X$ is compact
  and $Y\subseteq \Omega_Y$ is finite and that $\rho\in\Probac(X)\cap
  \LL^\infty(X)$. Then:
  \begin{enumi}
  \item $G^\eta$ converges pointwise to $G$ as $\eta\to 0$, i.e.
    $$\forall y\in Y,~ \RLag^{\eta}_{y}(\psi)
      \xrightarrow[\LL^1(X)]{\eta\to 0} \one_{\Lag_y(\psi)}.$$
      \item $G^\eta$ is $L$-Lipschitz, where $L$ depends on $c$, $X$
        and $N$ only.
      \item $G^\eta$ converges locally uniformly to $G$.
  \end{enumi}
 \end{proposition}
\begin{remark}
  The formula \eqref{eq:d2-Keta} implies that $G^\eta$ is
  $\frac{1}{\eta}$-Lipschitz continous, see
  \cite{genevay2016stochastic} or Remark~5.1 in
  \cite{peyre2019computational}.  When the cost is twisted, the
  previous proposition shows that the family of functions
  $(G^\eta)_{\eta> 0}$ is in fact uniformly Lipschitz.
\end{remark}
\begin{proof}
  (i) To prove this statement, it suffices to remark that
  $$ \lim_{\eta\to 0} \RLag^\eta_y(\psi)(x) =
  \begin{cases}
    1 \hbox{ if } x\in\SLag_y(\psi), \\
    0 \hbox{ if } x\in X\setminus \Lag_y(\psi).
  \end{cases}
  $$ Thus, $\RLag^\eta_y(\psi)$ converges to $\one_{\Lag_y(\psi)}$
  pointwise on the complement in $X$ of $\Lag_y(\psi)\setminus
  \SLag_y(\psi)$. Since the cost is twisted, $\Lag_y(\psi)\setminus
  \SLag_y(\psi)$ is Lebesgue-negligible, therefore proving that
  $\RLag^\eta_y(\psi)$ converges almost everywhere to
  $\one_{\Lag_y(\psi)}$. One concludes by applying Lebesgue's
  dominated convergence theorem.

(ii) To prove that $G^\eta$ is Lipschitz, we compute an upper bound
  on $\D G^\eta = \D^2 \Kant^\eta$, recalling that for $z\neq y$,
  $$ \frac{\partial^2 \Kant^\eta}{\partial \one_z
    \partial\one_y}(\psi) = \frac{1}{\eta}
  \sca{\RLag^{\eta}_{y}(\psi)\RLag_{z}^{\eta}(\psi)}{\rho}.$$ For
  getting such an upper bound, as in Proposition~\ref{prop:Glip}, we
  will apply the co-area formula using the function
$$f(x) = c(x,z) + \psi(z) - (c(x,y) + \psi(y))$$
We note that
$$ \RLag^{\eta}_{y}(\psi)(x)\RLag_{z}^{\eta}(\psi)(x) = \frac{e^{- \frac{c(x,y) + \psi(y) + c(x,z) + \psi(z)}{\eta}}}{\left(\sum_{w\in Y} e^{- \frac{c(x,w) + \psi(w)}{\eta}}\right)^2}.$$
 When $f(x) \geq 0$, we use the equality $c(x,z) + \psi(z)=c(x,y) +
\psi(y) + f(x)$ to obtain the upper bound
$$ \frac{e^{- \frac{c(x,y) + \psi(y) + c(x,z) + \psi(z)}{\eta}}}{\left(\sum_{w\in Y} e^{- \frac{c(x,w) + \psi(w)}{\eta}}\right)^2} \leq \frac{\left(e^{- \frac{c(x,y) + \psi(y)}{\eta}}\right)^2 e^{-\frac{f(x)}{\eta}}}{\left(\sum_{w\in Y} e^{- \frac{c(x,w) + \psi(w)}{\eta}}\right)^2} \leq e^{-\frac{f(x)}{\eta}}.$$
Reasoning similarly when $f(x)\leq 0$, we obtain 
$$ \RLag^{\eta}_{y}(\psi)(x)\RLag_{z}^{\eta}(\psi)(x) \leq e^{-\frac{\abs{f(x)}}{\eta}}. $$
Using the co-area formula and the previous upper bound, we get 
$$ \begin{aligned}
  &\sca{\RLag^{\eta}_{y}(\psi)\RLag_{z}^{\eta}(\psi)}{\rho}\\
  &\quad= \int_X \RLag^{\eta}_{y}(\psi)(x)\RLag_{z}^{\eta}(\psi)(x) \rho(x) \dd x \\
  &\quad= \int_{-\infty}^\infty \int_{f^{-1}(t)}\RLag^{\eta}_{y}(\psi)(x)\RLag_{z}^{\eta}(\psi)(x) \frac{\rho(x)}{\nr{\nabla f(x)}}  \dd \Haus^{d-1}(x) \dd t\\
  &\quad\leq \int_{-\infty}^\infty \int_{f^{-1}(t) \cap X} e^{-\frac{\abs{t}}{\eta}} \frac{\rho(x)}{\nr{\nabla f(x)}}  \dd \Haus^{d-1}(x) \dd t 
\end{aligned}
$$ We now apply Lemma~\ref{lemma:bnd-area}, which gives an upper bound on
$\Haus^{d-1}(f^{-1}(t)\cap X)$ in terms of the constants $\kappa_{yz} =
\min_X \nr{\nabla f}$ and $M_{yz}= \max_X \nr{\DD^2 f}$:
\begin{align*}
  &\sca{\RLag^{\eta}_{y}(\psi)\RLag_{z}^{\eta}(\psi)}{\rho} \\
  &\quad \leq c(d) \frac{\nr{\rho}_\infty}{\kappa_{yz}} \left(1+\frac{M_{yz}}{\kappa_{yz}} \diam(X)\right) \diam(X)^{d-1}  \int_{-\infty}^\infty e^{-\frac{\abs{t}}{\eta}} \dd t \\
  &\quad \leq C \eta,
\end{align*}
where the constant $C$ depends on the domain, $\rho$ and the cost only. In other words, for $z\neq y$,
$$ \abs{\frac{\partial^2 \Kant^\eta}{\partial \one_z
    \partial\one_y}(\psi)} = \frac{1}{\eta}
\sca{\RLag^{\eta}_{y}(\psi)\RLag_{z}^{\eta}(\psi)}{\rho} \leq C.$$ A
similar upper bound holds Since the diagonal elements, thus ensuring
that $G^\eta$ is $L$-Lipschitz with $L$ independent on $\eta$. (iii)
follows at once from pointwise convergence and the uniform Lipschitz
estimate.
\end{proof}


To finish this section, we show that under the genericity assumption
introduced in Section~\ref{subsec:hessian:kantorovich}, the Hessian of
Kantorovich's regularized functional $\D^2 \Kant^\eta$ converges
pointwise to $\D^2 \Kant$ as $\eta$ converges to $0$.

\begin{theorem} \label{thm:smooth-laguerre}
Assume that $c\in \Class^1(X\times Y)$ is twisted
(Def~\ref{def:twist}), that $Y$ is generic with respect to $c$ and
$\partial X$ (Def~\ref{def:genericity}), that
$\rho_{\mid X} \in \Class^0(X)$. Then, 
$$\forall \psi\in \Rsp^Y,~ \lim_{\eta \to 0} \D G^\eta(\psi) = \D G(\psi). $$
\end{theorem}




\begin{proof}
We let $f(x) = c(x,z) + \psi(z) - (c(x,y) + \psi(y)) $ and $\epsilon >0$. From the proof of Proposition~\ref{prop:regularityGeta}, one has for every $x\in X\setminus f^{-1}([-\eps,\eps])$ that 
$$
\RLag^{\eta}_{y}(\psi)(x) \RLag^{\eta}_{z}(\psi)(x) \leq e^{-\frac{\epsilon}{\eta}}.
$$
This implies that 
 $$\lim_{\eta\to 0} \int_{X\setminus f^{-1}([-\eps,\eps])}
\RLag^{\eta}_{y}(\psi)(x) \RLag^{\eta}_{z}(\psi)(x) \rho(x) \dd x =
0.$$
Since the smoothed Laguerre are non-negative on $X$, we get 
\begin{align*}
G_{yz}^\eta(\psi) = \frac{1}{\eta} \int_{X} \RLag^{\eta}_{y}(\psi)(x) \RLag^{\eta}_{z}(\psi)(x) \rho(x) \dd x
\underset{\eta\to 0}{\sim} G_{yz}^{\eta,\eps}(\psi)
\end{align*}
where 
$$
G_{yz}^{\eta,\eps}(\psi):=\frac{1}{\eta}\int_{X \cap f^{-1}([-\eps,\eps])} \RLag^{\eta}_{y}(\psi)(x) \RLag^{\eta}_{z}(\psi)(x) \rho(x) \dd x.
$$
As in the proof of Lemma~\ref{lemma:Gyz:cont}, we construct $\Phi:
[-\eps,\eps]\times M \to \Omega_X$ (for some positive $\eps$),  with $M =
f^{-1}(0)$ by solving the Cauchy problem
$$\begin{cases} \Phi(0,x) = x \\
\frac{d}{d t} \Phi(t,x) = \frac{\nabla f(\Phi(t,x))}{\nr{\nabla f(\Phi(t,x))}^2},
\end{cases}$$
so that $f(\Phi(t,f^{-1}(0))) = t$. Then one has $\Phi([-\epsilon,\epsilon]\times M)=X\cap f^{-1}([-\epsilon,\epsilon])$, and by a change of variable
formula, using the definition of the smoothed indicator function of
Laguerre cells and the definition of $f$, we get
\begin{align*}
&G_{yz}^{\eta,\eps}(\psi) \\ =  &\frac{1}{\eta} \int_M \int_{-\eps}^\eps \RLag^{\eta}_{y}(\psi)(\Phi(t,x))
\RLag^{\eta}_{z}(\psi)(\Phi(t,x)) \rho(\Phi(t,x))J_\Phi(t,x)\dd t \dd \Haus^{d-1}(x).\\
\end{align*}
Remark that
\begin{align*}
\RLag^{\eta}_{y}(\psi)(\Phi(t,x)) \RLag^{\eta}_{z}(\psi)(\Phi(t,x)) 
&= \frac{e^{- \frac{c(\Phi(t,x),y) + \psi(y)+c(\Phi(t,x),z) + \psi(z)}{\eta}}}{\left(\sum_{z\in Y} e^{- \frac{c(\Phi(t,x),z) + \psi(z)}{\eta}}\right)^2}\\
&= \chi_\eta(t,x) e^{- \frac{\abs{f(\Phi(t,x))}}{\eta}}\\
&= \chi_\eta(t,x)e^{- \frac{\abs{t}}{\eta}},
\end{align*}
where we put 
$$ \chi_\eta(t,x) :=
\frac{e^{- 2\min\left(\frac{c(\Phi(t,x),y) + \psi(y)}{\eta}, \frac{c(\Phi(t,x),z) + \psi(z)}{\eta}\right)}}{\left(\sum_{z\in Y} e^{- \frac{c(\Phi(t,x),z) + \psi(z)}{\eta}}\right)^2}.
$$
We deduce that  one gets 
$$
G_{yz}^{\eta,\eps}(\psi) = \int_{M} g_\eta(x)  \dd\Haus^{d-1}(x)
$$
with
$$ g_\eta(x) = \int_{-\eps}^\eps \chi_\eta(t,x)   \frac{e^{- \frac{\abs{t}}{\eta}}}{\eta} \rho(\Phi(t,x)) J_\Phi(t,x) \dd t$$
We first note that $\abs{\chi_\eta(t,x)} \leq 1$, so that
$(g_\eta)_{\eta}$ is  bounded in $\LL^1(M)$.  We now prove
that $g_\eta(x)$ converges to $g(x) = \rho(x)J_\Phi(0,x)\one_{\Lag_{yz}(\psi)}(x)$
$\Haus^{d-1}$-almost everywhere. More precisely, we show convergence
for any $x$ belonging to the following set $E$, which has full
$\Haus^{d-1}$ measure in $\Lag_{yz}(\psi)$ by the genericity
assumption (Def. \ref{def:genericity}):
$$E = \Lag_{yz}(\psi) \setminus \left(\partial
X \cup \bigcup_{w\not\in \{y,z\}}\Lag_w(\psi)\right).$$ We split the
integral defining $g_\eta$ by distinguishing the case $t\leq 0$ and
$t\geq 0$. For $t\geq 0$,
$$ t = f(\Phi(t,x)) = c(\Phi(x,t),z) +\psi(z) - (c(\Phi(x,t),y) + \psi(y)) \geq 0$$
giving
\begin{align*}
\chi_\eta(t,x) &=
\frac{\left(e^{- \frac{c(\Phi(t,x),y) + \psi(y)}{\eta}}\right)^2}
{\left(\sum_{w\in Y} e^{- \frac{c(\Phi(t,x),w) + \psi(w)}{\eta}}\right)^2}\ \\
&=\left(\sum_{w\in Y} e^{- \frac{c(\Phi(t,x),w) + \psi(w) - (c(\Phi(t,x),y)+\psi(y))}{\eta}}\right)^{-2}\ \\
&= \left(1 + e^{- \frac{t}{\eta}} + r_\eta(t,x)\right)^{-2}\ \\
\end{align*}
with
$$ r_\eta(t,x) = \sum_{w\in Y\setminus\{y,z\}}
e^{- \frac{1}{\eta}(c(\Phi(t,x),w) + \psi(w) - (c(\Phi(t,x),y)+\psi(y))}. $$
Now, by assumption on the point $x$, for any $w\not\in \{y,z\}$, one has $c(x,y)
+\psi(y) < c(x,w) + \psi(w)$ so that $r_\eta(t,x)$ is negligible
A similar computation can be done for $t\leq 0$, giving us the  estimation
$$ g_\eta(x) \underset{\eta\to 0}{\sim} \int_{-\eps}^\eps \frac{e^{- \frac{\abs{t}}{\eta}}}{\eta(1+e^{- \frac{\abs{t}}{\eta}})^2}
\rho(\Phi(t,x)) J_\Phi(t,x) \dd t \xrightarrow{\eta\to 0}
\rho(x)J_\Phi(0,x).$$ On the other hand, one can show that
for almost every $x$ in $M$ but not in $\Lag_{yz}(\psi)$,
$\abs{\chi_\eta(t,x)}$ tends to zero when $\eta$ goes to zero, thus implying that 
the sequence $(g_\eta(x))$ also converges to $0$.  In other words,
$$ g_\eta(x) \xrightarrow[a.e.]{\eta\to 0}\rho(x)
J_\Phi(0,x) \one_{\Lag_{yz}(\psi)}(x) = \frac{\rho(x)}{\nr{\nabla f(x)}} \one_{\Lag_{yz}(\psi)}(x).$$
By Lebesgue's dominated convergence
theorem, we get
\begin{equation*}
\lim_{\eta\to 0} G_{yz}^\eta(\psi) = \lim_{\eta\to 0} \int_{M} g_\eta(x)  \dd\Haus^{d-1}(x) = \int_{\Lag_{yz}(\psi)} \frac{\rho(x)}{\nr{\nabla f(x)}} \dd\Haus^{d-1}(x). 
\end{equation*}
From the relation $\nr{\nabla f(x)} = \nr{\nabla_xc(x,y)-\nabla_x c(x,z)}$ we get as desired,
\begin{equation*}
\lim_{\eta\to 0} G_{yz}^\eta(\psi) = G_{yz}(\psi) \qedhere
\end{equation*}
\end{proof}
\section{Appendix}

\subsection{Convex analysis}\label{sec:appendixconvexanalysis}

We recall a few relevant definitions and facts from convex analysis (adapted to concave functions).

\begin{definition}\label{def:superdifferential}
The superdifferential of function $F:\Rsp^N\to\Rsp \cup\{-\infty\}$ at
$x\in \Rsp^N$ is the set of vectors $v\in \Rsp^N$ such that
$$ \forall y\in Y,\quad F(y) \leq F(x) + \sca{v}{y-x}.$$
This set is denoted $\partial^+ F(x)$.
\end{definition}

\begin{proposition}\label{prop:superdifferential} The following hold:
\begin{itemize}
\item 
A function $F: \Rsp^N \to \Rsp$ is concave if and only if
$$ \forall x\in \Rsp^N, \partial^+ F(x) \neq \emptyset. $$
\item The superdifferential can be characterized by (\cite[Theorem 25.6]{rockafellar1970convex}):
\begin{equation} \label{eq:SDcalc}
\partial^+ F (x)=\mathrm{conv}\left\{\lim_{n\to\infty} \nabla F(x_n)\mid (x_n)_{n \in \Nsp} \in S\right\},
\end{equation} where $\mathrm{conv}(Z)$ denotes the convex envelope of the set $Z$ and
$$S  = \{ (x_n)_{n\in\Nsp} \mid \forall n\geq 1, \nabla F(x_n) \hbox{ exists } \hbox{ and } \nabla F(x_n) \hbox{ exists} \}.$$
\end{itemize}
\end{proposition}

\subsection{Coarea formula}

We consider two Riemannian sub-manifolds $M$ and $N$, respectively of dimensions  $m$ and $n$, of two Euclidean spaces and assume that $n\leq m$. 
Let $\Phi :M\to N$ be a function of class $C^1$ between the two manifolds. The \emph{Jacobian determinant} of $\Phi: E\subseteq M\to N$ at $x$ is defined by
\begin{equation}\label{eq:J}
J_\Phi(x) = \sqrt{\det(\D\Phi(x) \D\Phi(x)^T)}
\end{equation}  
Note that if $M=\Rsp^m$, $N=\Rsp^n$ and $\Phi = (\Phi_1,\hdots,\Phi_n)$, one has 
$$ \D\Phi(x) \D\Phi(x)^T = \left(\sca{\nabla \Phi_i(x)}{\nabla\Phi_j(x)}\right)_{1\leq i,j\leq n}. $$
In particular, for $n=1$, one has $ J\Phi(x) = \nr{\nabla \Phi_1(x)}$,  and for $n=2$ one gets
\begin{align*}
J\Phi(x)^2 
 &= \nr{\nabla\Phi_1(x)}^2\nr{\nabla \Phi_2(x)}^2
 - \sca{\nabla\Phi_1(x)}{\nabla\Phi_2(x)}^2 \end{align*} which by
 Cauchy-Schwarz's inequality is always non-negative and vanishes iff
 $\nabla\Phi_1(x)$ and $\nabla\Phi_2(x)$ are collinear.\\

\begin{theorem}[Coarea formula]  \label{th:coarea} Let $\Phi:M\to N$ be a function of class $C^1$.  For every $\Haus^m$-measurable function $u:M\to \Rsp$, one has
$$ \int_{M} u(x) J_\Phi(x) \dd\Haus^m(x)
= \int_{N} \int_{\Phi^{-1}(y)}
u(x) \dd\Haus^{m-n}(x) \dd \Haus^n(y). $$
If $J_\Phi(x)$ does not vanish on a measurable subset $E\subset M$, then
\begin{equation}\label{eq:coarea2}
\int_{M} u(x)  \dd\Haus^m(x)
= \int_{N} \int_{\Phi^{-1}(y)}
\frac{u(x)}{J_\Phi(x)} \dd\Haus^{m-n}(x) \dd \Haus^n(y).
\end{equation}
\end{theorem}

In particular, if $m=n$ and $\Phi:M\to N$ is an homeomorphism of class $C^1$, letting $v=u\circ \Phi^{-1}:N\to \Rsp$, one recovers the change of variable formula
\begin{equation}\label{eq:changeofvariable}
 \int_{M} v(\Phi(x)) J_\Phi(x) \dd\Haus^m(x)
= \int_{N} 
v(y) \dd \Haus^n(y).
\end{equation}

\bibliographystyle{amsplain}
\bibliography{ot}

\providecommand{\bysame}{\leavevmode\hbox to3em{\hrulefill}\thinspace}
\providecommand{\MR}{\relax\ifhmode\unskip\space\fi MR }
\providecommand{\MRhref}[2]{%
  \href{http://www.ams.org/mathscinet-getitem?mr=#1}{#2}
}
\providecommand{\href}[2]{#2}
\begin{thebibliography}{10}

\bibitem{agarwal2014approximation}
Pankaj~K Agarwal and R~Sharathkumar, \emph{Approximation algorithms for
  bipartite matching with metric and geometric costs}, Proceedings of the
  forty-sixth annual ACM symposium on Theory of computing, ACM, 2014,
  pp.~555--564.

\bibitem{agueh2011barycenters}
Martial Agueh and Guillaume Carlier, \emph{Barycenters in the wasserstein
  space}, SIAM Journal on Mathematical Analysis \textbf{43} (2011), no.~2,
  904--924.

\bibitem{altschuler2017near}
Jason Altschuler, Jonathan Weed, and Philippe Rigollet, \emph{Near-linear time
  approximation algorithms for optimal transport via sinkhorn iteration},
  Advances in Neural Information Processing Systems, 2017, pp.~1964--1974.

\bibitem{ambrosio2008gradient}
Luigi Ambrosio, Nicola Gigli, and Giuseppe Savar{\'e}, \emph{Gradient flows: in
  metric spaces and in the space of probability measures}, Springer Science \&
  Business Media, 2008.

\bibitem{ambrosio2019optimal}
Luigi Ambrosio, Federico Glaudo, and Dario Trevisan, \emph{On the optimal map
  in the 2-dimensional random matching problem}, arXiv preprint
  arXiv:1903.12153, 2019.

\bibitem{aurenhammer1998minkowski}
Franz Aurenhammer, Friedrich Hoffmann, and Boris Aronov, \emph{Minkowski-type
  theorems and least-squares clustering}, Algorithmica \textbf{20} (1998),
  no.~1, 61--76.

\bibitem{benamou2002monge}
J-D Benamou, Yann Brenier, and Kevin Guittet, \emph{The monge--kantorovitch
  mass transfer and its computational fluid mechanics formulation},
  International Journal for Numerical methods in fluids \textbf{40} (2002),
  no.~1-2, 21--30.

\bibitem{benamou2001mixed}
Jean-David Benamou and Yann Brenier, \emph{Mixed $\mathrm{L}^2$-{W}asserstein
  optimal mapping between prescribed density functions}, Journal of
  Optimization Theory and Applications \textbf{111} (2001), no.~2, 255--271.

\bibitem{benamou2015augmented}
Jean-David Benamou and Guillaume Carlier, \emph{Augmented lagrangian methods
  for transport optimization, mean field games and degenerate elliptic
  equations}, Journal of Optimization Theory and Applications \textbf{167}
  (2015), no.~1, 1--26.

\bibitem{benamou2015iterative}
Jean-David Benamou, Guillaume Carlier, Marco Cuturi, Luca Nenna, and Gabriel
  Peyr{\'e}, \emph{Iterative bregman projections for regularized transportation
  problems}, SIAM Journal on Scientific Computing \textbf{37} (2015), no.~2,
  A1111--A1138.

\bibitem{benamou2016numerical}
Jean-David Benamou, Guillaume Carlier, and Luca Nenna, \emph{A numerical method
  to solve multi-marginal optimal transport problems with coulomb cost},
  Splitting Methods in Communication, Imaging, Science, and Engineering,
  Springer, 2016, pp.~577--601.

\bibitem{benamou2019minimal}
Jean-David Benamou and Vincent Duval, \emph{Minimal convex extensions and
  finite difference discretisation of the quadratic monge--kantorovich
  problem}, European Journal of Applied Mathematics \textbf{30} (2019), no.~6,
  1041--1078.

\bibitem{benamou2014numerical}
Jean-David Benamou, Brittany~D Froese, and Adam~M Oberman, \emph{Numerical
  solution of the optimal transportation problem using the monge--amp{\`e}re
  equation}, Journal of Computational Physics \textbf{260} (2014), 107--126.

\bibitem{berman2017}
Robert~J. Berman, \emph{The {S}inkhorn algorithm, parabolic optimal transport
  and geometric {M}onge-{A}mpère equations}, arXiv preprint arXiv:1712.03082,
  2017.

\bibitem{berman2018convergence}
Robert~J Berman, \emph{Convergence rates for discretized monge-amp{\'e}re
  equations and quantitative stability of optimal transport}, arXiv preprint
  arXiv:1803.00785, 2018.

\bibitem{bertsekas1981new}
D.P. Bertsekas, \emph{A new algorithm for the assignment problem}, Mathematical
  Programming \textbf{21} (1981), no.~1, 152--171.

\bibitem{bertsekas1988dual}
D.P. Bertsekas and J.~Eckstein, \emph{Dual coordinate step methods for linear
  network flow problems}, Mathematical Programming \textbf{42} (1988), no.~1,
  203--243.

\bibitem{birkhoff1946tres}
Garrett Birkhoff, \emph{Tres observaciones sobre el algebra lineal}, Univ. Nac.
  Tucuman, Ser. A \textbf{5} (1946), 147--154.

\bibitem{brenier1991polar}
Yann Brenier, \emph{Polar factorization and monotone rearrangement of
  vector-valued functions}, Communications on pure and applied mathematics
  \textbf{44} (1991), no.~4, 375--417.

\bibitem{brenier1999minimal}
\bysame, \emph{Minimal geodesics on groups of volume-preserving maps and
  generalized solutions of the euler equations}, Communications on Pure and
  Applied Mathematics: A Journal Issued by the Courant Institute of
  Mathematical Sciences \textbf{52} (1999), no.~4, 411--452.

\bibitem{burkard2009assignment}
R.E. Burkard, M.~Dell'Amico, and S.~Martello, \emph{Assignment problems},
  Society for Industrial Mathematics, 2009.

\bibitem{buttazzo2012optimal}
Giuseppe Buttazzo, Luigi De~Pascale, and Paola Gori-Giorgi,
  \emph{Optimal-transport formulation of electronic density-functional theory},
  Physical Review A \textbf{85} (2012), no.~6, 062502.

\bibitem{buttazzo2009optimization}
Giuseppe Buttazzo, Chlo{\'e} Jimenez, and Edouard Oudet, \emph{An optimization
  problem for mass transportation with congested dynamics}, SIAM Journal on
  Control and Optimization \textbf{48} (2009), no.~3, 1961--1976.

\bibitem{caffarelli2008weak}
LA~Caffarelli and VI~Oliker, \emph{Weak solutions of one inverse problem in
  geometric optics}, Journal of Mathematical Sciences \textbf{154} (2008),
  no.~1, 39--49.

\bibitem{caffarelli2010free}
Luis Caffarelli and Robert~J McCann, \emph{Free boundaries in optimal transport
  and monge-ampere obstacle problems}, Annals of mathematics \textbf{171}
  (2010), no.~2, 673--730.

\bibitem{caffarelli1999problem}
Luis~A Caffarelli, Sergey~A Kochengin, and Vladimir~I Oliker, \emph{Problem of
  reflector design with given far-field scattering data}, Monge Amp{\`e}re
  Equation: Applications to Geometry and Optimization: NSF-CBMS Conference on
  the Monge Amp{\`e}re Equation, Applications to Geometry and Optimization,
  July 9-13, 1997, Florida Atlantic University, vol. 226, American Mathematical
  Soc., 1999, p.~13.

\bibitem{carlier2016vector}
Guillaume Carlier, Victor Chernozhukov, Alfred Galichon, et~al., \emph{Vector
  quantile regression: an optimal transport approach}, The Annals of Statistics
  \textbf{44} (2016), no.~3, 1165--1192.

\bibitem{carrillo2019primal}
Jose~A Carrillo, Katy Craig, Li~Wang, and Chaozhen Wei, \emph{Primal dual
  methods for wasserstein gradient flows}, arXiv preprint arXiv:1901.08081,
  2019.

\bibitem{charlier2017efficient}
Benjamin Charlier, Jean Feydy, Joan~Alexis Glaunes, and Alain Trouv{\'e},
  \emph{An efficient kernel product for automatic differentiation libraries,
  with applications to measure transport}, Working version, 2017.

\bibitem{chernozhukov2017monge}
Victor Chernozhukov, Alfred Galichon, Marc Hallin, Marc Henry, et~al.,
  \emph{Monge--kantorovich depth, quantiles, ranks and signs}, The Annals of
  Statistics \textbf{45} (2017), no.~1, 223--256.

\bibitem{chizat2018interpolating}
Lenaic Chizat, Gabriel Peyr{\'e}, Bernhard Schmitzer, and Fran{\c{c}}ois-Xavier
  Vialard, \emph{An interpolating distance between optimal transport and
  fisher--rao metrics}, Foundations of Computational Mathematics \textbf{18}
  (2018), no.~1, 1--44.

\bibitem{cotar2013density}
Codina Cotar, Gero Friesecke, and Claudia Kl{\"u}ppelberg, \emph{Density
  functional theory and optimal transportation with coulomb cost},
  Communications on Pure and Applied Mathematics \textbf{66} (2013), no.~4,
  548--599.

\bibitem{crane2013geodesics}
Keenan Crane, Clarisse Weischedel, and Max Wardetzky, \emph{Geodesics in heat:
  A new approach to computing distance based on heat flow}, ACM Transactions on
  Graphics (TOG) \textbf{32} (2013), no.~5, 152.

\bibitem{cullen1984extended}
Michael~JP Cullen and R~James Purser, \emph{An extended lagrangian theory of
  semi-geostrophic frontogenesis}, Journal of the atmospheric sciences
  \textbf{41} (1984), no.~9, 1477--1497.

\bibitem{cuturi2013sinkhorn}
Marco Cuturi, \emph{Sinkhorn distances: Lightspeed computation of optimal
  transport}, Advances in neural information processing systems, 2013,
  pp.~2292--2300.

\bibitem{cuturi2018semidual}
Marco Cuturi and Gabriel Peyr{\'e}, \emph{Semidual regularized optimal
  transport}, SIAM Review \textbf{60} (2018), no.~4, 941--965.

\bibitem{de2016far}
Pedro Machado~Manh{\~a}es de~Castro, Quentin M{\'e}rigot, and Boris Thibert,
  \emph{Far-field reflector problem and intersection of paraboloids},
  Numerische Mathematik \textbf{134} (2016), no.~2, 389--411.

\bibitem{de2012blue}
Fernando De~Goes, Katherine Breeden, Victor Ostromoukhov, and Mathieu Desbrun,
  \emph{Blue noise through optimal transport}, ACM Transactions on Graphics
  (TOG) \textbf{31} (2012), no.~6, 171.

\bibitem{de2015power}
Fernando de~Goes, Corentin Wallez, Jin Huang, Dmitry Pavlov, and Mathieu
  Desbrun, \emph{Power particles: an incompressible fluid solver based on power
  diagrams.}, ACM Trans. Graph. \textbf{34} (2015), no.~4, 50--1.

\bibitem{degournay2018discrete}
Fr{\'e}d{\'e}ric De~Gournay, Jonas Kahn, and L{\'e}o Lebrat, \emph{3/4-discrete
  optimal transport}, arXiv preprint arXiv:1806.09537, 2018.

\bibitem{degournay2019differentiation}
\bysame, \emph{Differentiation and regularity of semi-discrete optimal
  transport with respect to the parameters of the discrete measure}, Numerische
  Mathematik \textbf{141} (2019), no.~2, 429--453.

\bibitem{deleo2017numerical}
Roberto De~Leo, Cristian~E Guti{\'e}rrez, and Henok Mawi, \emph{On the
  numerical solution of the far field refractor problem}, Nonlinear Analysis
  \textbf{157} (2017), 123--145.

\bibitem{edmonds1972theoretical}
J.~Edmonds and R.M. Karp, \emph{Theoretical improvements in algorithmic
  efficiency for network flow problems}, Journal of the ACM (JACM) \textbf{19}
  (1972), no.~2, 248--264.

\bibitem{erbar2017computation}
Matthias Erbar, Martin Rumpf, Bernhard Schmitzer, and Stefan Simon,
  \emph{Computation of optimal transport on discrete metric measure spaces},
  arXiv preprint arXiv:1707.06859, 2017.

\bibitem{feydy2019fast}
Jean Feydy, Pierre Roussillon, Alain Trouv{\'e}, and Pietro Gori, \emph{Fast
  and scalable optimal transport for brain tractograms}, International
  Conference on Medical Image Computing and Computer-Assisted Intervention,
  Springer, 2019, pp.~636--644.

\bibitem{froese2012numerical}
Brittany~D Froese, \emph{A numerical method for the elliptic
  {M}onge--{A}mp\`ere equation with transport boundary conditions}, SIAM
  Journal on Scientific Computing \textbf{34} (2012), no.~3, A1432--A1459.

\bibitem{gabow1989faster}
H.N. Gabow and R.E. Tarjan, \emph{Faster scaling algorithms for network
  problems}, SIAM Journal on Computing \textbf{18} (1989), 1013.

\bibitem{galerne2018texture}
Bruno Galerne, Arthur Leclaire, and Julien Rabin, \emph{A texture synthesis
  model based on semi-discrete optimal transport in patch space}, SIAM Journal
  on Imaging Sciences \textbf{11} (2018), no.~4, 2456--2493.

\bibitem{galichon2018optimal}
Alfred Galichon, \emph{Optimal transport methods in economics}, Princeton
  University Press, 2018.

\bibitem{galichon2010matching}
Alfred Galichon and Bernard Salani{\'e}, \emph{Matching with trade-offs:
  Revealed preferences over competing characteristics}, Tech. report, CEPR
  Discussion Papers, 2010.

\bibitem{gangbo1996geometry}
Wilfrid Gangbo and Robert~J McCann, \emph{The geometry of optimal
  transportation}, Acta Mathematica \textbf{177} (1996), no.~2, 113--161.

\bibitem{genevay2016stochastic}
Aude Genevay, Marco Cuturi, Gabriel Peyr{\'e}, and Francis Bach,
  \emph{Stochastic optimization for large-scale optimal transport}, Advances in
  neural information processing systems, 2016, pp.~3440--3448.

\bibitem{gigli2011holder}
Nicola Gigli, \emph{On h{\"o}lder continuity-in-time of the optimal transport
  map towards measures along a curve}, Proceedings of the Edinburgh
  Mathematical Society \textbf{54} (2011), no.~2, 401--409.

\bibitem{goldberg1987efficient}
A.V. Goldberg, \emph{Efficient graph algorithms for sequential and parallel
  computers}, Ph.D. thesis, Massachussetts Institute of Technology, 1987.

\bibitem{gu2013variational}
Xianfeng Gu, Feng Luo, Jian Sun, and Shing-Tung Yau, \emph{Variational
  principles for minkowski type problems, discrete optimal transport, and
  discrete monge--amp{\`e}re equations}, Asian Journal of Mathematics
  \textbf{20} (2016), no.~2, 383--398.

\bibitem{guillen2019primer}
Nestor Guillen, \emph{A primer on generated jacobian equations: Geometry,
  optics, economics}, Notices of the American Mathematical Society \textbf{66}
  (2019), no.~9.

\bibitem{gutierrez2001monge}
Cristian~E Guti{\'e}rrez and Haim Brezis, \emph{The monge-ampere equation},
  vol.~44, Springer, 2001.

\bibitem{hartmann2017semi}
Valentin Hartmann and Dominic Schuhmacher, \emph{Semi-discrete optimal
  transport -- the case $p= 1$}, arXiv preprint arXiv:1706.07650, 2017.

\bibitem{henry2015optimal}
Morgane Henry, Emmanuel Maitre, and Val{\'e}rie Perrier, \emph{Optimal
  transport using helmholtz-hodge decomposition and first-order primal-dual
  algorithms}, 2015 IEEE International Conference on Image Processing (ICIP),
  IEEE, 2015, pp.~4748--4752.

\bibitem{hug2016analyse}
Romain Hug, \emph{Analyse math{\'e}matique et convergence d'un algorithme pour
  le transport optimal dynamique: cas des plans de transports non
  r{\'e}guliers, ou soumis {\`a} des contraintes}, Th\`ese de doctorat de
  l'Universit\'e Grenoble-Alpes, 2016.

\bibitem{hutter2019minimax}
Jan-Christian H{\"u}tter and Philippe Rigollet, \emph{Minimax rates of
  estimation for smooth optimal transport maps}, arXiv preprint
  arXiv:1905.05828, 2019.

\bibitem{jordan1998variational}
Richard Jordan, David Kinderlehrer, and Felix Otto, \emph{The variational
  formulation of the fokker--planck equation}, SIAM journal on mathematical
  analysis \textbf{29} (1998), no.~1, 1--17.

\bibitem{kerber2017geometry}
Michael Kerber, Dmitriy Morozov, and Arnur Nigmetov, \emph{Geometry helps to
  compare persistence diagrams}, Journal of Experimental Algorithmics (JEA)
  \textbf{22} (2017), 1--4.

\bibitem{kitagawa2014iterative}
Jun Kitagawa, \emph{An iterative scheme for solving the optimal transportation
  problem}, Calculus of Variations and Partial Differential Equations
  \textbf{51} (2014), no.~1-2, 243--263.

\bibitem{kitagawa2016newton}
Jun Kitagawa, Quentin M{\'e}rigot, and Boris Thibert, \emph{Convergence of a
  newton algorithm for semi-discrete optimal transport}, Journal of the
  European Mathematical Society (2019), OnlineFirst.

\bibitem{kondratyev2016new}
Stanislav Kondratyev, L{\'e}onard Monsaingeon, Dmitry Vorotnikov, et~al.,
  \emph{A new optimal transport distance on the space of finite radon
  measures}, Advances in Differential Equations \textbf{21} (2016), no.~11/12,
  1117--1164.

\bibitem{lavenant2019unconditional}
Hugo Lavenant, \emph{Unconditional convergence for discretizations of dynamical
  optimal transport}, arXiv preprint arXiv:1909.08790, 2019.

\bibitem{levy2015numerical}
Bruno L{\'e}vy, \emph{A numerical algorithm for l2 semi-discrete optimal
  transport in 3d}, ESAIM: Mathematical Modelling and Numerical Analysis
  \textbf{49} (2015), no.~6, 1693--1715.

\bibitem{lombardi2015eulerian}
Damiano Lombardi and Emmanuel Maitre, \emph{Eulerian models and algorithms for
  unbalanced optimal transport}, ESAIM: Mathematical Modelling and Numerical
  Analysis \textbf{49} (2015), no.~6, 1717--1744.

\bibitem{merigot2011multiscale}
Quentin M{\'e}rigot, \emph{A multiscale approach to optimal transport},
  Computer Graphics Forum \textbf{30} (2011), no.~5, 1583--1592.

\bibitem{merigot2019quantitative}
Quentin M{\'e}rigot, Alex Delalande, and Fr{\'e}d{\'e}ric Chazal,
  \emph{Quantitative stability of optimal transport maps and linearization of
  the 2-wasserstein space}, arXiv preprint arXiv:1910.05954 (2019).

\bibitem{merigot2018algorithm}
Quentin M{\'e}rigot, Jocelyn Meyron, and Boris Thibert, \emph{An algorithm for
  optimal transport between a simplex soup and a point cloud}, SIAM Journal on
  Imaging Sciences \textbf{11} (2018), no.~2, 1363--1389.

\bibitem{merigot2016minimal}
Quentin M{\'e}rigot and Jean-Marie Mirebeau, \emph{Minimal geodesics along
  volume-preserving maps, through semidiscrete optimal transport}, SIAM Journal
  on Numerical Analysis \textbf{54} (2016), no.~6, 3465--3492.

\bibitem{meyron2019light}
Jocelyn Meyron, Quentin M{\'e}rigot, and Boris Thibert, \emph{Light in power: a
  general and parameter-free algorithm for caustic design}, ACM Transactions on
  Graphics (TOG) \textbf{37} (2019), no.~6, 224.

\bibitem{mirebeau2015discretization}
Jean-Marie Mirebeau, \emph{Discretization of the 3d {M}onge-{A}mp\`ere
  operator, between wide stencils and power diagrams}, ESAIM: Mathematical
  Modelling and Numerical Analysis \textbf{49} (2015), no.~5, 1511--1523.

\bibitem{monge1781memoire}
Gaspard Monge, \emph{M{\'e}moire sur la th{\'e}orie des d{\'e}blais et des
  remblais}, 1781.

\bibitem{neilan2019monge}
Michael Neilan, Abner~J Salgado, and Wujun Zhang, \emph{The
  monge-amp$\backslash$$\{$e$\}$ re equation}, arXiv preprint arXiv:1901.05108,
  2019.

\bibitem{oberman2015efficient}
Adam~M Oberman and Yuanlong Ruan, \emph{An efficient linear programming method
  for optimal transportation}, arXiv preprint arXiv:1509.03668, 2015.

\bibitem{oliker1989numerical}
VI~Oliker and LD~Prussner, \emph{On the numerical solution of the equation and
  its discretizations, i}, Numerische Mathematik \textbf{54} (1989), no.~3,
  271--293.

\bibitem{oliker2003mathematical}
Vladimir Oliker, \emph{Mathematical aspects of design of beam shaping surfaces
  in geometrical optics}, Trends in Nonlinear Analysis, Springer, 2003,
  pp.~193--224.

\bibitem{papadakis2014optimal}
Nicolas Papadakis, Gabriel Peyr{\'e}, and Edouard Oudet, \emph{Optimal
  transport with proximal splitting}, SIAM Journal on Imaging Sciences
  \textbf{7} (2014), no.~1, 212--238.

\bibitem{pass2015multi}
Brendan Pass, \emph{Multi-marginal optimal transport: theory and applications},
  ESAIM: Mathematical Modelling and Numerical Analysis \textbf{49} (2015),
  no.~6, 1771--1790.

\bibitem{peyre2019computational}
Gabriel Peyr{\'e} and Marco Cuturi, \emph{Computational optimal transport},
  Foundations and Trends{\textregistered} in Machine Learning \textbf{11}
  (2019), no.~5-6, 355--607.

\bibitem{rachev1998mass}
Svetlozar~T Rachev and Ludger R{\"u}schendorf, \emph{Mass transportation
  problems: Volume i: Theory}, vol.~1, Springer Science \& Business Media,
  1998.

\bibitem{rachev2006mass}
\bysame, \emph{Mass transportation problems: Applications}, Springer Science \&
  Business Media, 2006.

\bibitem{rockafellar1970convex}
R~Tyrrell Rockafellar, \emph{Convex analysis}, vol.~28, Princeton university
  press, 1970.

\bibitem{rubner2000earth}
Yossi Rubner, Carlo Tomasi, and Leonidas~J Guibas, \emph{The earth mover's
  distance as a metric for image retrieval}, International journal of computer
  vision \textbf{40} (2000), no.~2, 99--121.

\bibitem{santambrogio2015optimal}
Filippo Santambrogio, \emph{Optimal transport for applied mathematicians},
  Springer, 2015.

\bibitem{schmitzer2016sparse}
Bernhard Schmitzer, \emph{A sparse multiscale algorithm for dense optimal
  transport}, Journal of Mathematical Imaging and Vision \textbf{56} (2016),
  no.~2, 238--259.

\bibitem{schmitzer2019stabilized}
\bysame, \emph{Stabilized sparse scaling algorithms for entropy regularized
  transport problems}, SIAM Journal on Scientific Computing \textbf{41} (2019),
  no.~3, A1443--A1481.

\bibitem{sinkhorn1964relationship}
Richard Sinkhorn, \emph{A relationship between arbitrary positive matrices and
  doubly stochastic matrices}, The annals of mathematical statistics
  \textbf{35} (1964), no.~2, 876--879.

\bibitem{sinkhorn1967concerning}
Richard Sinkhorn and Paul Knopp, \emph{Concerning nonnegative matrices and
  doubly stochastic matrices}, Pacific Journal of Mathematics \textbf{21}
  (1967), no.~2, 343--348.

\bibitem{solomon2015convolutional}
Justin Solomon, Fernando De~Goes, Gabriel Peyr{\'e}, Marco Cuturi, Adrian
  Butscher, Andy Nguyen, Tao Du, and Leonidas Guibas, \emph{Convolutional
  {W}asserstein distances: Efficient optimal transportation on geometric
  domains}, ACM Transactions on Graphics (TOG) \textbf{34} (2015), no.~4, 66.

\bibitem{solomon2014earth}
Justin Solomon, Raif Rustamov, Leonidas Guibas, and Adrian Butscher,
  \emph{Earth mover's distances on discrete surfaces}, ACM Transactions on
  Graphics (TOG) \textbf{33} (2014), no.~4, 67.

\bibitem{trudinger2014local}
Neil~S Trudinger, \emph{On the local theory of prescribed jacobian equations},
  Discrete \& Continuous Dynamical Systems-A \textbf{34} (2014), no.~4,
  1663--1681.

\bibitem{vialard:hal-02303456}
Fran{\c c}ois-Xavier Vialard, \emph{An elementary introduction to entropic
  regularization and proximal methods for numerical optimal transport},
  Lecture, May 2019.

\bibitem{villani2003topics}
C{\'e}dric Villani, \emph{Topics in optimal transportation}, no.~58, American
  Mathematical Soc., 2003.

\bibitem{villani2008optimal}
\bysame, \emph{Optimal transport: old and new}, vol. 338, Springer Science \&
  Business Media, 2008.

\bibitem{wang2004design}
Xu-Jia Wang, \emph{On the design of a reflector antenna ii}, Calculus of
  Variations and Partial Differential Equations \textbf{20} (2004), no.~3,
  329--341.

\end{thebibliography}

\end{document}